\documentclass[12pt]{amsart}
\usepackage[utf8]{inputenc}
\usepackage{amsmath,amsthm,amsfonts,amssymb,amsxtra}
\usepackage{mathrsfs}

\usepackage{ amsmath, amsthm, amsfonts, amssymb, color}
\usepackage{mathrsfs}
\usepackage{amsfonts, amsmath}
\usepackage{amsmath,amstext,amsthm,amssymb,amsxtra}
\usepackage{txfonts} %also pxfonts
\usepackage[colorlinks, citecolor=blue,pagebackref,hypertexnames=false]{hyperref}
\allowdisplaybreaks
\usepackage{pgf,tikz}
\usetikzlibrary{calc}
\usepackage{caption} % for \captionof
\usepackage{enumitem}

\newcommand\Neu{\operatorname{Neu}}

\newcommand\LapoN{\Delta^{(0, \infty)}_{\Neu, d}}

\newcommand\D{\mathcal{D}}
\newcommand\C{\mathbb{C}}

\newcommand\R{\mathbb{R}}

\newcommand\Lapd{\Delta_{d_1, d_2}}

\newcommand\Lapdd{\Delta_{d}}
\newcommand{\supp}{{\rm supp}{\hspace{.05cm}}}

\newtheorem{theorem}{Theorem}[section]
\newtheorem{proposition}[theorem]{Proposition}
\newtheorem{coro}[theorem]{Corollary}
\newtheorem{lemma}[theorem]{Lemma}

\newtheorem{remark}[theorem]{Remark}

\begin{document}

\title{Herz versus Fefferman: Symmetric and asymmetric Bochner--Riesz theory}
\author{Peng Chen, Dangyang He, Adam Sikora, Lixin Yan}  
\address{Peng Chen, Department of Mathematics, Sun Yat-sen (Zhongshan)
University, Guangzhou, 510275, P.R. China}
\email{achenpeng1981@163.com}
\address{Dangyang He, Department of Mathematics, Sun Yat-sen (Zhongshan) University, Guangzhou, 510275, P.R. China}
\email{hedy28@mail.sysu.edu.cn, dangyang.he@hdr.mq.edu.au}
\address{Adam Sikora, School of Mathematical and Physical Sciences, Macquarie University, NSW 2109, Australia}
\email{adam.sikora@mq.edu.au}
\address{Lixin Yan, Department of Mathematics, Sun Yat-sen (Zhongshan) University,
Guangzhou, 510275, P.R. China}
\email{mcsylx@mail.sysu.edu.cn}



\begin{abstract}
We study Bochner--Riesz summability for a one-dimensional model of
noncompact manifolds with ends. Each end has an effective Euclidean
dimension, and these dimensions may differ from one end to
another. The central point is the contrast between the symmetric and
asymmetric cases. When the end dimensions agree, the model follows
Herz's radial theory for the Euclidean Laplacian. When they are unequal,
a Fefferman-type obstruction appears, analogous to the ball
multiplier obstruction in higher-dimensional Fourier analysis, even
though the model itself is one-dimensional. We give a complete
characterisation of the \(L^p\)-boundedness of the corresponding spectral
projections and Bochner--Riesz means. In the asymmetric case, the
boundedness range contains an additional restriction depending on the
difference between the end dimensions. This restriction is absent from
Herz's radial model and shows that Bochner--Riesz summability on spaces
with unequal ends is governed not only by the maximal end dimension, but
also by the interaction between the ends.

\end{abstract}

\maketitle

\hypersetup{linkcolor=black}
\tableofcontents
\hypersetup{linkcolor=red}

\section{Introduction}

%{\color{blue}

%\noindent
This project concerns Bochner--Riesz summability for a one-dimensional model of
manifolds with ends. The guiding idea goes back to Herz~\cite{Herz}, who studied
boundedness properties of spectral projections of the Laplace operator acting on
radial functions. In the radial setting, the analysis reduces to a family of
one-dimensional operators.

%\noindent
Let \(d \ge 1\) and consider the Hilbert space \(L^2(\R_+, r^{d-1}\,dr)\).
For \(f,g\in C_c^\infty[0,\infty)\) define the quadratic form
\begin{equation}\label{Q1}
Q_d^{(0,\infty)}(f,g)
:=\int_0^\infty f'(r)\,\overline{g'(r)}\, r^{d-1}\,dr .
\end{equation}
The (radial) Laplace operator \(\LapoN\) %(denoted also just by $\Delta_d$ below) 
is defined as the self-adjoint operator
associated with \(Q_d^{(0,\infty)}\) via the Friedrichs construction, i.e.\ as the
generator of the closed form obtained by completing \(C_c^\infty[0,\infty)\) in the
form norm.

We shall investigate a ``manifold with ends'' analogue. Let
\[
\widetilde{\R}:=(-\infty,-1]\cup[1,\infty),
\]
and fix parameters $d_1,d_2>1$. Define the weight
\[
\sigma_{d_1,d_2}(r):=
\begin{cases}
|r|^{d_1-1}, & r\leq -1,\\
r^{d_2-1}, & r\geq 1.
\end{cases}
\]
We work in the space $L^2(\widetilde{\R},\sigma_{d_1,d_2}(r)\,dr)$.

Next set 
\[
\mathcal D_0
:=
\left\{
f\in C_c^\infty((-\infty,-1]\cup[1,\infty)):
f(-1)=f(1)
\right\}.
\]
For $f,g\in \mathcal D_0$ we define the quadratic form
\begin{align}\label{Q2}
\widetilde{Q}_{d_1,d_2}(f,g)
&:=\int_{\widetilde{\R}} f'(r)\,\overline{g'(r)}\,\sigma_{d_1,d_2}(r)\,dr \nonumber\\
&=\int_{-\infty}^{-1} f'(r)\,\overline{g'(r)}\,|r|^{d_1-1}\,dr
+\int_{1}^{\infty} f'(r)\,\overline{g'(r)}\,r^{d_2-1}\,dr .
\end{align}
We denote the closure of this form by $\widetilde{Q}_{d_1,d_2}$ as well. 
The operator $\Delta_{d_1,d_2}$ is then defined as the self-adjoint operator associated with the form $\widetilde{Q}_{d_1,d_2}$. 
In particular, in the symmetric case $d_1=d_2=d$, we use the notation $\Lapdd=\Delta_{d,d}$.
Note that if $f\in C_c^{\infty}(a,b)$ for $1<a<b$, then
$\Delta_{d_1,d_2}f(x)=-f''(x)-\frac{d_2-1}{x}f'(x)$, and a similar expression is valid with $d_1$ if $a<b<-1$.
Also note that if $f$ is in the domain of $\widetilde{Q}_{d_1,d_2}$,
then $f(1)=f(-1)$. If $f$ is in the domain of $\Delta_{d_1,d_2}$,
then, in addition, $f'(1)=f'(-1)$.

The operator \(\Delta_{d_1,d_2}\) provides a surprisingly robust one-dimensional
model for the analysis of manifolds with ends. From the perspective of the
present project, a convenient reference for the definition of manifolds with
ends and for the classical analytic framework on such spaces is the work of
Grigor'yan and Saloff-Coste~\cite{GS}; see also~\cite{GIS}.

% The study of the boundedness of the Riesz transform in this one dimensional model setting was initiated
% by Hassell and the third author in~\cite{HS1D} in the symmetric case \(d_1=d_2\),
% and subsequently extended to the asymmetric case \(d_1\neq d_2\) in
% \cite{HE_phd,Nix}. As expected, for the Riesz transform these one-dimensional
% model results coincide with the corresponding results for general manifolds with
% ends; see~\cite{HS,HNS,HE_phd}.

 The study of the Riesz transform in this one-dimensional model setting was initiated by Hassell and the third author in~\cite{HS1D}, where the symmetric case \(d_1=d_2\) was analysed. The asymmetric case \(d_1\neq d_2\) was subsequently treated in~\cite{HE_phd,Nix}. As expected, the resulting boundedness properties mirror those established for Riesz transforms on general manifolds with ends; see~\cite{HS,HNS,HE_phd}.

%\medskip

 We expect an analogous correspondence to hold, at least in part, for
Bochner--Riesz summability. The necessary conditions suggested by the model
appear quite plausible. However, one should be more cautious about expecting a
complete parallel in the positive direction. Proving sufficiency in the general
manifolds-with-ends setting remains open and may be technically demanding,
especially because the sufficiency problem touches on the more delicate, and
still unresolved, aspects of Bochner--Riesz theory. Nevertheless, our
investigations suggest that a general sufficiency theorem in this geometric
setting may still be more accessible than the classical Bochner--Riesz
conjecture.

%\noindent
As a first step, following a standard approach, we recall the functional calculus
associated with spectral decompositions. Let \((X,\mu)\) be a measure space and let
\(L\) be a non-negative self-adjoint operator on \(L^2(X)\). Then \(L\) admits a
spectral resolution \(\{E_L(\lambda)\}_{\lambda\ge 0}\). For any bounded Borel
function \(F:[0,\infty)\to\C\), one defines
\begin{equation}\label{eq1.1}
F(L)=\int_{0}^{\infty} F(\lambda)\, dE_L(\lambda),
\end{equation}
where the integral is understood in the sense of the spectral theorem. In
particular, \(F(L)\) is a well-defined bounded operator on \(L^2(X)\).

Spectral multiplier theorems provide sufficient conditions on \(F\) ensuring that
\(F(L)\) extends to a bounded operator on \(L^p(X)\) for some range of \(p\).
In this project, however, we restrict attention to the model Laplacians
\(\LapoN\) and \(\Delta_{d_1,d_2}\), and to the associated spectral projections
for their square roots, namely
\[
E_{\sqrt{\LapoN }}\big([0,R)\big)
\qquad\text{and}\qquad
E_{\sqrt{\Delta_{d_1,d_2}}}\big([0,R)\big),\qquad R>0.
\]

In what follows, we restrict attention to Bochner--Riesz spectral multipliers and
to spectral projections, which may be viewed as Bochner--Riesz means of order
$0$. For $R>0$ and $\delta\geq 0$, define the Bochner--Riesz cutoff
\[
\sigma_R^\delta(\lambda):=
\begin{cases}
\bigl(1-\lambda/R^2\bigr)^\delta, & 0\leq \lambda\leq R^2,\\[2mm]
0, & \lambda>R^2,
\end{cases}
\]
where, for $\delta=0$, we interpret $\sigma_R^0$ as the characteristic
function of $[0,R^2]$, and set
\begin{equation}\label{delBR}
\sigma_R^\delta(L):=
\int_{0}^{\infty}\sigma_R^\delta(\lambda)\,dE_L(\lambda),
\end{equation}
where the operator is defined via the spectral calculus \eqref{eq1.1}.
In the sequel, we often write the corresponding operator as
\[
\sigma_R^\delta(L)
=
\Bigl(1-\tfrac{L}{R^2}\Bigr)_+^\delta
=
\Bigl(1-\tfrac{(\sqrt{L})^2}{R^2}\Bigr)_+^\delta.
\]

The Bochner--Riesz means originate in Bochner's study
\cite{Bochner1936} of spherical summation of multiple Fourier series.
In the Euclidean setting, the Fourier transform diagonalises
\(-\Delta_{\mathbb R^d}\), and the corresponding means are given by the
radial multipliers
$\left(1-{|\xi|^2}/{R^2}\right)_+^\delta$.

A central problem in Bochner--Riesz theory is to determine, for each
\(1\leq p\leq\infty\), the critical index \(\delta_{\mathrm{cr}}(p)\) for which
\[
\sup_{R>0}\|\sigma_R^\delta(L)\|_{L^p\to L^p}<\infty
\qquad\text{whenever}\qquad
\delta>\delta_{\mathrm{cr}}(p).
\]
For \(1\leq p<\infty\), under the standard density assumptions, this uniform
boundedness is equivalent to the convergence
\(\sigma_R^\delta(L)f\to f\) in \(L^p\) as \(R\to\infty\).

The literature on Bochner--Riesz summability is vast. Useful starting points
include \cite{MiCh,chSo,CaS,TT,wang2025, LeeS}, together with the classical accounts
in \cite[Chapter
IX]{Stein},\cite[Chapter 5.2]{Sogge} and the references therein.

Our first result concerns Bochner--Riesz means of order \(0\), that is,
spectral projections. In the symmetric case \(d_1=d_2=d\), the
\(L^p\)-boundedness range for the projections associated with
\(\Delta_{d_1,d_2}=\Delta_d\) coincides exactly with Herz's range for the
radial Laplacian \(\LapoN\) in~\cite{Herz}. The same range, namely
\(\left|\frac{1}{p}-\frac{1}{2}\right|< \frac{1}{2d}\), appears in
C\'ordoba's closely related positive results for the disc multiplier
in~\cite{cord}. By contrast, when \(d_1\neq d_2\), the behaviour is
governed by a Fefferman-type obstruction~\cite{Fefferman}, which becomes
dominant.

%To avoid lengthy and repetitive calculations, we assume that $d_1,d_2>2$.

To keep the scope of the paper manageable, we assume that \(d_i>2\) for \(i=1,2\). See Question~Q1 in the Open Problems Section~\ref{oppr} below, for further discussion.

\begin{theorem}\label{main_spectral_projection}
Let $d_1,d_2>2$. Then the spectral projection $E_{\sqrt{\Lapd}}([0,R))$ ($R>0$) is bounded uniformly in $L^p(\widetilde{\R})$, i.e.,
\begin{align*}
    \sup_{R>0} \left\|  \mathbf{1}_{[0,R)}\left(\sqrt{\Lapd}\right)   \right\|_{p\to p} < \infty
\end{align*}
if and only if 
\begin{align*}
    \begin{cases}
        \left|\frac{1}{2}-\frac{1}{p}\right|<\frac{1}{2d}, & d_1=d_2=d,\\
        p=2, & d_1\ne d_2.
    \end{cases}
\end{align*}
\end{theorem}

%\noindent
Our main result establishes the $L^p$-continuity (i.e.\ uniform boundedness) of Bochner--Riesz means of arbitrary order. In the symmetric case $d_1=d_2$, the result coincides with the range predicted by Herz's theorem. In contrast, when $d_1\neq d_2$, we obtain a strengthened Fefferman-type obstruction, which rules out even very weak forms of the Bochner--Riesz conjecture in this setting.
Throughout the discussion of the asymmetric case $d_1\neq d_2$, we write
\[
D=d_1\vee d_2,\qquad d=d_1\wedge d_2.
\]

We continue to assume that \(d_i>2\) for \(i=1,2\).

\begin{theorem}\label{main_BR}
Let $d_1,d_2>2$. The family of Bochner--Riesz means $\sigma_{R}^\delta(\Lapd)$ ($R>0$) is uniformly bounded on $L^p(\widetilde{\R})$, i.e.,
\begin{align*}
    \sup_{R>0} \left\| \left( 1 - \frac{\Lapd}{R^2} \right)_+^{\delta}  \right\|_{p\to p} < \infty,
\end{align*}
if and only if
\begin{align*}
    \delta > \max \left( D\left|\frac{1}{p}-\frac{1}{2}\right|-\frac{1}{2},\, 0 \right),
\end{align*}
and
\begin{align}\label{asym}
    \delta \ge |d_1-d_2| \left|\frac{1}{p}-\frac{1}{2}\right|.
\end{align}
%where $D = d_1 \vee d_2$.
\end{theorem}

%\noindent
The theorem provides a complete characterisation of the $L^p$-boundedness of general Bochner--Riesz means in this setting. The admissible range is illustrated in Figure~\ref{f1}, where the boundedness region is shown in dark shading.

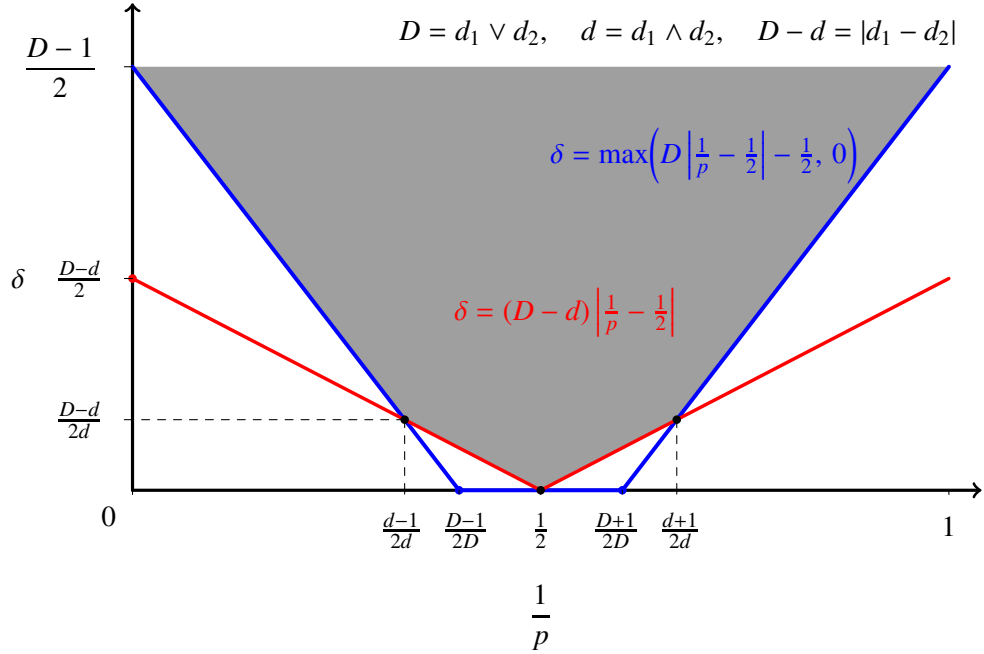
\begin{figure}[htbp]

\centering

% Parameters
\def\D{5}   % D = d_1 \vee d_2
\def\m{2}   % |d_1-d_2| = D-d

\pgfmathsetmacro{\dmin}{\D-\m}                  % d = d_1 \wedge d_2
\pgfmathsetmacro{\ytop}{(\D-1)/2}
\pgfmathsetmacro{\yredaxis}{\m/2}               % = (D-d)/2
\pgfmathsetmacro{\xblueL}{0.5-1/(2*\D)}
\pgfmathsetmacro{\xblueR}{0.5+1/(2*\D)}
\pgfmathsetmacro{\xintL}{0.5-1/(2*\dmin)}
\pgfmathsetmacro{\xintR}{0.5+1/(2*\dmin)}
\pgfmathsetmacro{\yint}{\m/(2*\dmin)}           % = (D-d)/(2d)

\begin{tikzpicture}[
    x=10.8cm,
    y=2.8cm,
    line cap=round,
    line join=round,
    font=\small
]

% axes
\draw[->, very thick] (0,0) -- (1.04,0);
\draw[->, very thick] (0,0) -- (0,\ytop+0.30);

% shaded boundedness region
\fill[gray!75]
  (0,\ytop) -- (1,\ytop) -- (\xintR,\yint) -- (0.5,0) -- (\xintL,\yint) -- cycle;

% blue normal BR line
\draw[blue, line width=1.5pt]
  (0,\ytop) -- (\xblueL,0) -- (\xblueR,0) -- (1,\ytop);

% red new restriction line
\draw[red, line width=1.3pt]
  (0,\yredaxis) -- (0.5,0) -- (1,\yredaxis);

% key points
\fill[blue] (\xblueL,0) circle (1.6pt);
\fill[blue] (\xblueR,0) circle (1.6pt);
\fill[red]  (0,\yredaxis) circle (1.6pt);
\fill       (0.5,0) circle (1.6pt);
\fill       (\xintL,\yint) circle (1.6pt);
\fill       (\xintR,\yint) circle (1.6pt);

% dashed guides from intersection points to axes
\draw[dashed] (\xintL,\yint) -- (\xintL,0);
\draw[dashed] (\xintL,\yint) -- (0,\yint);
\draw[dashed] (\xintR,\yint) -- (\xintR,0);

% x-axis ticks
\draw (0,0) -- +(0,-0.020);
\draw (\xblueL,0) -- +(0,-0.018);
\draw (0.5,0) -- +(0,-0.020);
\draw (\xblueR,0) -- +(0,-0.018);
\draw (\xintL,0) -- +(0,-0.018);
\draw (\xintR,0) -- +(0,-0.018);
\draw (1,0) -- +(0,-0.020) node[below=4pt] {$1$};

% y-axis ticks
\draw (0,\ytop) -- +(-0.010,0) node[left=6pt] {$\dfrac{D-1}{2}$};
\draw (0,\yredaxis) -- +(-0.010,0);
\draw (0,\yint) -- +(-0.010,0);

% only one 0 at the origin
\node[below left=2pt] at (0,0) {$0$};

% x-axis labels (staggered to avoid overlap)
\node[below=4pt] at (0.5,0) {$\frac{1}{2}$};

\node[below=4pt, anchor=north east] at (0.45,0)
  {$\frac{D-1}{2D}$};
\node[below=4pt, anchor=north west] at (0.29,0)
  {$\frac{d-1}{2d}$};

\node[below=4pt, anchor=north west] at (0.55,0)
  {$\frac{D+1}{2D}$};
\node[below=4pt, anchor=north east] at (0.71,0)
  {$\frac{d+1}{2d}$};

% y-axis labels
\node[left=7pt] at (0,\yredaxis) {$\frac{D-d}{2}$};
\node[left=7pt] at (0,\yint) {$\frac{D-d}{2d}$};

% axis labels
\node at (0.50,-0.60) {$\dfrac{1}{p}$};
\node at (-0.14,\ytop/2) {$\delta$};

% line labels
\node[blue] at (0.70,\ytop*0.79)
  {$\delta=\max\!\left(D\left|\frac{1}{p}-\frac{1}{2}\right|-\frac{1}{2},\,0\right)$};

\node[red] at (0.53,\ytop*0.42)
  {$\delta=(D-d)\left|\frac{1}{p}-\frac{1}{2}\right|$};

% definition box in upper-right corner
\node[
    anchor=north east,
    fill=white,
    rounded corners=2pt,
    inner sep=3pt
] at (1.02,\ytop+0.26)
{$D=d_1 \vee d_2,\quad d=d_1\wedge d_2,\quad D-d=|d_1-d_2|$};

\end{tikzpicture}

%The Bochner--Riesz boundedness range in the $\left(\frac{1}{p},\delta \right)$-plane.
%\caption
\caption{The shaded region is the sharp \(L^p\)-boundedness region. The blue curve
is the Herz/classical Bochner--Riesz constraint governed by the larger
dimension \(D\), while the red line is the additional Fefferman-type
obstruction caused by the difference between the two end dimensions.}
\label{f1}
\end{figure}

%\noindent
The blue curve corresponds to the classical Bochner--Riesz condition
\[
\delta=\max\!\left(D\left|\frac{1}{p}-\frac{1}{2}\right|-\frac{1}{2},\,0\right),
\]
while the red line represents the additional restriction
\[
\delta = (D-d)\left|\frac{1}{p}-\frac{1}{2}\right|,
%\qquad d = d_1 \wedge d_2,
\]
which is absent in the symmetric case. The intersection of these constraints determines the sharp boundedness region.

A similar two-scale phenomenon for Bochner--Riesz summability of Laguerre
expansions is studied in \cite{SharmaSikoraThangavelu}. However,
Fefferman's phenomenon does not appear to play a role in that setting.
It is an interesting question whether such a phenomenon arises in the
Bochner--Riesz problem for the harmonic oscillator considered in
\cite{LeeS}.

We next place the boundedness region in Theorem~\ref{main_BR} in the context of the classical Bochner--Riesz theory, focusing on its relation to Fefferman's ball multiplier obstruction and the associated Kakeya/Nikodym geometry.

\subsubsection*{Fefferman's ball multiplier obstruction and Kakeya/Nikodym geometry.}
Let
\[
  T_0 f := \mathcal{F}^{-1}\!\bigl(\chi_{|\xi|\le 1}\,\widehat{f}(\xi)\bigr)
\]
be the \emph{ball multiplier} in $\mathbb{R}^n$, so that $T_0 = \sigma^0_1(\Delta)$.
Fefferman proved in \cite{Fefferman} that, for every $n\geq 2$, $T_0$ is not bounded
on $L^p(\mathbb{R}^n)$ unless $p=2$. The obstruction is geometric: his construction
uses wave packets concentrated on long thin tubes pointing in many directions, whose
Fourier supports lie on thin plates tangent to $S^{n-1}$. Such configurations are
governed by Kakeya/Nikodym geometry, originating in Besicovitch's construction of
measure-zero sets containing a unit line segment in every direction. This geometry
also underlies obstructions to restriction and Bochner--Riesz estimates: strong
Bochner--Riesz estimates force Kakeya/Nikodym sets to be large, while thin
configurations obstruct boundedness. Bourgain obtained lower bounds for the
Minkowski dimension of Kakeya sets \cite{Bourgain}\footnote{The discussion towards
the end of this Introduction is inspired by Remarks~(5.2)--(5.4) in
\cite{Bourgain}.}, and related connections with maximal and oscillatory integral
estimates were developed by Stein \cite{Stein2}. On curved manifolds, Minicozzi and
Sogge \cite{MinicozziSogge} showed that Nikodym-type sets may have Minkowski dimension
as low as $(n+1)/2$.

Our setting is analogous in spirit to Nikodym-type geometry, but with an important
asymmetry: ``large'' and ``small'' scales live on different ends. Let $\Pi\subset M$
and define
\[
  \Pi^*
  :=
  \Bigl\{
  x\in M:\ \exists\ \text{a geodesic }\gamma_x\text{ of length }r>1
  \text{ with }x\in\gamma_x
  \text{ and }|\Pi\cap\gamma_x|>cr
  \Bigr\}.
\]
This picture is reflected in the counterexamples of Sections~\ref{SP} and~\ref{BR};
see \eqref{couterF}, \eqref{counter1} and Remarks~\ref{r53}--\ref{r54}. Let
$\varepsilon\geq1$ be large. In the radial two-end model, take
$\Pi^*_\varepsilon=[2\varepsilon,3\varepsilon]$ on the end of larger dimension $D$
and $\Pi_\varepsilon=[-\varepsilon,-2\varepsilon]$ on the end of smaller dimension
$d$. Their measures are comparable to $\varepsilon^D$ and $\varepsilon^d$,
respectively. Thus a set on the smaller end generates, through geodesics crossing
the junction, a family of points of full $D$-dimensional size on the larger end.
This is precisely the geometry exploited in Sections~\ref{SP} and~\ref{BR}: test
functions are supported on an interval of length comparable to $\varepsilon$ on
one end, while the leading part of the output is measured on a separated interval
of comparable size on the other. The spectral projection counterexample shows that,
when $d<D$, the off-diagonal part is uniformly $L^p$-bounded only for $p=2$, while
the Bochner--Riesz counterexample yields the additional obstruction
\[
  \delta \geq (D-d)\left|\frac1p-\frac12\right|.
\]
Thus these counterexamples provide a radial two-end analogue of the
Kakeya/Nikodym mechanism behind Fefferman's obstruction.

As noted above in connection with the Riesz transform, one expects that the model considered here should reflect, at least at the level of necessary conditions, the behaviour of genuine manifolds with ends; see, for example, \cite{GS,HS,HNS}. We do not pursue this level of generality in the present paper. Nevertheless, the obstruction expressed by condition \eqref{asym} appears to be geometric rather than purely one-dimensional: it comes from the interaction between ends with different volume growth. It is therefore natural to expect analogous negative results for Bochner--Riesz means on manifolds with unequal ends. The corresponding sufficiency problem is more delicate. However, if \eqref{asym} is the correct additional obstruction, then the remaining positive theory may be more accessible than the full classical Bochner--Riesz conjecture. This leads naturally to the open problem formulated in Section~\ref{oppr}.

\subsection{Organisation of the paper.}
The proofs are organised around the decomposition of the spectral measure
developed in Section~\ref{sec2}. That section constructs the resolvent and decomposes the spectral
measure into \(kl\) and \(kk\) components, see below, with the latter further split into
on-diagonal and off-diagonal terms.

Sections~\ref{sec3} and~\ref{SP} prove Theorem~\ref{main_spectral_projection} on spectral projections, equivalently
Bochner--Riesz means of order zero. Theorem~\ref{herz_kl} treats the \(kl\) component
using Herz's radial argument, while Theorem~\ref{herz_doubling} proves sufficiency in the
symmetric case \(d_1=d_2\). Section~\ref{SP} establishes sharpness: Theorem~\ref{fefferman_nondoubling}
shows that, when \(d_1\neq d_2\), uniform boundedness holds only on \(L^2\),
and Lemma~\ref{lem:negative:spectral} proves necessity of the Herz range in the symmetric case.

Sections~\ref{sec4} and~\ref{BR} prove Theorem~\ref{main_BR} for Bochner--Riesz means of positive
order. Section~\ref{sec4} establishes sufficiency in Theorem~\ref{thm:BRR}. The \(kl\)
component is reduced to the classical radial theory, while the \(kk\)
component is treated in the low- and high-energy regimes. Propositions~\ref{prop:smallR}
and~\ref{prop:largeR} control the off-diagonal terms, and Corollary~\ref{cor1} controls the
on-diagonal terms. Section~\ref{BR} proves necessity: Proposition~\ref{propBRR}, via
Lemmas~\ref{Counter1} and~\ref{Counter2}, yields respectively the additional obstruction caused
by the difference between the end dimensions and the classical obstruction
governed by \(D=d_1\vee d_2\). Together with Theorem~\ref{thm:BRR}, this completes the
proof of Theorem~\ref{main_BR}.

\section{Preliminaries}\label{sec2}

\subsection{Resolvent construction}

We begin by recalling the resolvent construction for the operators \(\Lapd\).
This will later be used to compute eigenfunctions and spectral projections
via the limiting absorption principle. Our approach follows
\cite{HS1D,HE_phd,Nix}, and we use their notation wherever possible.

Following \cite{HS1D}, consider
\begin{equation}\label{kl}
f''(r)+\frac{d-1}{r}f'(r)=f(r).
\end{equation}
Set
\[
F(r)=r^{\frac d2-1}f(r),
\qquad\text{equivalently}\qquad
f(r)=r^{1-\frac d2}F(r).
\]
Substitution in \eqref{kl} gives
\begin{equation}\label{bessel_mod}
r^2F''(r)+rF'(r)-\bigl(r^2+(d/2-1)^2\bigr)F(r)=0.
\end{equation}
Thus \(F\) satisfies the modified Bessel equation of order \(d/2-1\). Hence
\(F\) is a linear combination of \(I_{d/2-1}(r)\) and \(K_{d/2-1}(r)\);
see \cite[\S~9.6.1]{AS}.

More generally, for \(\lambda>0\), every solution of
\begin{equation}\label{wuj}
f''(r)+\frac{d-1}{r}f'(r)=\lambda^2 f(r)
\end{equation}
is a linear combination of
\[
r\mapsto l_d(\lambda r),
\qquad
r\mapsto k_d(\lambda r),
\]
where
\begin{equation}\label{kd_ld_def}
k_d(r)=r^{1-\frac d2}K_{|d/2-1|}(r),
\qquad
l_d(r)=r^{1-\frac d2}I_{d/2-1}(r).
\end{equation}
These functions form a fundamental system of solutions, adapted respectively
to the behaviour at infinity and at the origin.

%%%%%%%%%%%%%%%%%%%%%%%%%%%%
\subsection{Asymptotics for the special functions}

Before describing the resolvent kernels for the operators considered here, we
recall the basic asymptotic behaviour of \(l_d\) and \(k_d\). We use standard
asymptotic formulae for modified Bessel functions with purely imaginary
arguments; see \cite{AS}.

For \(d>2\) and \(\lambda\in\mathbb R\), we have
\[
\begin{array}{ll}
\displaystyle
k_d(i\lambda)\sim
\begin{cases}
|\lambda|^{2-d}, & |\lambda|\leq 1,\\
|\lambda|^{\frac{1-d}{2}}e^{-i\lambda}, & |\lambda|\geq 1,
\end{cases}
&
\displaystyle
l_d(i\lambda)\sim
\begin{cases}
1, & |\lambda|\leq 1,\\
|\lambda|^{\frac{1-d}{2}}e^{i\lambda}
+
|\lambda|^{\frac{1-d}{2}}e^{-i\lambda}, & |\lambda|\geq 1,
\end{cases}
\\[6mm]
\displaystyle
k_d'(i\lambda)\sim
\begin{cases}
|\lambda|^{1-d}, & |\lambda|\leq 1,\\
|\lambda|^{\frac{1-d}{2}}e^{-i\lambda}, & |\lambda|\geq 1,
\end{cases}
&
\displaystyle
l_d'(i\lambda)\sim
\begin{cases}
|\lambda|, & |\lambda|\leq 1,\\
|\lambda|^{\frac{1-d}{2}}e^{i\lambda}
+
|\lambda|^{\frac{1-d}{2}}e^{-i\lambda}, & |\lambda|\geq 1.
\end{cases}
\end{array}
\]

\begin{remark}\label{remark1}
Here and throughout, \(\sim\) denotes asymptotic equivalence up to non-zero
multiplicative constants and lower-order terms. For example, for
\(|\lambda|\geq 1\),
\[
k_d(i\lambda)
=
C_d\lambda^{\frac{1-d}{2}}e^{-i\lambda}
+
e^{-i\lambda}O\!\left(\lambda^{-\frac{d+1}{2}}\right),
\]
while
\[
\begin{aligned}
l_d(i\lambda)
&=
C_{d,+}\lambda^{\frac{1-d}{2}}e^{i\lambda}
+
C_{d,-}\lambda^{\frac{1-d}{2}}e^{-i\lambda}  \\
&\quad
+
O\!\left(\lambda^{-\frac{d+1}{2}}\right)e^{i\lambda}
+
O\!\left(\lambda^{-\frac{d+1}{2}}\right)e^{-i\lambda}.
\end{aligned}
\]
Thus \(l_d\) has two oscillatory leading terms. This does not create any
additional difficulty in the arguments below, since the contribution involving
\(l_d\) can be treated by the classical method of Herz, together with the
calculations in \cite{COSY}.

More explicitly, for \(|\lambda|>1\), using
\(I_\nu(i\lambda)=i^\nu J_\nu(\lambda)\) and the standard large-\(\lambda\)
expansion of \(J_\nu\), we obtain, up to the first three terms,
\[
\begin{aligned}
I_\nu(i\lambda)
\sim\;&
\sqrt{\frac{1}{2\pi\lambda}}
\Bigg[
e^{i\lambda-i\pi/4}
\left(
1+\frac{i(4\nu^2-1)}{8\lambda}
-\frac{(4\nu^2-1)(4\nu^2-9)}{128\lambda^2}
\right)
\\
&\qquad\qquad
+
e^{-i\lambda+i\nu\pi+i\pi/4}
\left(
1-\frac{i(4\nu^2-1)}{8\lambda}
-\frac{(4\nu^2-1)(4\nu^2-9)}{128\lambda^2}
\right)
\Bigg],
\end{aligned}
\]
and
\[
K_\nu(i\lambda)
\sim
\sqrt{\frac{\pi}{2\lambda}}\,e^{-i\lambda-i\pi/4}
\left[
1-\frac{i(4\nu^2-1)}{8\lambda}
-\frac{(4\nu^2-1)(4\nu^2-9)}{128\lambda^2}
\right].
\]
\end{remark}

\begin{remark} The cases \(d\leq 2\) are also interesting, but the analysis is not a routine extension of the arguments above. The asymptotics of \(k_d\) are different in this range, so the calculations would have to be substantially modified. We therefore leave the treatment of \(d\leq 2\) to future work; see Section~\ref{oppr}. \end{remark}

%%%%%%%%%%%%%%%%%%%%%%%%%%%%

\subsection{Operators \texorpdfstring{$\LapoN$}{LapoN} acting on
\texorpdfstring{$L^2((0,\infty), r^{d-1}dr)$}{L2}}

For $d\ge 1$, consider the space $L^2(\mathbb R_+, r^{d-1}dr)$ and the operator
$\LapoN$ associated with the quadratic form $Q_d^{(0,\infty)}$ defined in
\eqref{Q1}. Standard calculations give the following expression for the resolvent
kernel of $\LapoN$; see, for example, \cite{HS1D}:
\begin{equation}\label{laponker}
K_{(\LapoN + \lambda^2)^{-1}}(x,y)
=
\begin{cases}
\nu_d\, \lambda^{d-2}\, k_d(\lambda y)\, l_d(\lambda x), & y \ge x,\\[2mm]
\nu_d\, \lambda^{d-2}\, l_d(\lambda y)\, k_d(\lambda x), & x > y,
\end{cases}
\end{equation}
where $\nu_d$ depends only on $d$. This representation is valid for all $d\ge 1$.

\subsection{Operators \texorpdfstring{$\Lapd$}{Lapd} acting on
\texorpdfstring{$L^2(\widetilde{\mathbb R}, \sigma_{d_1,d_2}(r)dr)$}{L22}}
\label{seclap}

Let $\widetilde{\mathbb R}=(-\infty,-1]\cup[1,\infty)$. We consider the operator
$\Lapd$ associated with the quadratic form $\widetilde{Q}_{d_1,d_2}$ defined in
\eqref{Q2}.

For notational convenience, on the component indexed by $i=1,2$ we write
\[
k_i = k_{d_i}, \qquad l_i = l_{d_i}.
\]
The resolvent kernel of $\Lapd$ has the following representation.

If $y \ge 1$, then
\begin{equation}\label{eq_y>1}
K_{(\Lapd + \lambda^2)^{-1}}(x,y)
=
\begin{cases}
A(\lambda)\, k_1(\lambda|x|)\, k_2(\lambda|y|), & x \le -1,\\[2mm]
B_2(\lambda)\, k_2(\lambda|y|)\, k_2(\lambda|x|)
+ v_2 \lambda^{d_2-2}\, k_2(\lambda|y|)\, l_2(\lambda|x|), & 1 \le x \le y,\\[2mm]
B_2(\lambda)\, k_2(\lambda|y|)\, k_2(\lambda|x|)
+ v_2 \lambda^{d_2-2}\, l_2(\lambda|y|)\, k_2(\lambda|x|), & x \ge y.
\end{cases}
\end{equation}

If $y \le -1$, then
\begin{equation}\label{eq_y<1}
K_{(\Lapd + \lambda^2)^{-1}}(x,y)
=
\begin{cases}
A(\lambda)\, k_2(\lambda|x|)\, k_1(\lambda|y|), & x \ge 1,\\[2mm]
B_1(\lambda)\, k_1(\lambda|y|)\, k_1(\lambda|x|)
+ v_1 \lambda^{d_1-2}\, k_1(\lambda|y|)\, l_1(\lambda|x|), & y \le x \le -1,\\[2mm]
B_1(\lambda)\, k_1(\lambda|y|)\, k_1(\lambda|x|)
+ v_1 \lambda^{d_1-2}\, l_1(\lambda|y|)\, k_1(\lambda|x|), & x \le y.
\end{cases}
\end{equation}
Here $v_i$ depends only on $d_i$, and
\begin{align*}
A(\lambda) &= \frac{-1}{\lambda \,[k_1(\lambda) k_2(\lambda)]'},\\
B_2(\lambda) &= \frac{-v_2 \lambda^{d_2-2}\,[k_1(\lambda) l_2(\lambda)]'}
{[k_1(\lambda) k_2(\lambda)]'},\\
B_1(\lambda) &= \frac{-v_1 \lambda^{d_1-2}\,[k_2(\lambda) l_1(\lambda)]'}
{[k_1(\lambda) k_2(\lambda)]'}.
\end{align*}
In the symmetric case $d_1=d_2=d$, we write $\Lapdd$. The notation is chosen
to remain as close as possible to that of \cite{HS1D}, although some changes are
unavoidable when $d_1\neq d_2$.

The asymptotics of the special functions give
\begin{align*}
A(i\lambda) &\sim
\begin{cases}
|\lambda|^{d_1+d_2-4}, & |\lambda| \le 1,\\
|\lambda|^{\frac{d_1+d_2-4}{2}} e^{2i\lambda}, & |\lambda| \ge 1,
\end{cases}\\
B_2(i\lambda) &\sim
\begin{cases}
|\lambda|^{2d_2-4}, & |\lambda| \le 1,\\
|\lambda|^{d_2-2}, & |\lambda| \ge 1,
\end{cases}\\
B_1(i\lambda) &\sim
\begin{cases}
|\lambda|^{2d_1-4}, & |\lambda| \le 1,\\
|\lambda|^{d_1-2}, & |\lambda| \ge 1.
\end{cases}
\end{align*}

\begin{remark}
The notation ``$\sim$'' is used to indicate only the leading behaviour of the
special functions and of the coefficients $A$, $B_1$, and $B_2$. These
asymptotic relations are not literal equalities. In all cases used below, the
omitted terms have the same oscillatory structure as the leading term, but with
additional decay in the relevant parameter. Hence, once the leading terms are
estimated, the lower-order contributions are treated by the same arguments and
give better bounds. For completeness, subsection~\ref{appendix} 
records the more
precise expansions with remainders that are used implicitly in the proofs.
\end{remark}

\begin{remark}
The leading terms of $B_1$ and $B_2$ for large $|\lambda|$ contain no
oscillation. This is due to the additional phase coming from $l$. Indeed, for
$|\lambda|\ge 1$,
\begin{equation*}
k_j(i\lambda)
=
c_j\lambda^{\frac{1-d_j}{2}}e^{-i\lambda}
+O\left(\lambda^{-\frac{d_j+1}{2}}\right)e^{-i\lambda},
\end{equation*}
whereas
\begin{equation*}
l_j(i\lambda)
=
a_j\lambda^{\frac{1-d_j}{2}}e^{i\lambda}
+b_j\lambda^{\frac{1-d_j}{2}}e^{-i\lambda}
+O\left(\lambda^{-\frac{d_j+1}{2}}\right)e^{i\lambda}
+O\left(\lambda^{-\frac{d_j+1}{2}}\right)e^{-i\lambda}.
\end{equation*}
Thus $k_1(i\lambda)l_2(i\lambda)$ contains one non-oscillatory term and one
term with phase $e^{-2i\lambda}$. After differentiation, the non-oscillatory
term loses one power of $\lambda$, while the $e^{-2i\lambda}$ term remains of
leading order. Therefore
\begin{equation*}
[k_1(i\lambda)l_2(i\lambda)]'
=
c\lambda^{-\frac{d_1+d_2-2}{2}}e^{-2i\lambda}
+O\left(\lambda^{-\frac{d_1+d_2}{2}}\right),
\end{equation*}
and similarly
\begin{equation*}
[k_1(i\lambda)k_2(i\lambda)]'
=
c'\lambda^{-\frac{d_1+d_2-2}{2}}e^{-2i\lambda}
+O\left(\lambda^{-\frac{d_1+d_2}{2}}\right).
\end{equation*}
The oscillatory phases therefore cancel in the quotient, giving
\begin{equation*}
B_2(i\lambda)\sim \lambda^{d_2-2},
\qquad
B_1(i\lambda)\sim \lambda^{d_1-2},
\qquad |\lambda|\ge 1.
\end{equation*}
\end{remark}

\subsection{Decomposition of the spectral measure}\label{ssec2.5}

We first record the following observation.

%\begin{remark}\label{remark_important}
In all cases above, the resolvent kernel naturally decomposes into two parts.
We denote by $(\Lapd+\lambda^2)^{-1}_{kk}$ the part consisting of terms of the
form $k(\cdot)k(\cdot)$, and by $(\Lapd+\lambda^2)^{-1}_{kl}$ the remaining
part, consisting of terms of the form $k(\cdot)l(\cdot)$. Thus
\begin{equation*}
(\Lapd+\lambda^2)^{-1}
=
(\Lapd+\lambda^2)^{-1}_{kk}
+
(\Lapd+\lambda^2)^{-1}_{kl}.
\end{equation*}
By the limiting absorption principle,
\begin{equation*}
dE_{\sqrt{\Lapd}}(\lambda)(x,y)
=
\frac{i\lambda}{\pi}
\left(
\left(\Lapd+(i\lambda)^2\right)^{-1}(x,y)
-
\left(\Lapd+(-i\lambda)^2\right)^{-1}(x,y)
\right),
\end{equation*}
the kernel of $dE_{\sqrt{\Lapd}}(\lambda)$ decomposes accordingly into a
$kk$ part and a $kl$ part, denoted by
$dE_{\sqrt{\Lapd}}^{kk}(\lambda)$ and
$dE_{\sqrt{\Lapd}}^{kl}(\lambda)$.

After applying the limiting absorption principle, the $kl$ contribution is
equivalently an $ll$ contribution. This follows by writing $k$ as a linear
combination involving $I$ and using the evenness of $I$; see
\cite[III.4.4]{COSY} for details.

For the $kk$ part, however, the limiting absorption principle does not appear
to produce any cancellation. We therefore estimate
$\lambda(\Lapd+(i\lambda)^2)^{-1}(x,y)$ and
$\lambda(\Lapd+(-i\lambda)^2)^{-1}(x,y)$ separately.
%\end{remark}

The main aim of this article is to study the Bochner--Riesz means of order
$\delta\ge 0$,
\begin{align*}
\sigma_R^\delta(\Lapd)
&=
\left(1-\frac{\Lapd}{R^2}\right)_+^\delta  \\
&=
\int_0^R
\left(1-\frac{\lambda^2}{R^2}\right)^\delta
dE_{\sqrt{\Lapd}}(\lambda)  \\
&=
\int_0^R
\left(1-\frac{\lambda^2}{R^2}\right)^\delta
dE_{\sqrt{\Lapd}}^{kk}(\lambda)
+
\int_0^R
\left(1-\frac{\lambda^2}{R^2}\right)^\delta
dE_{\sqrt{\Lapd}}^{kl}(\lambda).
\end{align*}
For $\delta=0$, we write $E_{\sqrt{\Lapd}}([0,R))$ for the spectral projection,
so that
\begin{equation*}
\sigma_R^0(\Lapd)
=
E_{\sqrt{\Lapd}}([0,R))
=
E_{\sqrt{\Lapd}}^{kk}([0,R))
+
E_{\sqrt{\Lapd}}^{kl}([0,R)).
\end{equation*}
For $\delta>0$, we similarly write
\begin{equation*}
\sigma_R^\delta(\Lapd)
=
\sigma_R^{\delta,kk}(\Lapd)
+
\sigma_R^{\delta,kl}(\Lapd).
\end{equation*}

For all $\delta\ge 0$, we further split the $kk$ part, up to harmless complex
constants, as
\begin{equation*}
\sigma_R^{\delta,kk}(\Lapd)
=
\sigma_{\mathrm{on}}^\delta(R)
+
\sigma_{\mathrm{off}}^\delta(R),
\end{equation*}
where
\begin{align}\label{kkkernel}
\sigma_{\mathrm{on}}^\delta(R)(x,y)
&:=
\sum_{\ell=1}^2
\int_0^R
\left(1-\frac{\lambda^2}{R^2}\right)^\delta
\lambda
\Big[
B_\ell(i\lambda) k_\ell(i\lambda |x|) k_\ell(i\lambda |y|)
-
B_\ell(-i\lambda) k_\ell(-i\lambda |x|) k_\ell(-i\lambda |y|)
\Big]\,d\lambda,\\
\sigma_{\mathrm{off}}^\delta(R)(x,y)
&:=
\sum_{1\le j\ne \ell\le 2}
\int_0^R
\left(1-\frac{\lambda^2}{R^2}\right)^\delta
\lambda
\Big[
A(i\lambda) k_\ell(i\lambda |x|) k_j(i\lambda |y|)
-
A(-i\lambda) k_\ell(-i\lambda |x|) k_j(-i\lambda |y|)
\Big]\,d\lambda.
\end{align}
Here ``on'' denotes the on-diagonal part, where $x$ and $y$ lie on the same
end, while ``off'' denotes the off-diagonal part, where they lie on different
ends.

In particular, for $\delta=0$ we write
\begin{equation*}
E_{\sqrt{\Lapd}}^{kk}([0,R))
=
E_{\mathrm{on}}(R)
+
E_{\mathrm{off}}(R),
\end{equation*}
where
\begin{equation*}
E_{\mathrm{on}}(R)=\sigma_{\mathrm{on}}^0(R),
\qquad
E_{\mathrm{off}}(R)=\sigma_{\mathrm{off}}^0(R).
\end{equation*}

%%%%%%%%%%%%%%%%%%%%%%%%%%

%%%%%%%%%%%%%%%%%%%%%

\subsection{Hardy--Hilbert type inequalities}
We shall use the following observation from \cite{HE_phd}.

\begin{lemma}[\cite{HE_phd}, Lemma~1.1.2]\label{He}
Let $n_1,n_2,p>1$, and let $K$ be an integral operator with kernel
$K(x,y)$. Assume that $K(x,y)$ is non-negative and homogeneous of degree
$-\delta$, where
\(\delta=\frac{n_2}{p}+\frac{n_1}{p'}.
\)
Suppose further that
\begin{equation}
    \int_0^\infty K(x,1)x^{n_2/p-1}\,dx
    =
    \int_0^\infty K(1,y)y^{n_1/p'-1}\,dy
    <\infty .
\end{equation}
Then $K$ is bounded from
\[
    L^p(\mathbb{R}_+,r^{n_1-1}\,dr)
    \quad\text{to}\quad
    L^p(\mathbb{R}_+,r^{n_2-1}\,dr).
\]
\end{lemma}

We next record a slightly stronger version of \cite[Lemma~2.6]{Nix}.

\begin{lemma}\label{Nix}
Let $n_1,n_2>1$, and let $K$ be the integral operator with kernel
\begin{equation}
    K(x,y)=
    \begin{cases}
       x^{-\alpha}y^{-\beta}, & 1\le x\le y,\\
       x^{-\alpha'}y^{-\beta'}, & x>y,
    \end{cases}
\end{equation}
where $\alpha+\beta=\alpha'+\beta'$. If
\( p(\alpha+\beta-n_1)\ge n_2-n_1
\)
and
\begin{equation*}
    \frac{n_2}{n_2\wedge \alpha'}
    < p <
    \frac{n_1}{0\vee(n_1-\beta)},
\end{equation*}
then the corresponding operator $K$ is bounded from
$L^p([1,\infty),r^{n_1-1}\,dr)$ to $
    L^p([1,\infty),r^{n_2-1}\,dr).$
\end{lemma}

\begin{proof}
The case
\[
    p(\alpha+\beta-n_1)>n_2-n_1
\]
is contained in \cite[Lemma~2.6]{Nix}. It remains to consider the endpoint
case
\[
    p(\alpha+\beta-n_1)=n_2-n_1.
\]
This is equivalent to
\[
    \alpha+\beta=\alpha'+\beta'
    =
    \frac{n_2}{p}+\frac{n_1}{p'}.
\]
Define a kernel on $\mathbb{R}_+\times\mathbb{R}_+$ by
\[
\widetilde K(x,y)=
    \begin{cases}
       x^{-\alpha}y^{-\beta}, & 0<x\le y,\\
       x^{-\alpha'}y^{-\beta'}, & x>y>0.
    \end{cases}
\]
Then $\widetilde K$ is homogeneous of degree
\[
    -\delta=-(\alpha+\beta)=-(\alpha'+\beta').
\]
By Lemma~\ref{He}, it is enough to check that
\[
    \int_0^\infty \widetilde K(x,1)x^{n_2/p-1}\,dx<\infty .
\]
Indeed,
\[
    \int_0^\infty \widetilde K(x,1)x^{n_2/p-1}\,dx
    =
    \int_0^1 x^{-\alpha+n_2/p-1}\,dx
    +
    \int_1^\infty x^{-\alpha'+n_2/p-1}\,dx,
\]
which is finite precisely when
\[
    \alpha<\frac{n_2}{p}
    \qquad\text{and}\qquad
    \alpha'>\frac{n_2}{p}.
\]
Equivalently,
\[
    \frac{n_2}{n_2\wedge \alpha'}
    <
    p
    <
    \frac{n_2}{0\vee\alpha}
    =
    \frac{n_1}{0\vee(n_1-\beta)}.
\]
The corresponding condition for the integral in the $y$ variable is the
same, using
\[
    \alpha+\beta=\alpha'+\beta'
    =
    \frac{n_2}{p}+\frac{n_1}{p'}.
\]
Hence $\widetilde K$ is bounded from
$L^p(\mathbb{R}_+,r^{n_1-1}\,dr)$ to
$L^p(\mathbb{R}_+,r^{n_2-1}\,dr)$. Since $K$ is the restriction of
$\widetilde K$ to $[1,\infty)\times[1,\infty)$, the required boundedness
of $K$ follows.
\end{proof}

\subsection{Remainder terms for special functions}
%\phantomsection
\label{appendix}

Throughout this subsection, all remainder estimates are understood in the symbol sense. In the small-argument regime, if $R(\lambda)=O(\lambda^a)$, then
\begin{equation*}
R^{(k)}(\lambda)=O(\lambda^{a-k}).
\end{equation*}
In the large-argument regime, if
\begin{equation*}
R(\lambda)=e^{\pm i\lambda}r(\lambda),
\qquad r(\lambda)=O(\lambda^a),
\end{equation*}
then
\begin{equation*}
r^{(k)}(\lambda)=O(\lambda^{a-k}).
\end{equation*}
For $d>2$ and $\lambda \in \mathbb R$, we have for $\rho := \min (2, d-2)$, 
\begin{align*}
    &k_d(i|\lambda|) = \begin{cases}
        C_d |\lambda|^{2-d} + \begin{cases}
            O\left( |\lambda|^{2-d+\rho} \right), & d\ne 4,\\
            O\left( |\log{\lambda}| \right), & d=4,
        \end{cases}, & |\lambda|\le 1, \\
        C_d |\lambda|^{-\frac{d-1}{2}} e^{-i\lambda} +  O\left( |\lambda|^{-\frac{d+1}{2}} \right) e^{-i|\lambda|}, & |\lambda|\ge 1.
    \end{cases}\\
    &k_d'(i|\lambda|) = \begin{cases}
        C_d |\lambda|^{1-d} + O\left(|\lambda|^{1-d+\rho} \right), & |\lambda|\le 1, \\
        C_d |\lambda|^{-\frac{d-1}{2}} e^{-i|\lambda|} +  O\left( |\lambda|^{-\frac{d+1}{2}} \right) e^{-i|\lambda|}, & |\lambda|\ge 1.
    \end{cases}\\
    &l_d(i|\lambda|) = \begin{cases}
        C_d + O\left(|\lambda|^2 \right), & |\lambda|\le 1, \\
        C_d |\lambda|^{-\frac{d-1}{2}} e^{i|\lambda|} + C_d|\lambda|^{-\frac{d-1}{2}} e^{-i|\lambda|} + O\left(|\lambda|^{-\frac{d+1}{2}}\right) e^{i|\lambda|} + O\left(|\lambda|^{-\frac{d+1}{2}}\right) e^{-i|\lambda|}, & |\lambda|\ge 1.
    \end{cases}\\
    &l_d'(i|\lambda|) = \begin{cases}
        C_d |\lambda| + O\left(|\lambda|^3 \right), & |\lambda|\le 1, \\
        C_d|\lambda|^{-\frac{d-1}{2}} e^{i|\lambda|} + C_d|\lambda|^{-\frac{d-1}{2}} e^{-i|\lambda|} + O\left(|\lambda|^{-\frac{d+1}{2}}\right) e^{i|\lambda|} + O\left(|\lambda|^{-\frac{d+1}{2}}\right) e^{-i|\lambda|}, & |\lambda|\ge 1.
    \end{cases}
\end{align*}
For $d_1, d_2>2$, we have for $\rho = \min (2, d_1-2, d_2-2)$,

\begin{align*}
    &A(i|\lambda|) = \begin{cases}
        C_{d_1,d_2} |\lambda|^{d_1+d_2-4} + O\left(|\lambda|^{d_1+d_2-4+\rho} \right), & |\lambda|\le 1, \\
        C_{d_1,d_2} |\lambda|^{\frac{d_1+d_2-4}{2}} e^{2i|\lambda|} + O\left(|\lambda|^{\frac{d_1+d_2-6}{2}} \right) e^{2i|\lambda|}, & |\lambda|\ge 1.
    \end{cases}\\
    &B_2(i|\lambda|) = \begin{cases}
        C_{d_1,d_2} |\lambda|^{2d_2-4} + O\left(|\lambda|^{2d_2-4+\rho} \right), & |\lambda|\le 1, \\
        C_{d_1,d_2}|\lambda|^{d_2-2} + O\left(|\lambda|^{d_2-3}\right) e^{2i |\lambda|} + O\left(|\lambda|^{d_2-3}\right), & |\lambda|\ge 1.
    \end{cases}\\
    &B_1(i|\lambda|) = \begin{cases}
        C_{d_1,d_2} |\lambda|^{2d_1-4} + O\left(|\lambda|^{2d_1-4+\rho} \right), & |\lambda|\le 1, \\
        C_{d_1,d_2}|\lambda|^{d_1-2} + O\left(|\lambda|^{d_1-3}\right) e^{2i |\lambda|} + O\left(|\lambda|^{d_1-3}\right), & |\lambda|\ge 1.
    \end{cases}
\end{align*}

%%%%%%%%%%%%%%%%%%%%%%%%%

\part{Positive estimates: the Herz-type theory}

\section{\texorpdfstring{$L^p$}{Lp} boundedness of spectral projections}\label{sec3}
%%%%%%%%%%%%%%%%%%%%%%

\subsection{Spectral projection for the \texorpdfstring{$kl$}{kl} part}

We consider the $kl$ part of the order-zero Bochner--Riesz operator,
$E_{\sqrt{\Lapd}}^{kl}(R)$. This term falls precisely within the scope of
Herz's original result. Set
\begin{equation}
    \mu =
    \begin{cases}
        d_2, & x,y\in [1,\infty),\\
        d_1, & x,y\in (-\infty,-1],
    \end{cases}
    \qquad
    \nu=\frac{\mu}{2}-1.
\end{equation}
By the resolvent formula from the preceding section, the $kl$ part of
$(\Lapd+\lambda^2)^{-1}(x,y)$ is
\begin{align*}
    C\lambda^{\mu-2}
    \begin{cases}
        k_{\mu}(\lambda |y|)l_{\mu}(\lambda |x|), & |y|\ge |x|,\\
        k_{\mu}(\lambda |x|)l_{\mu}(\lambda |y|), & |y|\le |x|,
    \end{cases}
    =
    C(|x||y|)^{1-\mu/2}
    \begin{cases}
        K_{\nu}(\lambda |y|)I_{\nu}(\lambda |x|), & |y|\ge |x|,\\
        K_{\nu}(\lambda |x|)I_{\nu}(\lambda |y|), & |y|\le |x|.
    \end{cases}
\end{align*}
Here $C>0$ depends only on the dimensional parameters.

Following the computation in \cite[III.4.4]{COSY}, the corresponding
spectral measure for the Neumann operator is
\begin{align}\label{spectral_measure}
    dE_{\sqrt{\LapoN}}(\lambda)(x,y)
    =
    i\nu e^{i\pi\nu}\lambda
    |x|^{1-\mu/2}|y|^{1-\mu/2}
    I_{\nu}(i\lambda |x|)I_{\nu}(i\lambda |y|).
\end{align}
Although \cite{COSY} treats the inverse-square potential in dimension
$d>2$, the same argument applies to $\Lapd$ and $\LapoN$ for all
$d,d_1,d_2>1$.

We begin with the following observation, whose key ingredient goes back to Herz.

\begin{theorem}\label{herz_kl}
Let $d_1,d_2\ge 1$ and set $D=d_1\vee d_2$. Then
\[
    \sup_{R>0}\left\|E_{\sqrt{\Lapd}}^{kl}(R)\right\|_{L^p\to L^p}
    <\infty
\]
for all
\[
    \left|\frac{1}{p}-\frac{1}{2}\right|<\frac{1}{2D}.
\]
\end{theorem}

\begin{proof}
It suffices first to prove the corresponding estimate for the Neumann
operator on one half-line. We write $\LapoN$ for this operator with
dimension parameter $\mu \ge 1$. The statement for $\Lapd$ then follows from
the decomposition
\begin{equation}\label{eq:Lapd-decomposition}
E_{\sqrt{\Lapd}}^{kl}(R)
=
M_{\chi_{(-\infty,-1]}}
E_{\sqrt{\Delta^{(-\infty,0)}_{\Neu,d_1}}}^{kl}(R)
M_{\chi_{(-\infty,-1]}}
+
M_{\chi_{[1,\infty)}}
E_{\sqrt{\Delta^{(0,\infty)}_{\Neu,d_2}}}^{kl}(R)
M_{\chi_{[1,\infty)}},
\end{equation}
where $M_{\chi_A}$ denotes multiplication by $\chi_A$, and
$\Delta^{(-\infty,0)}_{\Neu,d_i}$ is the symmetric reflection of
$\Delta^{(0,\infty)}_{\Neu,d_i}$.

The argument follows Herz \cite{Herz} and Córdoba \cite{cordoba}; we
include the details for completeness. By \eqref{spectral_measure}, the
kernel of the spectral projection is
\[
E_{\sqrt{\LapoN}}^{kl}(R)(x,y)
=
C(|x||y|)^{1-\mu/2}
\int_0^R I_{\nu}(i\lambda |x|)I_{\nu}(i\lambda |y|)
\lambda\,d\lambda .
\]
Using \cite[6.521]{Grad-Ry}, we obtain
\[
E_{\sqrt{\LapoN}}^{kl}(R)(x,y)
=
C(|x||y|)^{1-\mu/2}
\frac{
R|x| I'_{\nu}(iR|x|) I_\nu(iR|y|)
-
R|y| I'_{\nu}(iR|y|) I_\nu(iR|x|)
}{x^2-y^2}.
\]
Hence
\[
    E_{\sqrt{\LapoN}}^{kl}(R)(x,y)
    =
    C R^\mu E_{\sqrt{\LapoN}}^{kl}(1)(Rx,Ry).
\]
If $\delta_R f(x)=f(Rx)$, then
\(E_{\sqrt{\LapoN}}^{kl}(R)f(x)
    =
    E_{\sqrt{\LapoN}}^{kl}(1)(\delta_{R^{-1}}f)(Rx)
\).
The dilation factors cancel in
$L^p(\mathbb{R}_+,r^{\mu-1}\,dr)$, so it remains to prove the boundedness
of $E_{\sqrt{\LapoN}}^{kl}(1)$.

Using
\(zI_\nu'(z)=\nu I_\nu(z)+zI_{\nu+1}(z),
\)
we get
\begin{align}
E_{\sqrt{\LapoN}}^{kl}(1)(x,y)
&=
C(|x||y|)^{1-\mu/2}
\frac{
|x| I_\nu'(i|x|)I_\nu(i|y|)
-
|y| I_\nu'(i|y|)I_\nu(i|x|)
}{x^2-y^2}
\nonumber\\
&=
C(|x||y|)^{1-\mu/2}
\frac{
i|x|I_{\nu+1}(i|x|)I_\nu(i|y|)
-
i|y|I_{\nu+1}(i|y|)I_\nu(i|x|)
}{x^2-y^2}
\nonumber\\
&=
C(|x||y|)^{\nu+1-\mu/2}
\frac{
x^2 l_{\mu+2}(i|x|)l_\mu(i|y|)
-
y^2 l_{\mu+2}(i|y|)l_\mu(i|x|)
}{x^2-y^2}.
\label{1}
\end{align}
Recall that
\(l_\mu(ir)=C r^{-\nu}I_\nu(ir),
\)
and, asymptotically,
\[
|l_\mu(ir)|\sim C
\begin{cases}
1, & r\to 0,\\[1mm]
r^{-\nu-\frac12}, & r\to\infty .
\end{cases}
\]
It is therefore enough to prove the boundedness of the two operators
\begin{align*}
T_1 f(x)
&=
\int_0^\infty
(xy)^{\nu+1-\mu/2}
\frac{x^2 l_{\mu+2}(ix)l_\mu(iy)}{x^2-y^2}
f(y)y^{\mu-1}\,dy,\\
T_2 f(x)
&=
\int_0^\infty
(xy)^{\nu+1-\mu/2}
\frac{y^2 l_{\mu+2}(iy)l_\mu(ix)}{x^2-y^2}
f(y)y^{\mu-1}\,dy .
\end{align*}
We treat $T_1$; the estimate for $T_2$ follows by duality.

Set $t=x^2$, $s=y^2$, and
\(\phi(s)=f(\sqrt{s})s^{(\mu-2)/(2p)}.
\)
Then
\begin{align*}
\|T_1f\|_{L^p(\mathbb{R}_+,r^{\mu-1}\,dr)}^p
=
C\int_0^\infty
\left|\int_0^\infty
\left(\frac{s}{t}\right)^{
\frac{\mu-2}{4}-\frac{\mu-2}{2p}-\frac14}
\left[t^{\frac{\nu}{2}+\frac34}l_{\mu+2}(i\sqrt t)\right]
\left[s^{\frac{\nu}{2}+\frac14}l_\mu(i\sqrt s)\right]
\frac{\phi(s)}{t-s}\,ds
\right|^p dt .
\end{align*}
The factors
\(t^{\frac{\nu}{2}+\frac34}l_{\mu+2}(i\sqrt t)\) and
    \(s^{\frac{\nu}{2}+\frac14}l_\mu(i\sqrt s)\)
are uniformly bounded on $(0,\infty)$. Hence it remains to consider
\[
    \phi\mapsto
    \int_0^\infty
    \frac{\phi(s)}{t-s}
    \left(\frac{t}{s}\right)^{
    \frac14-\frac{\mu-2}{4}+\frac{\mu-2}{2p}}
    \,ds .
\]
The Hilbert transform gives the boundedness of the part without the
weight. Thus it remains to estimate
\begin{equation}\label{operator}
    \phi\mapsto
    \int_0^\infty
    \frac{\phi(s)}{t-s}
    \left[
    \left(\frac{t}{s}\right)^{
    \frac14-\frac{\mu-2}{4}+\frac{\mu-2}{2p}}
    -1
    \right]\,ds .
\end{equation}
By the Hardy--Littlewood--Pólya inequality, equivalently by
Lemma~\ref{He}, the operator in \eqref{operator} is bounded on
$L^p(\mathbb{R}_+,dt)$ in the range
\(\frac{2\mu}{\mu+1}<p<\frac{2\mu}{\mu-1}.
\)
Thus $T_1$ is bounded in this range; by duality, so is $T_2$. Therefore
\[
    \sup_{R>0}
    \left\|E_{\sqrt{\LapoN}}^{kl}(R)\right\|_{p\to p}
    <\infty
    \qquad
    \text{whenever}
    \qquad
    \left|\frac1p-\frac12\right|<\frac{1}{2\mu}.
\]
Applying this with $\mu=d_1$ and $\mu=d_2$, and using
$D=d_1\vee d_2$, gives the asserted estimate for $\Lapd$.
\end{proof}

\begin{remark}
Note that in \cite{Herz}, Herz considers the subspace of harmonic polynomials of degree $\sigma$. This leads to the equation involving 
$I_{\nu+\sigma}$ and is equivalent to our discussion for appropriate choice of $c$. 

Suppose $P$ is a harmonic polynomial with order $\nu$. Then, $P(t) = r^\nu \phi(\theta)$, where $r>0$, $\theta \in S^{n-1}$ and $\Delta_n P = 0$. Note that
\begin{align*}
    0 = \Delta_n P(t) &= \partial_r^2 P + \frac{n-1}{r}\partial_r P + \frac{1}{r^2}\Delta_{S^{n-1}}P\\
    &= \phi(\theta) \nu(\nu-1) r^{\nu-2} + \phi(\theta) \nu(n-1) r^{\nu-2} + r^{\nu-2} \Delta_{S^{n-1}}\phi(\theta).
\end{align*}
Hence,
\begin{align*}
    \Delta_{S^{n-1}}\phi(\theta) = - \nu(n+\nu-2)\phi(\theta).
\end{align*}
Let $f$ be a radial function. Set $F(t) = f(r)\frac{P(r,\theta)}{r^\nu}$. Then
\begin{align*}
    \Delta_n F(t) &= f''(r) \phi(\theta) + \frac{n-1}{r} f'(r) \phi(\theta) + \frac{f(r)}{r^2} \Delta_{S^{n-1}}\phi(\theta)\\
    &= - \phi(\theta) \left(-\partial_r^2 - \frac{n-1}{r}\partial_r + \frac{\nu(\nu+n-2)}{r^2}\right)f(r).
\end{align*}

\end{remark}

\subsection{A remark on Muckenhoupt weights}

Herz's proof of the boundedness of spectral projections for Bessel
operators is closely related to Muckenhoupt's theory of weighted singular
integrals. Indeed, several calculations in \cite{Muck1,Muck2,Muckenhoupt}
closely parallel Herz's argument, although the approaches appear to have
been developed independently. We briefly indicate how the method used
above can be reformulated in this language.

Recall that $|x|^\alpha\in A_p(\mathbb R)$ if and only if
$-1<\alpha<p-1$. Hence the Hilbert transform is bounded on
$L^p(\mathbb R,|x|^\alpha dx)$ precisely in this range. Equivalently,
\[
\left\|
\int_{\mathbb R}
\frac{g(y)}{x-y}
\left|\frac{x}{y}\right|^{\alpha/p}\,dy
\right\|_{L^p(\mathbb R)}
\lesssim
\|g\|_{L^p(\mathbb R)},
\qquad -1<\alpha<p-1.
\]
By duality, the analogous estimate with $(y/x)^{\alpha/p}$ holds in
$L^{p'}(\mathbb R)$.

For Herz's projection, the argument above reduces the problem to the
boundedness on $L^p(\mathbb R_+,dx)$ of the two model operators
\[
T_1\phi(x)=
\int_0^\infty
\frac{\phi(y)}{x-y}
\left(\frac{x}{y}\right)^\gamma dy,
\qquad
T_2\phi(x)=
\int_0^\infty
\frac{\phi(y)}{x-y}
\left(\frac{y}{x}\right)^\gamma dy,
\]
where
\(\gamma=\frac{3-n}{4}+\frac{n-2}{2p}.\)
Extending $\phi$ by zero to $\mathbb R$, the weighted Hilbert transform
estimate gives the boundedness of $T_1$ whenever
$-1<\gamma p<p-1$. This is equivalent to
\(p>\frac{2n}{n+1}\) 
    and, if $ n>3$, then  $p<\frac{2n}{n-3}$.
The estimate for $T_2$ follows by duality, giving the dual range
\(p<\frac{2n}{n-1}\)
    {and, if } $n>3$, then 
    $p>\frac{2n}{n+3}$.
The additional restrictions for $n>3$ are weaker on the intersection.
Thus the common admissible range is exactly
\[
    \frac{2n}{n+1}<p<\frac{2n}{n-1}.
\]

Thus the Muckenhoupt-weight formulation recovers the same range as
Herz's original argument based on subtracting the Hilbert transform. In
this sense, Herz's method can be viewed as an early instance of weighted
singular integral theory.

\subsection{Spectral projections for the \texorpdfstring{$kk$}{kk} part}

As noted in Section~\ref{ssec2.5}, the limiting absorption
principle does not appear to yield useful cancellation. We therefore
consider the following kernels directly:
\begin{align*}
E_{\mathrm{on}}(R)(x,y)
&=
\sum_{\ell=1}^2 \sum_{\sigma=\pm1}
\sigma \int_0^R
\lambda B_\ell(\sigma i\lambda)
k_\ell(\sigma i\lambda |x|)
k_\ell(\sigma i\lambda |y|)\,d\lambda
=
\sum_{\ell=1}^2 \sum_{\sigma=\pm1}
E_{\mathrm{on}}^{\ell,\sigma}(R)(x,y),\\
E_{\mathrm{off}}(R)(x,y)
&=
\sum_{\substack{j,\ell=1\\ j\ne \ell}}^2
\sum_{\sigma=\pm1}
\sigma \int_0^R
\lambda A(\sigma i\lambda)
k_\ell(\sigma i\lambda |x|)
k_j(\sigma i\lambda |y|)\,d\lambda
=
\sum_{\substack{j,\ell=1\\ j\ne \ell}}^2
\sum_{\sigma=\pm1}
E_{\mathrm{off}}^{\ell j,\sigma}(R)(x,y).
\end{align*}

\begin{theorem}\label{herz_doubling}
Let $d>2$. Then 
\begin{align*}
    \sup_{R>0} \| E_{\sqrt{\Lapdd}}(R)\|_{p \to p} \le C \quad \textit{if} \quad 
        \left| \frac{1}{2} - \frac{1}{p} \right| < \frac{1}{2d}.
\end{align*}
\end{theorem}

\begin{proof}
By Theorem~\ref{herz_kl}, it remains to prove that
\begin{equation*}
    \sup_{R>0} \| E_{\sqrt{\Lapdd}}^{kk}(R) \|_{p\to p}\le C.
\end{equation*}
Since $d=d_1=d_2$, we have $B_1(z)=B_2(z)$ and $k_d=k_1=k_2$. Moreover, the terms corresponding to the two signs are treated in the same way. Thus it suffices to estimate the representative kernels
\begin{equation*}
    E_{\mathrm{off}}^{\ell j,+}(R)(x,y) = \int_0^R  \lambda  A(i\lambda) k_d( i\lambda |x|) k_d( i\lambda |y|) d\lambda,
\end{equation*}
and
\begin{equation*}
    E_{\mathrm{on}}^{\ell,+}(R)(x,y) =  \int_0^R \lambda B_1( i\lambda) k_d( i\lambda |x|) k_d( i\lambda |y|) d\lambda.
\end{equation*}
Assume first that $R\ge 1$. We begin with the off-diagonal term. 
 Assume that $|x|\ge |y|$ and then split the integral $\int_0^R = \int_0^{|x|^{-1}} + \int_{|x|^{-1}}^{|y|^{-1}}+\int_{|y|^{-1}}^1 + \int_1^R:= \sum_{j=1}^4 \mathcal{I}_{\mathrm{off}}^j$. By asymptotic formulas for special functions, we obtain
\begin{align*}
    \mathcal{I}_{\mathrm{off}}^1(x,y) \sim \int_0^{|x|^{-1}} \lambda^{d-1} (\lambda |x|)^{2-d} (\lambda |y|)^{2-d} \lambda^{d-2} d\lambda \simeq (|x||y|)^{2-d} \int_0^{|x|^{-1}} \lambda d\lambda \sim |x|^{-d} |y|^{2-d}.
\end{align*}
Similarly,
\begin{align*}
    \mathcal{I}_{\mathrm{off}}^2(x,y) &\sim \int_{|x|^{-1}}^{|y|^{-1}} \lambda^{d-1} \lambda^{d-2} (\lambda |x|)^{\frac{1-d}{2}} (\lambda |y|)^{2-d} e^{- i\lambda |x|} d\lambda\\
    & \simeq |x|^{-d} |y|^{2-d} \int_1^{|x|/|y|} s^{\frac{d-1}{2}} e^{- is} ds.
\end{align*}
Note that by integration by parts,
\begin{align*}
    \int_1^{|x|/|y|} s^{\frac{d-1}{2}} e^{- is} ds = \frac{1}{- i} \left[ s^{\frac{d-1}{2}}e^{- is}\big|_{s=1}^{|x|/|y|} - \frac{d-1}{2} \int_1^{|x|/|y|} e^{- is} s^{\frac{d-1}{2}-1} ds \right].
\end{align*}
It follows that
\begin{align*}
    |\mathcal{I}_{\mathrm{off}}^2(x,y)|\lesssim |x|^{-d} |y|^{2-d} \left(\frac{|x|}{|y|}\right)^{\frac{d-1}{2}}.
\end{align*}
The same argument gives
\begin{align*}
    \mathcal{I}_{\mathrm{off}}^3(x,y) &\sim \int_{|y|^{-1}}^1 \lambda^{d-1} \lambda^{d-2} (\lambda |x|)^{\frac{1-d}{2}} (\lambda |y|)^{\frac{1-d}{2}} e^{- i\lambda(|x|+|y|)} d\lambda\\
    &\simeq (|x||y|)^{\frac{1-d}{2}} \int_{|y|^{-1}}^1 \lambda^{d-2} e^{-i\lambda (|x|+|y|)} d\lambda.
\end{align*}
Integrating by parts again yields
\begin{align*}
    |\mathcal{I}_{\mathrm{off}}^3(x,y)|\lesssim \frac{(|x||y|)^{\frac{1-d}{2}}}{|x|+|y|}.
\end{align*}
As for $\mathcal{I}_{\mathrm{off}}^4$, it is plain that
\begin{align*}
    \mathcal{I}_{\mathrm{off}}^4(x,y) &\sim \int_1^R \lambda^{d-1} \lambda^{1-d} (|x||y|)^{\frac{1-d}{2}} e^{- i\lambda(|x|+|y|-2)} d\lambda \\
    &\simeq (|x||y|)^{\frac{1-d}{2}} \int_1^R e^{- i\lambda(|x|+|y|-2)} d\lambda,
\end{align*}
which implies 
\begin{align*}
    |\mathcal{I}_{\mathrm{off}}^4(x,y)| \lesssim \frac{(|x||y|)^{\frac{1-d}{2}}}{|x|+|y|-2}.
\end{align*}
It follows by symmetry (for the case $|x|\le |y|$, we only need to change the roles of $x$ and $y$ in the above estimates) that 
\begin{align*}
    |E_{\mathrm{off}}(R)(x,y)| \lesssim \begin{cases}
        |x|^{-\frac{3d-5}{2}} |y|^{-\frac{d+1}{2}}, & |x|\le |y|,\\
        |x|^{-\frac{d+1}{2}} |y|^{-\frac{3d-5}{2}}, & |x|\ge |y|,
    \end{cases} + \frac{(|x||y|)^{\frac{1-d}{2}}}{|x|+|y|-2}.
\end{align*}
By Lemma~\ref{Nix}, the first term above is bounded in $L^p([1,\infty),r^{d-1}dr)$ for 
\begin{equation*}
    \frac{d}{d \wedge \frac{d+1}{2}} < p < \frac{d}{0 \vee (d-\frac{d+1}{2})} \implies \frac{2d}{d+1}<p<\frac{2d}{d-1}
\end{equation*}
as expected. While for the second term, if $|x|+|y|\ge 4$ and hence $|x|+|y|-2 \ge c(|x|+|y|)$, then we just need to treat kernel
\begin{equation*}
 \frac{(|x||y|)^{\frac{1-d}{2}}}{|x|+|y|},
\end{equation*}
which is apparently homogeneous of degree $-d$. By Lemma~\ref{He}, it acts as a bounded operator on $L^p([1,\infty), r^{d-1}dr)$ provided
\begin{align*}
    \int_0^\infty x^{\frac{1-d}{2}} (x+1)^{-1} x^{\frac{d}{p}-1} dx <\infty \implies \frac{2d}{d+1}<p<\frac{2d}{d-1}.
\end{align*}
As for the case $2\le |x|+|y|\le 4$, one simply bounds the kernel by
\begin{equation*}
    \frac{\mathbf{1}_{2\le x+y\le 4}}{(x-1)+(y-1)}.
\end{equation*}
By the Hardy--Littlewood--Pólya inequality \cite[Theorem~319]{HLP}
\begin{align*}
    \int_1^{\infty} &\left| \int_1^\infty \frac{\chi_{2\le x+y\le 4}}{(x-1)+(y-1)} f(y) y^{d-1}dy \right|^p x^{d-1}dx\\
    &\lesssim \int_0^\infty \left| \int_0^\infty \frac{|f(s+1)|\chi_{s\le 2}(s)}{t+s} ds \right|^p dt\\
    &\lesssim \int_0^\infty |f(1+s)|^p (1+s)^{d-1} ds\\
    &\simeq \|f\|_{L^p([1,\infty), r^{d-1}dr)}^p.
\end{align*}
for all $1<p<\infty$. Consequently
\begin{equation*}
    \sup_{R\ge 1} \| E_{\mathrm{off}}(R)\|_{p\to p} \le C,\quad \mbox{\rm for all}\quad  \frac{2d}{d+1}<p<\frac{2d}{d-1}.
\end{equation*}
To this end, we also need to consider $E_{\mathrm{on}}(R)$. A parallel argument as above yields
\begin{align*}
|E_{\mathrm{on}}(R)(x,y)|\lesssim \begin{cases}
|x|^{-\frac{3d-1}{2}} |y|^{-\frac{d+1}{2}}, & |x|\le |y|,\\
|x|^{-\frac{d+1}{2}} |y|^{-\frac{3d-1}{2}}, & |x|\ge |y|,
\end{cases}+ \frac{(|x||y|)^{\frac{1-d}{2}}}{|x|+|y|-2}.
\end{align*}
By Lemma~\ref{Nix}, Lemma~\ref{He} and the Hardy--Littlewood--Pólya inequality \cite[Theorem~319]{HLP}, we conclude that $\sup_{R\ge 1}\| E_{\mathrm{on}}(R)\|_{p\to p}\le C $ for $|1/p-1/2|<(2d)^{-1}$. This completes the proof of Theorem~\ref{herz_doubling} for $R\ge 1$.

The case $0<R<1$ is similar. Indeed, we decompose
\begin{equation*}
    \int_0^R = \int_0^{R\wedge |x|^{-1}} + \int_{|x|^{-1}}^{R\wedge |y|^{-1}} + \int_{|y|^{-1}}^R
\end{equation*}
and analyze them similarly as before. We omit details.

\end{proof}

\section{ \texorpdfstring{$L^p$}{Lp} boundedness of the Bochner--Riesz mean}\label{sec4}

We now study the Bochner--Riesz means
\begin{equation*}
    \sigma_{R}^\delta(\Lapd) = \left(1-\frac{\Lapd}{R^2}\right)_+^\delta,
\qquad R>0.
\end{equation*}
We first record the standard one-dimensional oscillatory integral estimate.

\begin{lemma}\label{lem:osc-expansion-order-M}
Let $\delta>-1$, let $M\in\mathbb{N}$ with $M\geq 1$, and let
$\psi\in C_c^\infty((-1,1))$. Then, for every $t\geq 1$,
\begin{equation}\label{eq_oscillatory_expansion}
    \int_0^\infty e^{its}s^\delta\psi(s)\,ds
    =
    \sum_{m=0}^{M-1}
    c_{m,\delta}\psi^{(m)}(0)t^{-1-\delta-m}
    +R_M(t),
\end{equation}
where
\begin{equation*}
    c_{m,\delta}
    :=
    \frac{e^{\frac{i\pi}{2}(\delta+m+1)}
    \Gamma(\delta+m+1)}{m!},
\end{equation*}
and
\begin{equation}\label{eq_oscillatory_remainder}
    |R_M(t)|
    \leq
    C_{\delta,M}t^{-M}
    \|\psi^{(M)}\|_\infty .
\end{equation}
Here $C_{\delta,M}$ is independent of $t$ and $\psi$.
\end{lemma}

\begin{proof}
We follow the construction used in
\cite[Theorem~2.10]{IoLi2014}. In the notation of that theorem,
$\lambda=\delta+1$. Its stated range $0<\lambda<1$ corresponds to
$-1<\delta<0$; the same argument applies when $\delta\geq0$ after
modifying the estimate for the successive primitives.

Fix $t\geq1$ and set
\begin{equation*}
    g(z):=e^{itz}z^\delta,
\end{equation*}
where $z^\delta$ denotes the principal branch on
$\mathbb{C}\setminus(-\infty,0]$. For $k\geq1$ and $s>0$, define
\begin{equation*}
    (\mathcal A_t^k g)(s)
    :=
    \frac{1}{(k-1)!}
    \int_{s+i\infty}^{s}
    (s-z)^{k-1}g(z)\,dz,
\end{equation*}
where the integral is taken downwards along the vertical ray
$\{s+i\xi:\xi\geq0\}$. Making the substitution $z=s+i\xi$, we obtain
\begin{equation}\label{eq_vertical_primitive}
    (\mathcal A_t^k g)(s)
    =
    \frac{e^{its-\frac{ik\pi}{2}}}{(k-1)!}
    \int_0^\infty
    e^{-t\xi}\xi^{k-1}(s+i\xi)^\delta\,d\xi .
\end{equation}
Letting $s\to 0$ in \eqref{eq_vertical_primitive} gives
\begin{equation}\label{eq_vertical_primitive_endpoint}
    (\mathcal A_t^k g)(0) =
    \frac{e^{-\frac{ik\pi}{2}}}{(k-1)!}
    \int_0^\infty
    e^{-t\xi}\xi^{k-1}(i\xi)^\delta\,d\xi  =
    e^{\frac{i\pi}{2}(\delta-k)}
    \frac{\Gamma(k+\delta)}{(k-1)!}
    t^{-k-\delta}.
\end{equation}
We next estimate $\mathcal A_t^k g$. If $-1<\delta<0$, then $(s^2+\xi^2)^{\delta/2}\leq s^\delta$, and hence 
\begin{equation}\label{eq_negative_delta_primitive}
    |(\mathcal A_t^k g)(s)|
    \leq s^\delta t^{-k}.
\end{equation}
If $\delta\geq0$, then $(s^2+\xi^2)^{\delta/2}
    \leq C_\delta(s^\delta+\xi^\delta)$, so \eqref{eq_vertical_primitive} yields
\begin{equation}\label{eq_positive_delta_primitive}
    |(\mathcal A_t^k g)(s)|
    \leq
    C_{\delta,k}
    \left(
        s^\delta t^{-k}
        +t^{-k-\delta}
    \right).
\end{equation}
It follows from \eqref{eq_negative_delta_primitive} and
\eqref{eq_positive_delta_primitive} that, for every $\delta>-1$,
\begin{equation}\label{eq_integrated_primitive}
    \int_0^1 |(\mathcal A_t^k g)(s)|\,ds
    \leq C_{\delta,k}t^{-k},
    \qquad t\geq1.
\end{equation}
The functions $\mathcal A_t^k g$ are successive primitives of $g$:
\begin{equation*}
    \frac{d}{ds}\mathcal A_t^1g(s)=g(s),
    \qquad
    \frac{d}{ds}\mathcal A_t^kg(s)
    =\mathcal A_t^{k-1}g(s),
    \quad k\geq2.
\end{equation*}
Applying integration by parts repeatedly on $[\varepsilon,1]$ and
then letting $\varepsilon\to 0$, justified by the estimates
above, gives
\begin{align*}
    \int_0^1\psi(s)g(s)\,ds =
    \left[
        \sum_{m=0}^{M-1}
        (-1)^m\psi^{(m)}(s)
        (\mathcal A_t^{m+1}g)(s)
    \right]_{s=0}^{s=1} + (-1)^M
    \int_0^1
    \psi^{(M)}(s)(\mathcal A_t^Mg)(s)\,ds .
\end{align*}
Since $\psi\in C_c^\infty((-1,1))$, all the boundary terms at
$s=1$ vanish. Moreover, by
\eqref{eq_vertical_primitive_endpoint},
\begin{align*}
    (-1)^{m+1}(\mathcal A_t^{m+1}g)(0) =
    e^{\frac{i\pi}{2}(\delta+m+1)}
    \frac{\Gamma(\delta+m+1)}{m!}
    t^{-1-\delta-m} =c_{m,\delta}t^{-1-\delta-m}.
\end{align*}
Consequently, \eqref{eq_oscillatory_expansion} holds with
\begin{equation*}
    R_M(t)
    :=
    (-1)^M
    \int_0^1
    \psi^{(M)}(s)(\mathcal A_t^Mg)(s)\,ds.
\end{equation*}
Finally, \eqref{eq_integrated_primitive} implies
\begin{equation*}
    |R_M(t)|
    \leq
    \|\psi^{(M)}\|_\infty
    \int_0^1|(\mathcal A_t^Mg)(s)|\,ds
    \leq
    C_{\delta,M}t^{-M}\|\psi^{(M)}\|_\infty,
\end{equation*}
which proves \eqref{eq_oscillatory_remainder}.
\end{proof}

% \begin{remark}
% Classical Erd\'elyi-type endpoint expansions (see \cite{Er1956}) for oscillatory integrals of this form can be found, for example, in \cite[p.~355, Section~5.1(d)]{Stein} and \cite[Example~7.1.17]{Hormander}. In particular, these results give the required leading term, and indeed a full asymptotic expansion, when the amplitude is fixed. In our applications below, however, the amplitude depends on the spatial variables $(x,y)$ and, after rescaling, also on $R$. Thus, a remainder whose implicit constant is allowed to depend on the amplitude does not by itself provide the uniform estimates required below. Lemma~\ref{lem:osc-expansion-order-M} makes this dependence explicit. Following the method of Iosevich and Liflyand, its proof constructs successive primitives of $e^{its}s^\delta$ by integration along a vertical ray in the upper half-plane. Their endpoint values are evaluated by Gamma integrals, while repeated integration by parts yields a remainder controlled by $t^{-M}\|\psi^{(M)}\|_\infty$.
% \end{remark}

\begin{remark}
Classical Erd\'elyi-type endpoint expansions for oscillatory integrals of this kind (see \cite{Er1956}) are given, for example, in \cite[p.~355, Section~5.1(d)]{Stein} and \cite[Example~7.1.17]{Hormander}. For a fixed amplitude, they yield both the required leading term and a full asymptotic expansion. In our applications, however, the amplitude depends on the spatial variables $(x,y)$ and, after rescaling, on $R$. A remainder estimate whose implicit constant may depend on the amplitude is therefore insufficient for the uniform bounds needed below. Lemma~\ref{lem:osc-expansion-order-M} makes this dependence explicit. Following Iosevich and Liflyand, its proof constructs successive primitives of $e^{its}s^\delta$ by integration along a vertical ray in the upper half-plane, evaluates their endpoint values using Gamma integrals, and obtains a remainder bounded by $t^{-M}\|\psi^{(M)}\|_\infty$ through repeated integration by parts.
\end{remark}

We now state the main result of this section.

\begin{theorem}\label{thm:BRR}
Assume that $d_1,d_2>2$.
The Bochner--Riesz means of order $\delta$, $\sigma_{R}^\delta(\Lapd)$, is bounded in $L^p$ uniformly in $R>0$, that is,
\begin{equation*}
    \sup_{R>0}\left\|\left(1-\frac{\Lapd}{R^2}\right)_+^\delta\right\|_{p\to p}<\infty,
\end{equation*}
provided that
\begin{equation}\label{BRrange}
\delta>\max\left(D\left|\frac1p-\frac12\right|-\frac12,0\right)
\quad\text{and}\quad
\delta \ge |d_1-d_2|\left|\frac1p-\frac12\right|,
\tag{9}
\end{equation}
where \(D=d_1\vee d_2\).
\end{theorem}

\begin{proof}%[Proof of Theorem~\ref{thm:BRR}]
We begin with some standard reductions. By the decomposition from Subsection~\ref{ssec2.5}
\begin{equation*}
    \sigma_{R}^\delta(\Lapd)
=
\sigma_{R}^{\delta,kl}(\Lapd) + \sigma_{R}^{\delta,kk}(\Lapd).
\end{equation*}
As in the previous sections, it suffices to consider the $kk$-part. Indeed, the $kl$-part can be handled by writing
\begin{equation*}
    \sigma_{R}^{\delta,kl}(\Lapd) = M_{\chi_{(-\infty,-1]}} \sigma_{R}^{\delta,kl}\big(\Delta^{(-\infty,0)}_{\mathrm{Neu},d_1}\big)  M_{\chi_{(-\infty,-1]}} + M_{\chi_{[1,\infty)}} \sigma_{R}^{\delta,kl}\big(\Delta^{(0,\infty)}_{\mathrm{Neu},d_2}\big)  M_{\chi_{[1,\infty)}},
\end{equation*}
and we notice that $\sigma_{R}^{\delta,kl}\big(\LapoN \big)$ is nothing but the usual Bochner--Riesz operator on $\mathbb{R}^d$ restricted to radial functions. 

Now, we focus on the $kk$-part. We need to estimate $\sigma_{\mathrm{on}}^\delta(R)$ and $\sigma_{\mathrm{off}}^\delta(R)$, where 
\begin{align*}
    \sigma_{\mathrm{off}}^\delta(R)(x,y) &= \sum_{1\le j \ne \ell \le 2} \int_0^R \left( 1 - \frac{\lambda^2}{R^2} \right)^\delta \lambda \left[ A(i \lambda) k_\ell(i \lambda |x|) k_j(i \lambda |y|) - A(-i \lambda) k_\ell(-i \lambda |x|) k_j(-i \lambda |y|)    \right] d\lambda\\
    &= \sum_{j\ne \ell} \sum_{\pm} \pm \int_0^R \left( 1 - \frac{\lambda^2}{R^2} \right)^\delta \lambda  A(\pm i \lambda) k_\ell(\pm i \lambda |x|) k_j(\pm i \lambda |y|)  d\lambda\\
    &:= \sum_{j\ne \ell} \sum_{\pm} \sigma_{\mathrm{off}}^{\ell j, \pm}(R)(x,y),
\end{align*}
and
\begin{align*}
    \sigma_{\mathrm{on}}^\delta(R)(x,y) &= \sum_{\ell =1}^2 \int_0^R \left( 1 - \frac{\lambda^2}{R^2} \right)^\delta \lambda \left[ B_\ell(i \lambda) k_\ell(i \lambda |x|) k_\ell(i \lambda |y|) - B_\ell(-i \lambda) k_\ell(-i \lambda |x|) k_\ell(-i \lambda |y|)    \right] d\lambda\\
    &= \sum_{\ell=1}^2 \sum_{\pm} \pm \int_0^R \left( 1 - \frac{\lambda^2}{R^2} \right)^\delta \lambda  B_\ell(\pm i \lambda) k_\ell(\pm i \lambda |x|) k_\ell(\pm i \lambda |y|)  d\lambda\\
    &:= \sum_{\ell=1}^2 \sum_{\pm} \sigma_{\mathrm{on}}^{\ell, \pm}(R)(x,y).
\end{align*}

Since the two sign contributions are estimated in the same way, it is
enough to consider the \(+\) part. Moreover, the coefficient functions
\(A(i\lambda)\), \(B_1(i\lambda)\), and \(B_2(i\lambda)\) have the same
type of asymptotic behaviour:
\[
    \begin{cases}
        |\lambda|^{d_\ell+d_j-4}, & |\lambda|\le 1,\\[1mm]
        |\lambda|^{\frac{d_\ell+d_j}{2}-2}e^{2i\lambda},
        & |\lambda|\ge 1.
    \end{cases}
\]
Thus it suffices to study
\(\sigma_{\mathrm{off}}^{\ell j,+}(R)\).

Theorem~\ref{thm:BRR} then follows from
Propositions~\ref{prop:smallR} and~\ref{prop:largeR} below, together with
Corollary~\ref{cor1}.

\end{proof}

To prove the required uniform \(L^p\)-bounds for
\(\sigma_{\mathrm{off}}^{\ell j,+}(R)\), we treat separately the
low-energy case \(0<R\leq 1\) and the high-energy case \(R>1\), as in
Propositions~\ref{prop:smallR} and~\ref{prop:largeR}, respectively.

\begin{remark}
The decomposition into the two cases $0<R\le 1$ and $R>1$ is not merely technical; it reflects two genuinely different regimes of the spectral parameter $\lambda$. Indeed, the Bochner--Riesz kernel is obtained by integrating over $0\le \lambda\le R$, and the coefficient functions $A$, $B_1$, $B_2$ as well as the special functions $k_\ell(i\lambda|x|)$, $k_j(i\lambda|y|)$ change their behaviour at the threshold $\lambda\sim 1$.

When $0<R\le 1$, the whole integral lies in the low-energy region $\lambda\le 1$. In this regime the coefficients are described only by their small-argument asymptotics, and the main issue is to understand how this low-frequency behaviour interacts with the spatial scales $R|x|$ and $R|y|$. After the rescaling $\lambda\mapsto R\lambda$, the parameter $R$ enters only through these quantities, so the analysis is essentially a low-energy one.

By contrast, when $R>1$, the integration interval $[0,R]$ crosses the threshold $\lambda=1$ and therefore contains both low and high frequencies. In particular, the part $\lambda\gtrsim 1$ is genuinely oscillatory: the large-argument asymptotics of $A$, $B_1$, $B_2$ and of the functions $k_\ell$, $k_j$ have to be used, and one gains cancellation from oscillation, typically after integrating by parts or applying Lemma~\ref{lem:osc-expansion-order-M} near the endpoint. Thus the case $R>1$ is not a rescaled version of the case $R<1$; it involves a different mechanism, namely the coexistence of low-energy singular behaviour and high-frequency oscillation.

For this reason the proof naturally splits into the two propositions corresponding to $0<R\le 1$ and $R>1$.
\end{remark}

\begin{proposition}\label{prop:smallR}
Assume $0<R\le 1$. Then $\sigma_{\mathrm{off}}^{\ell j,+}(R)$ defines a bounded operator on $L^p$, uniformly in $R$, for all $p$ satisfying \eqref{BRrange}.
\end{proposition}

\begin{proof}
In the following argument, we assume \(|x|\ge |y|\) and the case where $|x|\le |y|$ follows by symmetry. 

We consider
\begin{equation*}
\sigma_{\mathrm{off}}^{\ell j,+}(R)(x,y)
=
\int_0^R
\left(1-\frac{\lambda^2}{R^2}\right)^\delta
\lambda A(i\lambda)k_\ell(i\lambda|x|)k_j(i\lambda|y|)\, d\lambda.
\tag{10}
\end{equation*}
Let $\nu_j = d_j/2-1$. Recall that $k_j(r) = r^{-\nu_j}K_{\nu_j}(r)$. Using the low-energy asymptotics of $A$, we have
\begin{equation*}
\sigma_{\mathrm{off}}^{\ell j,+}(R)(x,y)
\sim
|x|^{1-\frac{d_\ell}{2}}|y|^{1-\frac{d_j}{2}}
\int_0^R
\left(1-\frac{\lambda^2}{R^2}\right)^\delta
\lambda^{\frac{d_\ell+d_j}{2}-1}
K_{\nu_\ell}(i\lambda|x|)K_{\nu_j}(i\lambda|y|) d\lambda,
\end{equation*}
and after the change of variables $\lambda\mapsto R\lambda$,
\begin{equation}\label{12}
=
R^{\frac{d_\ell+d_j}{2}}
|x|^{1-\frac{d_\ell}{2}}|y|^{1-\frac{d_j}{2}}
\int_0^1
(1-\lambda^2)^\delta
\lambda^{\frac{d_\ell+d_j}{2}-1}
K_{\nu_\ell}(i\lambda R|x|)K_{\nu_j}(i\lambda R|y|) d\lambda.
\end{equation}
Let $\phi\in C_c^\infty(\R)$ satisfy $0\le \phi\le 1$, 
\begin{align*}
    \phi(t) = \begin{cases}
        1, & t\in [0,\frac{1}{2}],\\
        0, & t\in (-\infty, -\frac{1}{2}] \cup [\frac{3}{4}, \infty).
\end{cases}
\end{align*}
We decompose
\begin{equation*}
    \eqref{12} = \mathcal{I}_{L,+}^R(x,y)+\mathcal{I}_{H,+}^R(x,y),
\end{equation*}
where
\begin{equation*}
    \mathcal{I}_{L,+}^R(x,y)
:=
R^{\frac{d_\ell+d_j}{2}}
|x|^{1-\frac{d_\ell}{2}}|y|^{1-\frac{d_j}{2}}
\int_0^1
\phi(\lambda R|x|)(1-\lambda^2)^\delta
\lambda^{\frac{d_\ell+d_j}{2}-1}
K_{\nu_\ell}(i\lambda R|x|)K_{\nu_j}(i\lambda R|y|) d\lambda,
\end{equation*}
and
\begin{equation*}
\mathcal{I}_{H,+}^R(x,y) :=
R^{\frac{d_\ell+d_j}{2}}
|x|^{1-\frac{d_\ell}{2}}|y|^{1-\frac{d_j}{2}}
\int_0^1
(1-\phi(\lambda R|x|))(1-\lambda^2)^\delta
\lambda^{\frac{d_\ell+d_j}{2}-1}
K_{\nu_\ell}(i\lambda R|x|)K_{\nu_j}(i\lambda R|y|) d\lambda.
\end{equation*}
We first estimate the kernel \(\mathcal{I}_{L,+}^R(x,y)\). Since \(\phi(\lambda R|x|)\) is supported in \(\lambda\lesssim (R|x|)^{-1}\), we may use the small-argument asymptotics of the Bessel functions. This yields
\begin{equation*}
|\mathcal{I}_{L,+}^R(x,y)|
\lesssim
R^{\frac{d_\ell+d_j}{2}}
|x|^{1-\frac{d_\ell}{2}}|y|^{1-\frac{d_j}{2}}
\int_0^{2R^{-1}|x|^{-1}}
\lambda^{\frac{d_\ell+d_j}{2}-1}
(\lambda R|x|)^{1-\frac{d_\ell}{2}}
(\lambda R|y|)^{1-\frac{d_j}{2}} d\lambda
\lesssim
|x|^{-d_\ell}|y|^{2-d_j}.
\end{equation*}
By Lemma~\ref{Nix}, $\mathcal{I}_{L,+}^R$ acts boundedly in $L^p$ for all $1<p<\infty$.

We next decompose $\mathcal{I}_{H,+}^R$ as
\begin{equation*}
    \mathcal{I}_{H,+}^R(x,y)=\mathcal{I}_{+}^R(x,y)+\mathcal{II}_{+}^R(x,y),
\end{equation*}
where
\begin{equation*}
\mathcal{I}_{+}^R(x,y) :=
R^{\frac{d_\ell+d_j}{2}}
|x|^{1-\frac{d_\ell}{2}}|y|^{1-\frac{d_j}{2}}
\int_0^1
(1-\phi(\lambda R|x|))\phi(\lambda R|y|)(1-\lambda^2)^\delta
\lambda^{\frac{d_\ell+d_j}{2}-1}
K_{\nu_\ell}(i\lambda R|x|)K_{\nu_j}(i\lambda R|y|) d\lambda,
\end{equation*}
and
\begin{equation*}
\mathcal{II}_{+}^R(x,y) :=
R^{\frac{d_\ell+d_j}{2}}
|x|^{1-\frac{d_\ell}{2}}|y|^{1-\frac{d_j}{2}}
\int_0^1
(1-\phi(\lambda R|x|))(1-\phi(\lambda R|y|))(1-\lambda^2)^\delta
\lambda^{\frac{d_\ell+d_j}{2}-1}
K_{\nu_\ell}(i\lambda R|x|)K_{\nu_j}(i\lambda R|y|) d\lambda.
\end{equation*}

To continue the argument, we need the following observation.

\begin{lemma}\label{lem:I-smallR}
The kernel $\mathcal{I}_{+}^R(x,y)$ defines a bounded operator on $L^p$, uniformly in $0<R\le 1$, for all 
\begin{equation*}
    \delta>\max\left(D\left|\frac1p-\frac12\right|-\frac12,0\right),
\end{equation*}
where $D=d_1 \vee d_2$.
\end{lemma}

\begin{proof}
Using the large-argument asymptotics in the $x$-variable and the small-argument asymptotics in the $y$-variable, we obtain
\begin{equation*}
    \mathcal{I}_{+}^R(x,y)
\sim
R^{\frac{d_\ell+1}{2}}
|x|^{\frac{1-d_\ell}{2}}|y|^{2-d_j}
\int_0^1
(1-\phi(\lambda R|x|))\phi(\lambda R|y|)(1-\lambda^2)^\delta
\lambda^{\frac{d_\ell-1}{2}}e^{-i\lambda R|x|} d\lambda.
\end{equation*}
We first note that on the support of $\mathcal{I}_{+}^R(x,y)$, we must have $R|x|\gtrsim 1$ or $ 1 - \phi(\lambda R|x|) $ vanishes identically for all $0<\lambda<1$. 

Let $N \in \mathbb{N}$ and $N>\frac{d_\ell+1}{2}$. Fix $\rho \in C_c^\infty(\mathbb{R})$ such that
\begin{equation*}
    \rho(\lambda) = \begin{cases}
        1, & 0\le \lambda \le \frac{1}{8},\\
        0, & \lambda \ge \frac{1}{4}\quad \textrm{or}\quad \lambda \le -\frac{1}{2}.
    \end{cases}
\end{equation*}
We then further split 
\begin{align*}
    \int_0^1
(1-\phi(\lambda R|x|))\phi(\lambda R|y|)(1-\lambda^2)^\delta
\lambda^{\frac{d_\ell-1}{2}}e^{-i\lambda R|x|}\, d\lambda = I_0 + I_1,
\end{align*}
where
\begin{equation*}
    I_0 = \int_0^1 \rho(1-\lambda)
(1-\phi(\lambda R|x|))\phi(\lambda R|y|)(1-\lambda^2)^\delta
\lambda^{\frac{d_\ell-1}{2}}e^{-i\lambda R|x|} d\lambda.
\end{equation*}
We first treat $I_0$. After the change of variables $s=1-\lambda$, one yields
\begin{equation*}
    I_0 = e^{-iR|x|}
\int_0^\infty \psi_{x,y,R}(s)s^\delta e^{isR|x|} ds,
\end{equation*}
where
\begin{equation*}
\psi_{x,y,R}(s) =
\rho(s)(1-\phi((1-s)R|x|))\phi((1-s)R|y|)(2-s)^\delta(1-s)^{\frac{d_\ell-1}{2}}.
\end{equation*}
Note that $\psi_{x,y,R}$ is smooth and is supported in $[-\frac{1}{2},\frac{1}{4}]$.

We next claim that 
\begin{equation}\label{eq:uniform-a-derivative}
    \left\| \psi_{x,y,R}^{(m)} \right\|_\infty \le c_{m,\delta},\quad \forall m\ge 0.
\end{equation}
Indeed, on the support of $\psi_{x,y,R}$, functions $\rho(s), (1-s)^{\frac{d_\ell-1}{2}}, (2-s)^\delta$ are smooth and all of their derivatives are uniformly bounded. If the derivatives hit $1-\phi\left( (1-s) R|x| \right)$, then $\partial_s^m \left( \phi\left( (1-s) R|x| \right) \right) \ne 0$ only if $(1-s)R|x| \in [\frac{1}{2}, \frac{3}{4}]$. Now, since $1-s\in [\frac{3}{4},\frac{3}{2}]$, this forces $R|x|\in [\frac{1}{3}, 1]$ and hence the factor $(R|x|)^m$ is harmless and is uniformly bounded. The same argument applies to
$\phi((1-s)R|y|)$. This proves \eqref{eq:uniform-a-derivative}.

Thus, by Lemma~\ref{lem:osc-expansion-order-M},
\begin{align*}
    I_0 = e^{-iR|x|} \sum_{m=0}^{N-1} c_{m,\delta} \psi_{x,y,R}^{(m)}(0) (R|x|)^{-1-\delta-m} + O\left((R|x|)^{-N}\right).
\end{align*}
Observe that $|\psi_{x,y,R}^{(m)}(0)|$ is uniformly bounded and vanishes if $R|y|\ge 3/4$. We then bound $|I_0|$ by
\begin{equation}\label{A0}
    (R|x|)^{-1-\delta} \mathbf{1}_{R|y|\le \frac{3}{4}} + (R|x|)^{-N}.
\end{equation}
Next, we consider $I_1$. Define
\begin{equation*}
    \widetilde{\psi}_{x,y,R}(\lambda) = \left(1-\rho(1-\lambda)\right) (1-\phi(\lambda R|x|))\phi(\lambda R|y|)(1-\lambda^2)^\delta
\lambda^{\frac{d_\ell-1}{2}}.
\end{equation*}
Therefore
\begin{equation*}
    I_1 = \int_0^1 \widetilde{\psi}_{x,y,R}(\lambda) e^{-i\lambda R|x|} d\lambda.
\end{equation*}
We note that $\widetilde{\psi}_{x,y,R}$ is smooth and supported in $[\frac{1}{2R|x|}, \frac{7}{8}] \subset (0,1)$.

Integrating by parts $N$ times gives
\begin{align*}
    |I_1| &\lesssim (R|x|)^{-N} \int_0^1 \left| \frac{d^N}{d\lambda^N} \widetilde{\psi}_{x,y,R}(\lambda)     \right| d\lambda\\
    &\lesssim (R|x|)^{-N} \int_{(2R|x|)^{-1}}^1 \lambda^{\frac{d_\ell-1}{2} - N} d\lambda\\
    &\lesssim R^{-\frac{d_\ell+1}{2}} |x|^{-\frac{d_\ell+1}{2}}.
\end{align*}
Combining this and \eqref{A0}, we conclude for $|x|\ge |y|$
\begin{align}\label{I_+^R}
    |\mathcal{I}_+^R(x,y)| \lesssim R^{\frac{d_\ell-1}{2}-\delta} |x|^{-\delta - \frac{d_\ell+1}{2}} |y|^{2-d_j} \mathbf{1}_{R|y|\le \frac{3}{4}} + R^{\frac{d_\ell+1}{2}-N} |x|^{-\frac{d_\ell-1}{2}-N} |y|^{2-d_j} + |x|^{-d_\ell} |y|^{2-d_j}.
\end{align}
Since again $R|x|\gtrsim 1$, the last two terms above is $O(|x|^{-d_\ell} |y|^{2-d_j})$. Next, if $\delta > \frac{d_\ell-1}{2}$, then the first term of \eqref{I_+^R} is again bounded by $O(|x|^{-d_\ell} |y|^{2-d_j})$. While if $\delta \le \frac{d_\ell-1}{2}$, we use the fact 
\begin{equation*}
    R|y|\le \frac{3}{4} \implies R^{\frac{d_\ell-1}{2}-\delta} \lesssim |y|^{\delta - \frac{d_\ell-1}{2}}
\end{equation*}
to deduce that the first term of \eqref{I_+^R} is bounded by $|x|^{-\delta-\frac{d_\ell+1}{2}} |y|^{2-d_j+\delta-\frac{d_\ell-1}{2}}$.

It then follows by symmetry, we obtain 
\begin{align*}
    |\mathcal{I}_+^R(x,y)|\lesssim \begin{cases}
        |x|^{2-d_\ell+\delta-\frac{d_j-1}{2}} |y|^{-\delta-\frac{d_j+1}{2}} + |x|^{2-d_\ell}|y|^{-d_j}, & |x|\le |y|,\\
        |x|^{-\delta-\frac{d_\ell+1}{2}} |y|^{2-d_j+\delta-\frac{d_\ell-1}{2}} + |x|^{-d_\ell} |y|^{2-d_j}, & |x|\ge |y|.
    \end{cases}
\end{align*}
By Lemma~\ref{Nix}, $\mathcal{I}_+^R$ is bounded in $L^p$ for
\begin{equation*}
    \frac{d_\ell}{d_\ell \wedge \left(\frac{d_\ell+1}{2}+\delta\right)} < p < \frac{d_j}{0 \vee \left(\frac{d_j-1}{2} - \delta\right)}
\end{equation*}
provided
\begin{equation*}
    p\left( \delta + \frac{d_\ell+1}{2} + d_j - 2 -\delta + \frac{d_\ell-1}{2} - d_j \right) \ge d_\ell - d_j \iff p \ge \frac{d_\ell-d_j}{d_\ell-2}.
\end{equation*}
The proof of Lemma~\ref{lem:I-smallR} is now complete.

\end{proof}

\begin{lemma}\label{lem:II-smallR}
The kernel \(\mathcal{II}_{+}^R(x,y)\) defines a bounded operator on \(L^p\), uniformly in \(0<R\le 1\), for all \(p\) satisfying \eqref{BRrange}.
\end{lemma}

\begin{proof}
Using the large-argument asymptotics for Bessel functions, we obtain
\begin{equation*}
    \mathcal{II}_{+}^R(x,y)
\sim
R^{\frac{d_\ell+d_j}{2}-1}
|x|^{\frac{1-d_\ell}{2}}|y|^{\frac{1-d_j}{2}}
\int_0^1
(1-\phi(\lambda R|x|))(1-\phi(\lambda R|y|))(1-\lambda^2)^\delta
\lambda^{\frac{d_\ell+d_j}{2}-2}
e^{-i\lambda R(|x|+|y|)} d\lambda.
\end{equation*}
It is easy to see that on the support of $\mathcal{II}_+^R$, we must have $R|x|, R|y| \gtrsim 1$ or there is nothing to prove.

Decompose again via function $\rho(1-\lambda)$. We have
\begin{equation*}
    \mathcal{II}_+^R(x,y) \sim R^{\frac{d_\ell+d_j}{2}-1}
|x|^{\frac{1-d_\ell}{2}}|y|^{\frac{1-d_j}{2}} \left( II_0 + II_1 \right),
\end{equation*}
where
\begin{equation*}
    II_0 = \int_0^1 \rho(1-\lambda)
(1-\phi(\lambda R|x|))(1-\phi(\lambda R|y|))(1-\lambda^2)^\delta
\lambda^{\frac{d_\ell+d_j}{2}-2}
e^{-i\lambda R(|x|+|y|)} d\lambda.
\end{equation*}
We treat $II_0$ first. After the change of variables $s=1-\lambda$, it becomes
\begin{equation*}
    II_0 = e^{-iR(|x|+|y|)} \int_0^\infty \varphi_{x,y,R}(s) s^\delta e^{isR(|x|+|y|)} ds,
\end{equation*}
where
\begin{equation*}
    \varphi_{x,y,R}(s)
=
\rho(s) (1-\phi((1-s)R|x|))(1-\phi((1-s)R|y|))(2-s)^\delta(1-s)^{\frac{d_\ell+d_j}{2}-2}.
\end{equation*}
It is clear that $\varphi_{x,y,R}$ is smooth and supported in $[-\frac{1}{2},\frac{1}{4}]$ and by a similar argument as in the claim \eqref{eq:uniform-a-derivative}, we deduce $\left\| \varphi_{x,y,R}^{(m)}  \right\|_\infty \le c_{m,\delta}$ for all $m\ge 0$. Let $N \in \mathbb{N}$ such that $N>\frac{d_\ell+d_j}{2}-1$. Lemma~\ref{lem:osc-expansion-order-M} then gives 
\begin{equation*}
    II_0 = e^{-iR(|x|+|y|)} \sum_{m=0}^{N-1} c_{m,\delta} \varphi_{x,y,R}^{(m)}(0) \left(R(|x|+|y|)\right)^{-1-\delta-m} + O\left(\left(R(|x|+|y|)\right)^{-N}\right).
\end{equation*}
We then crudely bound $|II_0|$ by
\begin{equation}\label{II_0_crude}
    R^{-1-\delta}(|x|+|y|)^{-1-\delta} + R^{-N}(|x|+|y|)^{-N}.
\end{equation}
Next, we analyze $II_1$. Define
\begin{equation*}
    \widetilde{\varphi}_{x,y,R}(\lambda) = \left(1-\rho(1-\lambda)\right)
(1-\phi(\lambda R|x|))(1-\phi(\lambda R|y|))(1-\lambda^2)^\delta
\lambda^{\frac{d_\ell+d_j}{2}-2}.
\end{equation*}
Then 
\begin{equation*}
    II_1 = \int_0^1 \widetilde{\varphi}_{x,y,R}(\lambda) e^{-i\lambda R(|x|+|y|)} d\lambda.
\end{equation*}
Integration by parts $N$ times. We obtain
\begin{align*}
    |II_1|&\le R^{-N} \left( |x|+|y| \right)^{-N} \int_0^1 \left| \frac{d^N}{d\lambda^N}   \widetilde{\varphi}_{x,y,R}(\lambda)    \right| d\lambda\\
    &\lesssim R^{-N} \left( |x|+|y| \right)^{-N} \int_{(2R|y|)^{-1}}^1 \lambda^{\frac{d_\ell+d_j}{2}-2-N} d\lambda\\
    &\lesssim R^{1-\frac{d_\ell+d_j}{2}} |x|^{-N} |y|^{N+1-\frac{d_\ell+d_j}{2}}.
\end{align*}
With \eqref{II_0_crude}, we get for $|x|\ge |y|$
\begin{equation}\label{II_+^Rkernel}
    |\mathcal{II}_+^R(x,y)|\lesssim R^{\frac{d_\ell+d_j}{2}-2-\delta} |x|^{-\delta-\frac{d_\ell+1}{2}} |y|^{-\frac{d_j-1}{2}} + R^{\frac{d_\ell+d_j}{2}-1-N} |x|^{-N-\frac{d_\ell-1}{2}} |y|^{-\frac{d_j-1}{2}} + |x|^{-N-\frac{d_\ell-1}{2}} |y|^{N+1-\frac{d_\ell+2d_j-1}{2}}.
\end{equation}
Using estimates $R^{-\varepsilon}\lesssim |y|^\varepsilon$ and $|x|^{-1} \le |y|^{-1}$, we deduce the last two terms of \eqref{II_+^Rkernel} are $O\left(|x|^{-d_\ell}|y|^{2-d_j}\right)$. Next, for the first term, if $\delta > \frac{d_\ell+d_j}{2}-2$, then it is $O\left( |x|^{-\delta-\frac{d_\ell+1}{2}} |y|^{\delta+2-\frac{d_\ell+2d_j-1}{2}} \right)$. While if $\delta \le \frac{d_\ell+d_j}{2}-2$, it is $O\left( |x|^{-\delta-\frac{d_\ell+1}{2}} |y|^{-\frac{d_j-1}{2}} \right)$. 

Consequently
\begin{align*}
    |\mathcal{II}_+^R(x,y)|\lesssim \begin{cases} 
    |x|^{\delta+2-\frac{d_j+2d_\ell-1}{2}} |y|^{-\delta-\frac{d_j+1}{2}} + |x|^{-\frac{d_\ell-1}{2}} |y|^{-\delta-\frac{d_j+1}{2}} + |x|^{2-d_\ell} |y|^{-d_j}, & |x|\le |y|,\\
    |x|^{-\delta-\frac{d_\ell+1}{2}} |y|^{\delta+2-\frac{d_\ell+2d_j-1}{2}} + |x|^{-\delta-\frac{d_\ell+1}{2}} |y|^{-\frac{d_j-1}{2}} + |x|^{-d_\ell}|y|^{2-d_j}, & |x|\ge |y|.
    \end{cases}
\end{align*}
By Lemma~\ref{Nix}, it is bounded in $L^p$ for 
\begin{equation*}
    \frac{d_\ell}{d_\ell \wedge \left(\frac{d_\ell+1}{2}+\delta\right)} < p < \frac{d_j}{0 \vee \left(\frac{d_j-1}{2} - \delta\right)}
\end{equation*}
provided
\begin{equation*}
    p\left( \delta + \frac{d_\ell+1}{2} + \frac{d_\ell+2d_j-1}{2} - 2 -\delta -d_j\right) \ge d_\ell - d_j \iff p \ge \frac{d_\ell-d_j}{d_\ell-2},
\end{equation*}
and
\begin{equation*}
    p \left( \delta + \frac{d_\ell+1}{2} + \frac{d_j-1}{2} -d_j \right) \ge d_\ell-d_j \iff \delta \ge (d_\ell-d_j) \left( \frac{1}{p} - \frac{1}{2} \right),
\end{equation*}
and so it is bounded for $p$ in \eqref{BRrange}.

Lemma~\ref{lem:II-smallR} then follows immediately. 

\end{proof}

Combining Lemmas~\ref{lem:I-smallR} and \ref{lem:II-smallR}, we conclude that $\sigma_{\mathrm{off}}^{\ell j,+}(R)$ is bounded on $L^p$ (for $p$ in \eqref{BRrange}) uniformly for $0<R\le 1$. This completes the proof of Proposition~\ref{prop:smallR}.
\end{proof}

Next, we consider the case where $R>1$.

\begin{proposition}\label{prop:largeR}
Assume \(R>1\). Then $\sigma_{\mathrm{off}}^{ \ell j,+}(R)$ defines a bounded operator on $L^p$, uniformly in $R$, for all $p$ satisfying \eqref{BRrange}.
\end{proposition}

\begin{proof}
Again, by symmetry, it suffices to assume $|x|\ge |y|$. Recall
\begin{equation*}
\sigma_{\mathrm{off}}^{\ell j,+}(R)(x,y)
=
\int_0^R
\left(1-\frac{\lambda^2}{R^2}\right)^\delta
\lambda A(i\lambda)k_\ell(i\lambda|x|)k_j(i\lambda|y|) d\lambda
\end{equation*}
and decompose it as
\begin{equation*}
=\mathcal{III}_{+}^R(x,y)+\mathcal{IV}_{+}^R(x,y),
\end{equation*}
where
\begin{equation*}
    \mathcal{III}_{+}^R(x,y)
=
\int_0^R
\phi(\lambda)
\left(1-\frac{\lambda^2}{R^2}\right)^\delta
\lambda A(i\lambda)k_\ell(i\lambda|x|)k_j(i\lambda|y|) d\lambda,
\end{equation*}
and
\begin{equation*}
\mathcal{IV}_{+}^R(x,y)
=
\int_0^R
(1-\phi(\lambda))
\left(1-\frac{\lambda^2}{R^2}\right)^\delta
\lambda A(i\lambda)k_\ell(i\lambda|x|)k_j(i\lambda|y|) d\lambda.    
\end{equation*}

\begin{lemma}\label{lem:high-largeR}
The kernel $\mathcal{IV}_{+}^R(x,y)$ defines a bounded operator on $L^p$, uniformly in $R>1$, for all $p$ satisfying \eqref{BRrange}.
\end{lemma}

\begin{proof}
By the support property of $1-\phi$, we may use the large-argument asymptotics of $A$, $k_\ell$, and $k_j$ to obtain
\begin{align*}\nonumber
\mathcal{IV}_{+}^R(x,y)
&\sim
|x|^{-\frac{d_\ell-1}{2}}|y|^{-\frac{d_j-1}{2}}
\int_0^R
(1-\phi(\lambda))
\left(1-\frac{\lambda^2}{R^2}\right)^\delta
e^{-i\lambda(|x|+|y|-2)} d\lambda\\ 
&= R |x|^{-\frac{d_\ell-1}{2}}|y|^{-\frac{d_j-1}{2}}
\int_0^1
(1-\phi(R\lambda))
(1-\lambda^2)^\delta
e^{-iR\lambda(|x|+|y|-2)} d\lambda\\ \nonumber
&:= R |x|^{-\frac{d_\ell-1}{2}}|y|^{-\frac{d_j-1}{2}} \left(IV_0+IV_1\right),
\end{align*}
where
\begin{equation*}
    IV_0 = \int_0^1 \rho(1-\lambda)
(1-\phi(R\lambda))
(1-\lambda^2)^\delta
e^{-iR\lambda(|x|+|y|-2)} d\lambda.
\end{equation*}
For $IV_0$, after the change of variables $\lambda=(1-s)$, this becomes
\begin{equation*}
=e^{-iR(|x|+|y|-2)}
\int_0^\infty
\psi_R(s)s^\delta e^{isR(|x|+|y|-2)} ds,
\end{equation*}
where
\begin{equation*}
    \psi_R(s)=\rho(s)(1-\phi(R(1-s)))(2-s)^\delta.
\end{equation*}
Note that $\psi_R$ is smooth and compactly supported on $[-1,\frac{1}{4}]$. Moreover, $\left\|  \psi_R^{(m)}  \right\|_\infty \le c_{m,\delta}$ for all $m\ge 0$. Indeed, if derivatives hit $1-\phi(R(1-s))$, then $|\partial_s^m(1-\phi(R(1-s)))| = R^m |\phi^{(m)}(R(1-s))| \ne 0$ only if $R(1-s)\in [\frac{1}{2},\frac{3}{4}]$. Since $1-s\in [\frac{3}{4},2]$, this forces $R \lesssim 1$, making the term $R^m$ harmless.

Let $N\in \mathbb{N}$ and $N>\delta+1$. Suppose $|x|+|y|\ge 4$ (so that $R(|x|+|y|-2)\gtrsim 1$ and $|x|+|y|-2 \simeq |x|+|y|$). Applying Lemma~\ref{lem:osc-expansion-order-M}, we obtain 
\begin{align*}
    IV_0 = \sum_{m=0}^{N-1} c_{m,\delta} \psi_R^{(m)}(0) R^{-1-\delta-m} (|x|+|y|-2)^{-1-\delta-m} + O\left( R^{-N} (|x|+|y|-2)^{-N} \right),
\end{align*}
and hence
\begin{equation*}
    |IV_0|\lesssim R^{-1-\delta} (|x|+|y|)^{-1-\delta} + R^{-N} (|x|+|y|)^{-N}. 
\end{equation*}
As for $|x|+|y|<4$, integrating by parts once to yield (note that $s^\delta \psi_R(s)$ vanishes at $0$ and $1$)
\begin{align*}
    |IV_0| \le R^{-1}(|x|+|y|-2)^{-1} \int_0^1 \left|\frac{d}{ds} \left(s^\delta \psi_R(s)\right)\right| ds \lesssim R^{-1}(|x|+|y|-2)^{-1}
\end{align*}
since
\begin{equation*}
    \int_0^1 \left| \frac{d}{ds} \left( \psi_R(s) s^\delta \right) \right| ds \lesssim \int_0^1 s^\delta ds + \int_0^1 s^{\delta-1} ds \lesssim 1.
\end{equation*}
For $IV_1$, define
\begin{equation*}
    \widetilde{\psi}_R(\lambda) = \left(1-\rho(1-\lambda) \right)
(1-\phi(R\lambda))
(1-\lambda^2)^\delta.
\end{equation*}
We have by integration by parts (note that $\widetilde{\psi}_R$ is smooth and compactly supported in $[\frac{1}{2R}, \frac{7}{8}] \subset (0,1)$)
\begin{align*}
    |IV_1| &= \left|  \int_0^1 \widetilde{\psi}_R(\lambda)
e^{-iR\lambda(|x|+|y|-2)} d\lambda  \right|\\
&\le R^{-N}(|x|+|y|-2)^{-N} \int_0^1 \left|\frac{d^N}{d\lambda^N} \widetilde{\psi}_R(\lambda) \right| d\lambda \\
&\lesssim R^{-N}(|x|+|y|-2)^{-N} \int_{(2R)^{-1}}^1 \lambda^{-N} d\lambda\\
&\lesssim R^{-1}(|x|+|y|-2)^{-N}. 
\end{align*}
In particular, if we only integrate by parts once, it becomes
\begin{align*}
    |IV_1| &\lesssim R^{-1} (|x|+|y|-2)^{-1} \int_0^1 |\widetilde{\psi}_R'(\lambda)| d\lambda\\
    &\lesssim R^{-1} (|x|+|y|-2)^{-1} R \int_0^1  \mathbf{1}_{ \frac{1}{2R} \le  \lambda \le \frac{3}{4R}}(\lambda) d\lambda\\
    &\lesssim R^{-1} (|x|+|y|-2)^{-1}.
\end{align*}
Therefore, for $|x|+|y|\ge 4$, 
\begin{align*}
    |\mathcal{IV}_+^R(x,y)| \lesssim R^{-\delta} \frac{|x|^{-\frac{d_\ell-1}{2}}|y|^{-\frac{d_j-1}{2}}}{(|x|+|y|)^{\delta+1}} + \left(1+R^{1-N}\right) \frac{|x|^{-\frac{d_\ell-1}{2}}|y|^{-\frac{d_j-1}{2}}}{(|x|+|y|)^{N}}.
\end{align*}
Since $R>1$ and $N>\delta+1$, it is bounded by the kernel
\begin{equation*}
    K(x,y):=
|x|^{-\frac{d_\ell-1}{2}}|y|^{-\frac{d_j-1}{2}}(|x|+|y|)^{-1-\delta}.
\end{equation*}
While for the case $|x|+|y|<4$, we simply use upper bound
\begin{equation*}
    |\mathcal{IV}_+^R(x,y)| \lesssim R^{-1} (|x|+|y|-2)^{-1}.
\end{equation*}
Consequently, for $p$ in the range \eqref{BRrange},
\begin{align*} 
\|\mathcal{IV}_{+}^R f\|_{L^p([1,\infty),r^{d_\ell-1}dr)}^p
&\lesssim
\int_1^\infty
\left|
\int_1^\infty
x^{-\frac{d_\ell-1}{2}}y^{-\frac{d_j-1}{2}}
\mathbf 1_{\{x+y\le 4\}}
\frac{f(y)y^{d_j-1}}{(x-1)+(y-1)} dy
\right|^p
x^{d_\ell-1} dx
\\ \nonumber
&\quad+
\int_1^\infty
\left|
\int_1^\infty
K(x,y)\mathbf 1_{\{x+y\ge 4\}}f(y)y^{d_j-1}dy
\right|^p
x^{d_\ell-1}dx
\\ \nonumber
&\lesssim
\int_0^\infty
\left|
\int_0^\infty
\frac{|f(s+1)|\mathbf 1_{\{s\le 2\}}}{s+t} ds
\right|^p
dt
\\ \nonumber
&\quad+
\int_1^\infty
\left|
\int_1^\infty
x^{-\frac{d_\ell-1}{2}}y^{-\frac{d_j-1}{2}}(x+y)^{-\delta-1}
f(y)y^{d_j-1} dy
\right|^p
x^{d_\ell-1} dx
\\ \nonumber
&\lesssim
\int_0^\infty |f(s+1)|^p\mathbf 1_{\{s\le 2\}} ds
+
\|f\|_{L^p\left([1,\infty),r^{d_j-1}dr\right)}^p
\\ \nonumber
&\lesssim
\|f\|_{L^p\left([1,\infty),r^{d_j-1}dr\right)}^p,
\end{align*}
where the second inequality follows by changing variables $y \mapsto s+1$, $x\mapsto t+1$, and the second to last inequality is a consequence of the Hardy-Littlewood-Pólya inequality \cite[Theorem~319]{HLP} and Lemma~\ref{Nix}. 

Thus $\mathcal{IV}_{+}^R$ is bounded in $L^p$ uniformly in $R>1$ for $p$ in the desired range \eqref{BRrange}, completing the proof of Lemma~\ref{lem:high-largeR}
\end{proof}

\begin{lemma}\label{lem:low-largeR}
The kernel $\mathcal{III}_{+}^R(x,y)$ defines a bounded operator on $L^p$, uniformly in $R>1$, for all $1<p<\infty$.
\end{lemma}

\begin{proof}
Using the low-energy asymptotics of $A$, we obtain
\begin{equation*}
    \mathcal{III}_{+}^R(x,y)
\sim
|x|^{1-\frac{d_\ell}{2}}|y|^{1-\frac{d_j}{2}}
\int_0^R
\phi(\lambda)
\left(1-\frac{\lambda^2}{R^2}\right)^\delta
\lambda^{\frac{d_\ell+d_j}{2}-1}
K_{\nu_\ell}(i\lambda|x|)K_{\nu_j}(i\lambda|y|) d\lambda.
\end{equation*}
After the change of variables $\lambda\mapsto R\lambda$, we write
\begin{equation*}
\mathcal{III}_{+}^R(x,y)
\sim
R^{\frac{d_\ell+d_j}{2}}
|x|^{1-\frac{d_\ell}{2}}|y|^{1-\frac{d_j}{2}}
\int_0^1
\phi(R\lambda)(1-\lambda^2)^\delta
\lambda^{\frac{d_\ell+d_j}{2}-1}
K_{\nu_\ell}(i\lambda R|x|)K_{\nu_j}(i\lambda R|y|) d\lambda
\end{equation*}
and decompose
\begin{equation*}
\mathcal{III}_{+}^R(x,y)\sim III_0^R(x,y) + III_1^R(x,y)+III_2^R(x,y),
\end{equation*}
where
\begin{align*}
    &III_0^R(x,y):=
R^{\frac{d_\ell+d_j}{2}}
|x|^{1-\frac{d_\ell}{2}}|y|^{1-\frac{d_j}{2}}
\int_0^1
\phi(\lambda R|x|)\phi(R\lambda)(1-\lambda^2)^\delta
\lambda^{\frac{d_\ell+d_j}{2}-1}
K_{\nu_\ell}(i\lambda R|x|)K_{\nu_j}(i\lambda R|y|) d\lambda,\\
&III_1^R(x,y):=\\
&\quad R^{\frac{d_\ell+d_j}{2}}
|x|^{1-\frac{d_\ell}{2}}|y|^{1-\frac{d_j}{2}}
\int_0^1
(1-\phi(\lambda R|x|))\phi(\lambda R|y|)\phi(R\lambda)(1-\lambda^2)^\delta
\lambda^{\frac{d_\ell+d_j}{2}-1}
K_{\nu_\ell}(i\lambda R|x|)K_{\nu_j}(i\lambda R|y|) d\lambda,
\end{align*}
and
\begin{align*}
&III_2^R(x,y):=\\
&\quad R^{\frac{d_\ell+d_j}{2}}
|x|^{1-\frac{d_\ell}{2}}|y|^{1-\frac{d_j}{2}}
\int_0^1
(1-\phi(\lambda R|x|))(1-\phi(\lambda R|y|))\phi(R\lambda)(1-\lambda^2)^\delta
\lambda^{\frac{d_\ell+d_j}{2}-1}
K_{\nu_\ell}(i\lambda R|x|)K_{\nu_j}(i\lambda R|y|) d\lambda.
\end{align*}
As before,
\begin{equation*}
    |III_0^R(x,y)|\lesssim |x|^{-d_\ell}|y|^{2-d_j},
\end{equation*}
which is bounded on $L^p$ for all $1<p<\infty$.

Next, observe that
\begin{equation*}
III_1^R(x,y)
\sim
R^{\frac{d_\ell+1}{2}}
|x|^{\frac{1-d_\ell}{2}}|y|^{2-d_j}
\int_0^1
(1-\phi(\lambda R|x|))\phi(\lambda R|y|)\phi(R\lambda)(1-\lambda^2)^\delta
\lambda^{\frac{d_\ell-1}{2}}e^{-i\lambda R|x|} d\lambda.
\end{equation*}
Let
\begin{equation*}
\Psi_{x,y,R}(\lambda) =
(1-\phi(\lambda R|x|))\phi(\lambda R|y|)\phi(R\lambda)(1-\lambda^2)^\delta \lambda^{\frac{d_\ell-1}{2}}.
\end{equation*}
Integrating by parts $N=\left\lceil \frac{d_\ell+1}{2}\right\rceil$ times gives upper bound (note that $\Psi$ is supported in $\left[\frac{1}{2}(R|x|)^{-1}, \frac{3}{4}(R|y|)^{-1}\right]$, which is an interval contained in $[0,1]$ so $\Psi$ vanishes at endpoints $0$ and $1$) 
\begin{equation}\label{20}
|III_1^R(x,y)|\lesssim R^{\frac{d_\ell+1}{2}-N}
|x|^{\frac{1-d_\ell}{2}-N}
|y|^{2-d_j}
\int_0^1
\left|
\frac{d^N}{d\lambda^N}\Psi_{x,y,R}(\lambda)
\right| d\lambda.
\end{equation}
Note that the integrand above is bounded by $C\lambda^{\frac{d_\ell-1}{2}-N} (1-\lambda)^{\delta-N}$. In particular, $(1-\lambda)^{\delta-N}$ is bounded by $1$ if $\delta \ge N$ and bounded by $\left(1-\frac{3}{4}(R|y|)^{-1}\right)^{\delta-N}\le 4^{N-\delta}$ if $\delta<N$. Thus \eqref{20} is bounded by
\begin{align*}
R^{\frac{d_\ell+1}{2}-N}
|x|^{\frac{1-d_\ell}{2}-N}
|y|^{2-d_j}
\int_{(2R|x|)^{-1}}^1 \lambda^{\frac{d_\ell-1}{2}-N}
 d\lambda \lesssim |x|^{-d_\ell}|y|^{2-d_j}.
\end{align*}
Similarly,
\begin{align*}
III_2^R(x,y) &\sim
R^{\frac{d_\ell+d_j}{2}-1}
|x|^{\frac{1-d_\ell}{2}}|y|^{\frac{1-d_j}{2}}
\int_0^1
(1-\phi(\lambda R|x|))(1-\phi(\lambda R|y|))\phi(R\lambda)(1-\lambda^2)^\delta
\lambda^{\frac{d_\ell+d_j}{2}-2}
e^{-i\lambda R(|x|+|y|)} d\lambda\\
&=:R^{\frac{d_\ell+d_j}{2}-1}
|x|^{\frac{1-d_\ell}{2}}|y|^{\frac{1-d_j}{2}}
\int_0^1
\Phi_{x,y,R}(\lambda)e^{-i\lambda R(|x|+|y|)} d\lambda,
\end{align*}
where apparently $\Phi_{x,y,R}$ is smooth and compactly supported in $\left[\frac{1}{2} (R|x|)^{-1}, \frac{3}{4}R^{-1}\right] \subset [0,1]$.

Integrating by parts $N=\left\lceil \frac{d_\ell+d_j-2}{2}\right\rceil$ times yields 
\begin{equation*}
|III_2^R(x,y)| \lesssim R^{\frac{d_\ell+d_j-2}{2}-N}
|x|^{\frac{1-d_\ell}{2}-N}
|y|^{\frac{1-d_j}{2}}
\int_0^1
\left|
\frac{d^N}{d\lambda^N}\Phi_{x,y,R}(\lambda)
\right| d\lambda.
\end{equation*}
The same estimate as above gives
\begin{equation*}
  |III_2^R(x,y)|\lesssim |x|^{-d_\ell}|y|^{2-d_j}.  
\end{equation*}
Therefore,
\begin{equation*}
    |\mathcal{III}_{+}^R(x,y)|\lesssim |x|^{-d_\ell}|y|^{2-d_j},
\qquad |x|\ge |y|.
\end{equation*}
It follows by Lemma~\ref{Nix} that $\mathcal{III}_{+}^R$ is bounded on $L^p$ for all $1<p<\infty$, uniformly in $R>1$. 

This completes the proof of Lemma~\ref{lem:low-largeR}
\end{proof}

Combining Lemmas~\ref{lem:high-largeR} and \ref{lem:low-largeR} proves Proposition~\ref{prop:largeR}. 
\end{proof}

A parallel argument yields the following Corollary.

\begin{coro}\label{cor1}
For $\ell=1,2$
\begin{equation*}
    \sup_{R>0}\left\| \sigma_{\mathrm{on}}^{\ell,\pm}(R) \right\|_{p\to p}<\infty,
\end{equation*}
provided
\begin{equation*}
\delta>\max\left(d_\ell\left|\frac1p-\frac12\right|-\frac12,0\right).
\end{equation*}

\end{coro}

\part{Counterexamples and sharpness: the Fefferman-type theory}

\section{Counterexample for the spectral projection}\label{SP}

  \begin{theorem}\label{fefferman_nondoubling}
Let $d_1,d_2>2$ and $d_1 \ne d_2$. Then
\begin{align*}
    \sup_{R>0} \left\| E_{\sqrt{\Lapd}}(R) \right\|_{p \to p} \le C \quad \textit{iff} \quad p=2.
\end{align*}

\end{theorem}

\begin{proof}
Thanks to Theorem~\ref{herz_kl}, we again only need to handle the $kk$ part of the kernel. To do so, we further decompose the $kk$ part into the diagonal part (up to complex constants)
\begin{align*}
    E_{\mathrm{on}}(R)(x,y) = \sum_{\ell=1}^2 \sum_{\pm} E_{\mathrm{on}}^{\ell,\pm}(R)(x,y),
\end{align*}
and the off-diagonal part
\begin{align*}
    E_{\mathrm{off}}(R)(x,y) = \sum_{\ell \ne j} \sum_{\pm} E_{\mathrm{off}}^{\ell j, \pm}(R)(x,y).
\end{align*}
The key point is the off-diagonal part. Indeed, we note that 
\begin{itemize}
    \item first, for $\ell=1,2$, $E_{\mathrm{on}}^{\ell,\pm}(R)$ acts as an operator from $L^p([1,\infty),|r|^{d_\ell-1} dr)$ to $L^p([1,\infty),|r|^{d_\ell-1} dr)$;
    \item second, the $kk$ term in its kernel takes the form $k_\ell(\pm i\lambda|x|)k_\ell(\pm i\lambda|y|)$;
    \item third, although the definition of the coefficient $B_1$ (resp. $B_2$) depends on $d_1$ and $d_2$, its asymptotic behaviour depends on $d_1$ (resp. $d_2$) only. 
\end{itemize}
The above three observations allow us to apply Theorem~\ref{herz_doubling} to the on-diagonal part $E_{\mathrm{on}}(R)$, yielding the following lemma.

\begin{lemma}\label{fefferman-on-diag}
Let \(D=d_1\vee d_2\). If
\(\left|1/p-1/2\right|<\frac1{2D}\), then
\[
\sup_{R>0}\|E_{\mathrm{on}}(R)\|_{L^p\to L^p}\leq C.
\]
\end{lemma}

Therefore, to show the unboundedness of $E_{\sqrt{\Lapd}}(R)$ on $L^p$ ($p\ne 2$), one only needs to prove
\begin{equation}\label{claim}
    \| E_{\mathrm{off}}(1)\|_{p \to p} = \infty,\quad \forall p\ne 2.
\end{equation}
In what follows,
we construct a counterexample establishing this assertion.

Note that the kernel of $E_{\mathrm{off}}(1)$ has the form (up to a constant)
\begin{align}\label{off-diag}
    \sum_{\ell \ne j} \left[ E_{\mathrm{off}}^{\ell j,+}(1)(x,y) - E_{\mathrm{off}}^{\ell j,-}(1)(x,y)  \right].
\end{align}
Suppose $|x|\ge |y|$, so that $0<|x|^{-1}\le |y|^{-1} \le 1$. We split the integration in $\lambda$ from \eqref{off-diag} into $\int_0^1 = \int_0^{|x|^{-1}} + \int_{|x|^{-1}}^{|y|^{-1}} + \int_{|y|^{-1}}^1$ (and $\int_0^1 = \int_0^{|y|^{-1}} + \int_{|y|^{-1}}^{|x|^{-1}} + \int_{|x|^{-1}}^1$ if $|x|\le |y|$). Define
\begin{align*}
    I(x,y) = \begin{cases}
        \int_0^{|x|^{-1}} \dots, & |x|\ge |y|,\\
        \int_0^{|y|^{-1}}\dots, & |x|\le |y|,
    \end{cases}\quad
    II(x,y) = \begin{cases}
        \int_{|x|^{-1}}^{|y|^{-1}} \dots, & |x|\ge |y|,\\
    \int_{|y|^{-1}}^{|x|^{-1}}\dots, & |x|\le |y|,
    \end{cases}\quad
    III(x,y) = \begin{cases}
        \int_{|y|^{-1}}^1 \dots, & |x|\ge |y|,\\
    \int_{|x|^{-1}}^1\dots, & |x|\le |y|.
    \end{cases}
\end{align*}
We then claim that $I+II$ acts as a bounded operator on $L^p$ for $p\in \left( \frac{2D}{D+1}, \frac{2D}{D-1} \right)$. The proof is similar to the argument from Theorem~\ref{herz_doubling}. By asymptotic formulas, one has (for $|x|\ge |y|$)
\begin{align*}
     \int_0^{|x|^{-1}} \lambda A(\pm i\lambda) k_\ell(\pm i\lambda|x|) k_j(\pm i\lambda|y|) d\lambda \sim \int_0^{|x|^{-1}} |x|^{2-d_\ell} |y|^{2-d_j} \lambda^{d_1+d_2+1-d_\ell-d_j} d\lambda \simeq |x|^{-d_\ell} |y|^{2-d_j}.
\end{align*}
Thus, $|I| = O(|x|^{-d_\ell} |y|^{2-d_j})$. As for $II$, 
\begin{align*}
    \int_{|x|^{-1}}^{|y|^{-1}} \lambda A(\pm i\lambda) k_\ell(\pm i\lambda|x|) k_j(\pm i\lambda|y|) d\lambda \sim |x|^{-d_\ell} |y|^{2-d_j} \int_1^{|x|/|y|} s^{\frac{d_\ell-1}{2}} e^{\mp i s} ds.
\end{align*}
Integration by parts yields
\begin{equation*}
    \left|\int_1^{|x|/|y|} s^{\frac{d_\ell-1}{2}} e^{\mp i s} ds\right|\lesssim \left( \frac{|x|}{|y|} \right)^{\frac{d_\ell-1}{2}}.
\end{equation*}
Hence, it follows by symmetry that
\begin{align*}
    |(I+II)(x,y)|\lesssim \begin{cases}
        |x|^{-\left( d_\ell+\frac{d_j-1}{2}-2 \right)} |y|^{-\frac{d_j+1}{2}}, & |y|\ge |x|,\\
        |x|^{-\frac{d_\ell+1}{2}} |y|^{-\left( d_j + \frac{d_\ell-1}{2} -2 \right)}, & |y|\le |x|.
    \end{cases}
\end{align*}
By Lemma~\ref{Nix}, $I+II$ is bounded from 
\begin{equation*}
    L^p([1,\infty),r^{d_j-1}dr) \to L^p([1,\infty),r^{d_\ell-1}dr)
\end{equation*}
for all
\begin{equation}
    \frac{d_\ell}{\min \left( d_\ell, \frac{d_\ell+1}{2} \right)} < p < \frac{d_j}{\max\left( 0, d_j - \frac{d_j+1}{2} \right)} \iff \frac{2d_\ell}{d_\ell+1} < p < \frac{2d_j}{d_j-1}.
\end{equation}
Since the stated $D$-dependent range is contained in this interval for each $\ell\neq j$, the claim follows.

As a consequence, it suffices to establish the $L^p$-unboundedness of $III$. Note that for $|x|\ge |y|$,
\begin{align*}
    \int_{|y|^{-1}}^1  \lambda A(\pm i\lambda) k_\ell(\pm i\lambda|x|) k_j(\pm i\lambda|y|) d\lambda \sim |x|^{-\frac{d_\ell-1}{2}} |y|^{-\frac{d_j-1}{2}}\int_{|y|^{-1}}^1 \lambda^{\alpha} e^{\mp i t\lambda} d\lambda := |x|^{-\frac{d_\ell-1}{2}} |y|^{-\frac{d_j-1}{2}} \mathcal{I}_\alpha(x,y),
\end{align*}
where $\alpha = \frac{d_\ell+d_j}{2}-2>0$ and $t = |x|+|y|$. Integration by parts twice gives
\begin{align}\nonumber
    \mathcal{I}_\alpha(x,y) &= \frac{e^{\mp it}}{it} \pm \frac{|y|^{-\alpha} e^{\mp it/|y|}}{it} + \frac{\alpha e^{\mp it}}{t^2} - \frac{\alpha e^{\mp it/|y|} |y|^{1-\alpha}}{t^2}    - \frac{\alpha (\alpha-1)}{t^2} \int_{|y|^{-1}}^1 e^{\mp it \lambda}  \lambda^{\alpha-2} d\lambda \\ \label{I1}
    &= \frac{e^{\mp i(|x|+|y|)}}{i(|x|+|y|)} + O\left(\frac{1}{|y|^{\alpha}(|x|+|y|)}\right) + O\left(\frac{1}{(|x|+|y|)^2}\right) + O\left(\frac{1}{|y|^{\alpha-1}(|x|+|y|)^2}\right).
\end{align}
Symmetrically, for $|x|\le |y|$,
\begin{align}\label{I2}
    \mathcal{I}_\alpha(x,y)= \frac{e^{\mp i(|x|+|y|)}}{i(|x|+|y|)} + O\left(\frac{1}{|x|^{\alpha}(|x|+|y|)}\right) + O\left(\frac{1}{(|x|+|y|)^2}\right) + O\left(\frac{1}{|x|^{\alpha-1}(|x|+|y|)^2}\right).
\end{align}

\textit{Claim~1. The contributions of the last three terms from \eqref{I1} and \eqref{I2} in $III$ define bounded operators
$$L^p([1,\infty),r^{d_j-1}dr) \to L^p([1,\infty),r^{d_\ell-1}dr) $$
for $\frac{2D}{D+1}<p<\frac{2D}{D-1}$.}

To prove \textit{Claim 1}, one only needs to show the $L^p$-boundedness of $S_1, S_2, S_3$, where
\begin{align*}
    &S_1(x,y) = \begin{cases}
        |x|^{-\frac{d_\ell-1}{2}-\alpha}|y|^{-\frac{d_j-1}{2}}(|x|+|y|)^{-1}, & |y|\ge |x|,\\
        |x|^{-\frac{d_\ell-1}{2}} |y|^{-\frac{d_j-1}{2}-\alpha} (|x|+|y|)^{-1}, & |y|\le |x|,
    \end{cases}  \\
    &S_2(x,y) = |x|^{-\frac{d_\ell-1}{2}}|y|^{-\frac{d_j-1}{2}} (|x|+|y|)^{-2},\\
    &S_3(x,y) = \begin{cases}
        |x|^{-\frac{d_\ell-1}{2}+1-\alpha}  |y|^{-\frac{d_j-1}{2} }(|x|+|y|)^{-2}, & |y|\ge |x|,\\
         |x|^{-\frac{d_\ell-1}{2}} |y|^{-\frac{d_j-1}{2} +1-\alpha }(|x|+|y|)^{-2}, & |y|\le |x|.
    \end{cases}
\end{align*}

\begin{proof}[Proof of Claim~1]
One checks easily that 
\begin{align*}
&|S_1(x,y)| \lesssim \begin{cases}
        |x|^{-\frac{d_\ell-1}{2}-\alpha} |y|^{-\frac{d_j+1}{2}}, & |y| \ge |x|,\\
        |x|^{-\frac{d_\ell+1}{2}} |y|^{-\frac{d_j-1}{2}-\alpha}, & |y|\le |x|,
    \end{cases} \quad     |S_2(x,y)| \lesssim \begin{cases}
        |x|^{-\frac{d_\ell-1}{2}} |y|^{-\frac{d_j+3}{2}}, & |y| \ge |x|,\\
        |x|^{-\frac{d_\ell+3}{2}} |y|^{-\frac{d_j-1}{2}}, & |y|\le |x|,
    \end{cases}\\
    &|S_3(x,y)| \lesssim \begin{cases}
        |x|^{-\frac{d_\ell-1}{2}+1-\alpha} |y|^{-\frac{d_j+3}{2}}, & |y| \ge |x|,\\
        |x|^{-\frac{d_\ell+3}{2}} |y|^{-\frac{d_j-1}{2}+1-\alpha}, & |y|\le |x|.
    \end{cases}
\end{align*}
Note that the  operators $S_1,S_2,S_3$ are all of homogeneous type. Moreover, for all $p\ge 1$, we have
\begin{align*}
    p \left(\frac{d_\ell-1}{2} + \alpha + \frac{d_j+1}{2} - d_j \right) = p (d_\ell-2) \ge d_\ell-d_j.
\end{align*}
Therefore, by Lemma~\ref{Nix}, $S_1$ is bounded in $L^p$ for all
\begin{align}\label{S1}
    \frac{2d_\ell}{2d_\ell \wedge (d_\ell+1)} < p < \frac{2d_j}{0 \vee (d_j-1)} \iff \frac{2d_\ell}{d_\ell+1} < p < \frac{2d_j}{d_j-1}.
\end{align}
Next, to apply Lemma~\ref{Nix} to $S_2$, one needs
\begin{equation*}
    p\left( \frac{d_\ell-1}{2} + \frac{d_j+3}{2} - d_j \right) \ge d_\ell-d_j,
\end{equation*}
which is satisfied if
\begin{align*}
    \begin{cases}
        1\le p \le \infty, & d_\ell-d_j=-2, \\
        p\ge \frac{2(d_\ell-d_j)}{d_\ell-d_j+2}, & d_\ell-d_j>-2,\\
        p \le \frac{2(d_j-d_\ell)}{d_j-d_\ell-2}, & d_\ell-d_j<-2.
    \end{cases}
\end{align*}
A straightforward calculation shows that the above assumption holds if $\frac{2d_\ell}{d_\ell+1}<p<\frac{2d_j}{d_j-1}$. So, Lemma~\ref{Nix} guarantees that $S_2$ is $L^p$-bounded for
\begin{equation}\label{S2}
    \frac{2d_\ell}{2d_\ell \wedge (d_\ell+3)} <p< \frac{2d_j}{0 \vee (d_j-3)}.
\end{equation}
Finally, since for all $p\ge 1$, we have
\begin{equation*}
    p \left( \frac{d_\ell-1}{2} - 1 + \frac{d_\ell+d_j}{2} - 2 + \frac{d_j+3}{2} - d_j \right) \ge d_\ell-d_j \iff p \ge \frac{d_\ell-d_j}{d_\ell-2}.
\end{equation*}
Thus, by Lemma~\ref{Nix} again, $S_3$ is bounded in $L^p$ provided 
\begin{equation}\label{S3}
    \frac{2d_\ell}{2d_\ell \wedge (d_\ell+3)} < p < \frac{2d_j}{0 \vee (d_j-3)}.
\end{equation}
Note that $\left( \frac{2D}{D+1}, \frac{2D}{D-1} \right)$ is contained in each of the intervals in \eqref{S1}, \eqref{S2} and \eqref{S3}. The proof of Claim~1 is then complete.
\end{proof}

It remains to show that an operator with kernel $\frac{e^{\mp i(|x|+|y|)}}{|x|+|y|}$ can be bounded on $L^p$ only when $p=2$. That is, we have to establish the $L^p$-unboundedness of the linear combination of the following two operators:
\begin{equation*}
S_4f(x) := \sum_{\ell \ne j}|x|^{-\frac{d_\ell-1}{2}} \int_1^\infty |y|^{-\frac{d_j-1}{2}} \frac{e^{i(|x|+|y|)}f(y)}{|x|+|y|} |y|^{d_j-1} dy:= \sum_{\ell \ne j}S_4^{\ell j}f(x), 
\end{equation*}
and
\begin{equation*}
S_5f(x):= \sum_{\ell \ne j} |x|^{-\frac{d_\ell-1}{2}} \int_1^\infty |y|^{-\frac{d_j-1}{2}} \frac{e^{-i(|x|+|y|)}f(y)}{|x|+|y|} |y|^{d_j-1} dy:= \sum_{\ell \ne j}S_5^{\ell j}f(x)
\end{equation*}
for $p \ne 2$.

Without loss of generality, assume that \(d_1<d_2\).
Since the spectral projection is self-adjoint, it  suffices to
consider $p>2$.  A direct estimate shows  that the kernel of $S_4^{21}+S_{5}^{21}$
satisfies 
\begin{align*}
    |(S_4^{21}+S_{5}^{21})(x,y)| \lesssim \begin{cases}
        |x|^{-\frac{d_2-1}{2}} |y|^{-\frac{d_1+1}{2}}, & |y|\ge |x|,\\
        |x|^{-\frac{d_2+1}{2}} |y|^{-\frac{d_1-1}{2}}, & |y|\le |x|.
    \end{cases}
\end{align*}
Since 
\begin{equation*}
    p\left( \frac{d_2-1+d_1+1}{2} - d_1 \right) \ge d_2-d_1 \implies p\ge 2,
\end{equation*}
Lemma~\ref{Nix} tells us $S_4^{21}+S_{5}^{21}$ is bounded from $L^p([1,\infty), r^{d_1-1}dr) \to L^p([1,\infty), r^{d_2-1}dr)$ for 
\begin{equation*}
    2\le  p < \frac{2D}{D-1}.
\end{equation*}
Hence, the result of Theorem~\ref{fefferman_nondoubling} follows by showing the unboundedness of the linear combination of $S_4^{12}$ and $S_{5}^{12}$, say $c_1 S_4^{12} + c_2 S_5^{12}$ with $c_1,c_2 \in \mathbb C$ not simultaneously zero, from $L^p([1,\infty), r^{d_2-1}dr) \to L^p([1,\infty), r^{d_1-1}dr)$ for $p>2$ and $d_1<d_2$. 

Assume $c_1\ne 0$. Consider the function 
\begin{equation}\label{couterF}
    f_\varepsilon(y) = e^{-iy} \mathbf{1}_{[\varepsilon,2\varepsilon]}(y) y^{-\frac{d_2-1}{2}}, \quad \varepsilon\ge 1,
\end{equation}
and if $c_1=0$ (so $c_2\ne 0$), we put $f_\varepsilon(y) = e^{iy} \mathbf{1}_{[\varepsilon,2\varepsilon]}(y) y^{-\frac{d_2-1}{2}}$.

First,
\begin{align*}
    \| f_\varepsilon \|_p^p \lesssim \int_\varepsilon^{2\varepsilon} y^{-p \frac{d_2-1}{2}} y^{d_2-1} dy \lesssim \varepsilon^{d_2 - p \frac{d_2-1}{2}}.
\end{align*}
Second, for $x\in [2\varepsilon,3\varepsilon]$,
\begin{align*}
    |S_4^{12}f_{\varepsilon}(x)| = x^{-\frac{d_1-1}{2}} \int_\varepsilon^{2\varepsilon} \frac{dy}{x+y} \gtrsim \varepsilon^{-\frac{d_1-1}{2}}.
\end{align*}
Third, for $x\in [2\varepsilon,3\varepsilon]$, one also has
\begin{align*}
    |S_5^{12} f_\varepsilon(x)| &= x^{-\frac{d_1-1}{2}} \left| \int_\varepsilon^{2\varepsilon} \frac{e^{-2iy}dy}{x+y} \right| \\
    &= x^{-\frac{d_1-1}{2}} \left| -\frac{1}{2i} \int_\varepsilon^{2\varepsilon} \frac{d (e^{-2iy}) }{x+y} \right|\\
    &= \frac{1}{2} x^{-\frac{d_1-1}{2}} \left| \left[\frac{e^{-2iy}}{x+y}\right]_{y=\varepsilon}^{2\varepsilon} +   \int_\varepsilon^{2\varepsilon} \frac{e^{-2iy}}{(x+y)^2} dy  \right|\\
    &\lesssim \varepsilon^{-\frac{d_1+1}{2}}.
\end{align*} 
Hence, for sufficiently large $\varepsilon$ and  for $x\in [2\varepsilon,3\varepsilon]$ we have 
\begin{equation*}
    |c_1 S_4^{12}f_\varepsilon(x) + c_2 S_5^{12}f_\varepsilon(x)| \gtrsim |S_4^{12}f_\varepsilon(x)| \gtrsim \varepsilon^{-\frac{d_1-1}{2}}.
\end{equation*}
Thus,
\begin{align*}
 \| c_1 S_4^{12} f_\varepsilon + c_2S_5^{12} f_\varepsilon\|_{L^p([2\varepsilon,3\varepsilon], x^{d_1-1} dx)}^p \gtrsim \| c_1 S_4^{12} f_\varepsilon + c_2 S_5^{12} f_\varepsilon \|_{L^p([2\varepsilon,3\varepsilon], x^{d_1-1}dx)}^p \\ \gtrsim \int_{2\varepsilon}^{3\varepsilon} \varepsilon^{-p \frac{d_1-1}{2}} x^{d_1-1} dx \gtrsim \varepsilon^{d_1-p\frac{d_1-1}{2}}.
\end{align*}
Therefore, if $S_4+S_5$ is bounded in $L^p$ for some $p>2$, then we have for $\varepsilon  \gg 1$,
\begin{equation*}
    \varepsilon^{d_1-p\frac{d_1-1}{2}} \lesssim \varepsilon^{d_2 - p \frac{d_2-1}{2}} \iff \varepsilon^{(d_1-d_2)(1-p/2)} \lesssim 1,
\end{equation*}
which is not true as $\varepsilon\to \infty$. The proof of Theorem~\ref{fefferman_nondoubling} is now complete.
\end{proof}

\begin{remark}\label{r53}
The mechanism behind Theorem~\ref{fefferman_nondoubling} is closely
related, at least at the level of ideas, to Fefferman's counterexample
for the ball multiplier \cite{Fefferman}. In Fefferman's argument,
Kakeya geometry permits the construction of wave packets pointing in
many directions whose spatial supports exhibit highly non-uniform
overlap, leading to an incompatibility with \(L^p\)-boundedness when
\(p\ne2\). Although the present one-dimensional model has no directional
geometry, the mismatch between the volume-growth dimensions \(d_1\)
and \(d_2\) produces an analogous obstruction. Indeed, the test function
is placed on the larger end, whose measure on an interval of length
comparable to \(\varepsilon\) scales like \(\varepsilon^{d_2}\), whereas
the leading output is measured on the smaller end and has the scaling
dictated by \(d_1\). The contradiction arises from the incompatibility
between the \(L^p\)-scaling of the input on the \(d_2\)-dimensional end
and that of the leading output on the \(d_1\)-dimensional end. In this
sense, the dimensional asymmetry of the two ends reproduces, in a
one-dimensional setting, the testing mechanism behind Fefferman's
Kakeya-set argument.
\end{remark}

\begin{remark}\label{r54}
The counterexample used above will reappear, with minor variations, in
the negative results below. Its role is to isolate the principal
off-diagonal oscillatory interaction between the two ends. Analytically,
this interaction arises from the region of the off-diagonal \(kk\)
kernel in which both \(\lambda|x|\) and \(\lambda|y|\) are large, and
produces a model kernel of the form
\begin{equation*}
    \frac{|x|^{-\frac{d_\ell-1}{2}}
    |y|^{-\frac{d_j-1}{2}}}{|x|+|y|}
    e^{\mp i(|x|+|y|)}.
\end{equation*}
Geometrically, the test function is placed on one end and the output is
measured on the other; the off-diagonal interaction therefore couples
data on the two ends through their junction. In the asymmetric case,
this produces the Fefferman-type obstruction arising from the mismatch
between \(d_1\) and \(d_2\), while in the symmetric case it reflects the
obstruction at the boundary of the Herz range. Thus the variants of this
test function used below are dictated by the basic geometric and
oscillatory structure of the model.
\end{remark}

We now complete the proof of Theorem~\ref{main_spectral_projection} for the
symmetric case \(d_1=d_2=d\). The positive result was established in
Theorem~\ref{herz_doubling}. However, those estimates prove only sufficiency
and do not yield the sharp necessity statement outside Herz's range.

To establish necessity, we isolate the off-diagonal \(kk\) contribution and
treat it independently of the on-diagonal \(kl\) and \(kk\) terms. This
reduces the problem to the same oscillatory mechanism that appeared in
Section~\ref{sec4}. In what follows, we retain the cut-off functions and the
small- and large-frequency decomposition introduced there.

\begin{lemma}\label{lem:negative:spectral}
Let \(d>2\) and \(1\le p\le\infty\). The spectral projection
\(E_{\sqrt{\Delta_d}}([0,1))\) is unbounded on \(L^p(\widetilde{\R})\) whenever
\[
    \left|\frac{1}{p}-\frac{1}{2}\right|
    \ge \frac{1}{2d}.
\]
\end{lemma}

\begin{proof}
If \(E_{\sqrt{\Delta_d}}([0,1))\) were bounded on \(L^p\), then so would be
its off-diagonal localisation
$
T:=M_{\chi_{[1,\infty)}}E_{\sqrt{\Delta_d}}([0,1))
M_{\chi_{(-\infty,-1]}}.
$
It therefore suffices to prove that \(T\) is unbounded.
Up to harmless constants, the kernel of $T$ is a linear combination of 
\begin{equation*}
    \int_0^1 \lambda A(\pm i \lambda) k_d(\pm i \lambda |x|) k_d (\pm i \lambda |y|) d\lambda.
\end{equation*}
Recall the functions \(\phi,\rho\in C_c^\infty(\mathbb R)\) used in
Section~\ref{sec4}, chosen so that
%Recall the cut-off functions used in Section~\ref{sec4}:
\begin{align*}
    \phi(t) = \begin{cases}
        1, & 0\le t\le \frac{1}{2},\\
        0, & t\le -\frac{1}{2}\quad \textrm{or} \quad t\ge \frac{3}{4}.
    \end{cases} \qquad \rho(t) = \begin{cases}
        1, & 0\le t \le \frac{1}{8},\\
        0, & t\le -\frac{1}{2}\quad \textrm{or} \quad t\ge \frac{1}{4}.
    \end{cases}
\end{align*}
Suppose $|x|\ge |y|$. Following the idea of Proposition~\ref{prop:smallR}, we split the above kernel into $\mathcal{I}_{L,\pm}+\mathcal{I}_{H,\pm}$, where 
\begin{equation*}
\mathcal{I}_{L,\pm}(x,y) := \int_0^1 \phi(\lambda |x|) \lambda A(\pm i \lambda) k_d(\pm i \lambda |x|) k_d (\pm i \lambda |y|) d\lambda.
\end{equation*}
%A routine small-argument asymptotic estimate gives 
The small-argument asymptotics give
$|\mathcal{I}_{L,\pm}(x,y)|\lesssim |x|^{-d} |y|^{2-d}$ for $|x|\ge |y|$. We further decompose $\mathcal{I}_{H,\pm}$ into $\mathcal{I}_{\pm}+\mathcal{II}_{\pm}$, where
\begin{align*}
    \mathcal{I}_{\pm}(x,y):&= \int_0^1 \left(1-\phi(\lambda |x|) \right) \phi(\lambda |y|) \lambda A(\pm i \lambda) k_d(\pm i \lambda |x|) k_d (\pm i \lambda  |y|) d\lambda\\
    & \sim |x|^{\frac{1-d}{2}} |y|^{2-d} \int_0^1 \left(1-\phi(\lambda |x|) \right) \phi(\lambda |y|) \lambda^{\frac{d-1}{2}} e^{\pm i \lambda |x|}\, d\lambda. 
\end{align*}
Observe that the amplitude of the above integral is smooth and compactly supported in $[ \frac{|x|^{-1}}{2}, \frac{3|y|^{-1}}{4} ] \subset (0,1)$. Integrating by parts $N=\left\lceil \frac{d+1}{2} \right\rceil$ times, we obtain
\begin{align*}
    |\mathcal{I}_{\pm}(x,y)| \lesssim |x|^{-\frac{d-1}{2}-N} |y|^{2-d} \int_{2^{-1}|x|^{-1}}^1 \lambda^{\frac{d-1}{2}-N} d\lambda \lesssim |x|^{-d} |y|^{2-d}.
\end{align*}
We now estimate $\mathcal{II}_{\pm}$. After inserting the function $\rho(1-\lambda)$, the proof of Lemma~\ref{lem:II-smallR} shows that the integral 
\begin{equation*}
\int_0^1 (1-\rho(1-\lambda)) (1-\phi(\lambda |x|) ) (1-\phi(\lambda |y|) ) \lambda A(\pm i \lambda) k_d(\pm i \lambda  |x|) k_d (\pm i \lambda  |y|) d\lambda
\end{equation*}
is at most $O\left( |x|^{-N-\frac{d-1}{2}} |y|^{N+1-d-\frac{d-1}{2}} \right) = O(|x|^{-d} |y|^{2-d})$. It remains to treat
\begin{align*}
    II_{0,\pm} :&= \int_0^1 \rho(1-\lambda) (1-\phi(\lambda |x|)) (1-\phi(\lambda |y|)) \lambda A(\pm i \lambda) k_d(\pm i \lambda  |x|) k_d (\pm i \lambda |y|) d\lambda\\
    &\sim |x|^{-\frac{d-1}{2}} |y|^{-\frac{d-1}{2}} \int_0^1 \rho(1-\lambda)
(1-\phi(\lambda |x|))(1-\phi(\lambda |y|))
\lambda^{d-2}
e^{\mp i\lambda (|x|+|y|)} d\lambda.
\end{align*}
Making the change of variables $s=1-\lambda$, 
\begin{align*}
    II_{0,\pm} =|x|^{-\frac{d-1}{2}} |y|^{-\frac{d-1}{2}}  e^{\mp i (|x|+|y|)} \int_0^\infty \rho(s) (1-s)^{d-2} e^{\pm is(|x|+|y|)} ds.
\end{align*}
Here we  used the fact that, on \(\operatorname{supp}\rho\), one has \(s\le1/4\), and hence
\(\lambda=1-s\ge3/4\). Since \(|x|,|y|\ge1\), it follows that
$(1-\phi(\lambda |x|)) = (1-\phi(\lambda |y|))=1$.

%Integrating by parts twice in the above integral, we get
Integrating the last integral by parts twice
\begin{equation*}
    II_{0,\pm} = |x|^{-\frac{d-1}{2}} |y|^{-\frac{d-1}{2}} \frac{e^{\mp i (|x|+|y|)}}{|x|+|y|} + O\left( |x|^{-\frac{d+3}{2}} |y|^{-\frac{d-1}{2}} \right).
\end{equation*}
By symmetry, the kernel of $T$ is given by a linear combination of 
\begin{equation*}
     \frac{|x|^{-\frac{d-1}{2}} |y|^{-\frac{d-1}{2}}}{|x|+|y|}e^{\mp i (|x|+|y|)} + \begin{cases}
        |x|^{-\frac{d+3}{2}} |y|^{-\frac{d-1}{2}}, & |x|\ge |y|,\\
        |x|^{-\frac{d-1}{2}} |y|^{-\frac{d+3}{2}}, &|x|\le |y|.
    \end{cases} + \begin{cases}
        |x|^{-d} |y|^{2-d}, & |x|\ge |y|,\\
        |x|^{2-d} |y|^{-d}, &|x|\le |y|.
    \end{cases}
\end{equation*}
By Lemma~\ref{Nix}, the last two terms are bounded on $L^p$ for 
\begin{equation*}
    \frac{d}{d \wedge \frac{d+3}{2}} < p<\frac{d}{0\vee \frac{d-3}{2}} \iff \begin{cases}
        1<p<\infty, & 2<d\le 3,\\
        \frac{2d}{d+3}<p<\frac{2d}{d-3}, & d>3.
    \end{cases}
\end{equation*}
Thus, to prove Lemma~\ref{lem:negative:spectral}, it remains to show the $L^p([1,\infty),r^{d-1}dr)$ unboundedness of a linear combination of the operators
\begin{equation*}
    S_{\pm}: f \mapsto \int_1^\infty \frac{x^{-\frac{d-1}{2}} y^{-\frac{d-1}{2}}}{x+y}e^{\mp i (x+y)} f(y) y^{d-1} dy
\end{equation*}
for $\frac{2d}{d-1}\le p < \frac{2d}{d-3}$ and $d>3$. 
%The case $2<d\le 3$ is similar, and we omit the details.
When $2<d\le 3$ , the condition \(p<2d/(d-3)\) is absent, and the same
argument applies for every  \(p\ge2d/(d-1)\).

Consider the function $f_\varepsilon(y) = e^{iy}y^{-d/p} \mathbf{1}_{(2,\varepsilon)}(y)$ with $\varepsilon>100$ large. If $x\in [1,2]$, then
\begin{align*}
    \left|S_+ f_\varepsilon(x)\right| &= x^{-\frac{d-1}{2}} \left| \int_2^\varepsilon \frac{y^{\frac{d-1}{2}-\frac{d}{p}}}{x+y} dy \right|\\
    &\gtrsim \begin{cases}
        \varepsilon^{\frac{d-1}{2}-\frac{d}{p}}, & \frac{2d}{d-1}<p<\frac{2d}{d-3},\\
        \log{\varepsilon}, & p=\frac{2d}{d-1}.
    \end{cases}
\end{align*}
Moreover,
\begin{align*}
    |S_-f_\varepsilon(x)| &= x^{-\frac{d-1}{2}} \left| \int_2^\varepsilon \frac{y^{\frac{d-1}{2}-\frac{d}{p}}}{x+y} e^{2iy} dy \right|\\
    &= \frac{1}{2} x^{-\frac{d-1}{2}} \left| \frac{y^{\frac{d-1}{2}-\frac{d}{p}}e^{2iy}}{x+y}\bigg|_{y=2}^\varepsilon -  \int_2^\varepsilon e^{2iy} \frac{d}{dy}\left(\frac{y^{\frac{d-1}{2}-\frac{d}{p}}}{x+y}\right) dy \right|\\
    &\lesssim \varepsilon^{\frac{d-1}{2}-\frac{d}{p}-1} + 1\\
    &\lesssim 1,
\end{align*}
since $p<\frac{2d}{d-3}$. In addition, one has $\|f_\varepsilon\|_p\lesssim (\log{\varepsilon})^{\frac 1p}$. Write the linear combination of $S_\pm$ as $\sum_{\pm}c_\pm S_\pm$ with $c_\pm \in \mathbb C$ and $c_-,c_+$, not both zero. 
Assume first that \(c_+\ne0\). If \(c_+=0\), then \(c_-\ne0\), and the
same argument applies after replacing \(e^{iy}\) by \(e^{-iy}\).
Then the $L^p$ boundedness of $\sum_{\pm}c_\pm S_\pm$ would imply
\begin{align*}
\begin{cases}
        \varepsilon^{\frac{d-1}{2}-\frac{d}{p}}, & \frac{2d}{d-1}<p<\frac{2d}{d-3},\\
        \log{\varepsilon}, & p=\frac{2d}{d-1}.
    \end{cases} \lesssim \left\| \sum_{\pm} c_\pm S_\pm f_\varepsilon \right\|_p \lesssim \|f_\varepsilon\|_p \lesssim (\log{\varepsilon})^{\frac{1}{p}},
\end{align*}
which is impossible as $\varepsilon \to \infty$.

This proves Lemma~\ref{lem:negative:spectral} and completes the proof of Theorem~\ref{main_spectral_projection}.
\end{proof}

\section{Counterexample for the Bochner--Riesz means}\label{BR}

In this section, we prove the necessity part of Theorem~\ref{main_BR}, thereby showing
that the sufficient conditions established in Theorem~\ref{thm:BRR} are sharp.

\begin{proposition}\label{propBRR}
Let $1\le p\le \infty$. The Bochner--Riesz mean $\sigma_{1}^\delta(\Lapd)$ is unbounded in $L^p(\widetilde{\R})$ if
$$
\delta < |d_1-d_2| \left| \frac{1}{p} - \frac{1}{2} \right|,
$$
or
$$
0< \delta \le D \left|\frac{1}{p}-\frac{1}{2}\right|-\frac{1}{2},
$$
where $D=d_1 \vee d_2$.
\end{proposition}

\begin{proof}
It follows by complex interpolation that we only need to show $\sigma_{1}^\delta(\Lapd)$ is unbounded in $L^p$ in the following two cases: 
\begin{enumerate}[label=(\Alph*)]
    \item \begin{align*}
        \delta > \max \left( D\left|\frac{1}{p}-\frac{1}{2}\right|-\frac{1}{2}, 0 \right) \quad \textrm{and} \quad \delta < |d_1-d_2| \left|\frac{1}{p}-\frac{1}{2}\right|,
        \end{align*}
        \item \begin{align*}
            \delta \ge |d_1-d_2| \left|\frac{1}{p}-\frac{1}{2}\right|, \quad \textrm{and}\quad 0<\delta \le D\left|\frac{1}{p}-\frac{1}{2}\right|-\frac{1}{2}.
        \end{align*}
\end{enumerate}
The result now follows directly from Lemmas~\ref{Counter1}
and~\ref{Counter2}.
\end{proof}

\begin{remark}
The two unboundedness regions in  Proposition~\ref{propBRR} are illustrated below.

\begin{center}\label{fig2}
% Parameters
\def\D{5}   % D = d_1 \vee d_2
\def\m{2}   % |d_1-d_2| = D-d

\pgfmathsetmacro{\dmin}{\D-\m}                  % d = d_1 \wedge d_2
\pgfmathsetmacro{\ytop}{(\D-1)/2}
\pgfmathsetmacro{\yredaxis}{\m/2}               % = (D-d)/2
\pgfmathsetmacro{\xblueL}{0.5-1/(2*\D)}
\pgfmathsetmacro{\xblueR}{0.5+1/(2*\D)}
\pgfmathsetmacro{\xintL}{0.5-1/(2*\dmin)}
\pgfmathsetmacro{\xintR}{0.5+1/(2*\dmin)}
\pgfmathsetmacro{\yint}{\m/(2*\dmin)}           % = (D-d)/(2d)

\begin{tikzpicture}[
    x=10.8cm,
    y=2.8cm,
    line cap=round,
    line join=round,
    font=\small
]

% axes
\draw[->, very thick] (0,0) -- (1.04,0);
\draw[->, very thick] (0,0) -- (0,\ytop+0.30);

% shaded region (1)
\fill[gray!75]
  (\xintL,\yint) -- (0.5,0) -- (\xintR,\yint) -- (\xblueR,0) -- (\xblueL,0) -- cycle;

% shaded region (2)
\fill[gray!75]
  (0,\ytop) -- (0,\yredaxis) -- (\xintL,\yint) -- cycle;
\fill[gray!75]
  (1,\ytop) -- (1,\yredaxis) -- (\xintR,\yint) -- cycle;

% blue normal BR line
\draw[blue, line width=1.5pt]
  (0,\ytop) -- (\xblueL,0) -- (\xblueR,0) -- (1,\ytop);

% red new restriction line
\draw[red, line width=1.3pt]
  (0,\yredaxis) -- (0.5,0) -- (1,\yredaxis);

% key points
\fill[blue] (\xblueL,0) circle (1.6pt);
\fill[blue] (\xblueR,0) circle (1.6pt);
\fill[red]  (0,\yredaxis) circle (1.6pt);
\fill       (0.5,0) circle (1.6pt);
\fill       (\xintL,\yint) circle (1.6pt);
\fill       (\xintR,\yint) circle (1.6pt);

% dashed guides
\draw[dashed] (\xintL,\yint) -- (\xintL,0);
\draw[dashed] (\xintL,\yint) -- (0,\yint);
\draw[dashed] (\xintR,\yint) -- (\xintR,0);

% x-axis ticks
\draw (0,0) -- +(0,-0.020);
\draw (\xblueL,0) -- +(0,-0.018);
\draw (0.5,0) -- +(0,-0.020);
\draw (\xblueR,0) -- +(0,-0.018);
\draw (\xintL,0) -- +(0,-0.018);
\draw (\xintR,0) -- +(0,-0.018);
\draw (1,0) -- +(0,-0.020) node[below=4pt] {$1$};

% y-axis ticks
\draw (0,\ytop) -- +(-0.010,0) node[left=6pt] {$\dfrac{D-1}{2}$};
\draw (0,\yredaxis) -- +(-0.010,0);
\draw (0,\yint) -- +(-0.010,0);

% only one 0
\node[below left=2pt] at (0,0) {$0$};

% x-axis labels
\node[below=4pt] at (0.5,0) {$\frac{1}{2}$};
\node[below=4pt, anchor=north east] at (0.45,0) {$\frac{D-1}{2D}$};
\node[below=4pt, anchor=north west] at (0.29,0) {$\frac{d-1}{2d}$};
\node[below=4pt, anchor=north west] at (0.55,0) {$\frac{D+1}{2D}$};
\node[below=4pt, anchor=north east] at (0.71,0) {$\frac{d+1}{2d}$};

% y-axis labels
\node[left=7pt] at (0,\yredaxis) {$\frac{D-d}{2}$};
\node[left=7pt] at (0,\yint) {$\frac{D-d}{2d}$};

% axis labels
\node at (0.50,-0.60) {$\dfrac{1}{p}$};
\node at (-0.14,\ytop/2) {$\delta$};

% line labels
\node[blue] at (0.70,\ytop*0.79)
  {$\delta=\max\!\left(D\left|\frac{1}{p}-\frac{1}{2}\right|-\frac{1}{2},\,0\right)$};

\node[red] at (0.53,\ytop*0.42)
  {$\delta=(D-d)\left|\frac{1}{p}-\frac{1}{2}\right|$};

% definition box
\node[
    anchor=north east,
    fill=white,
    rounded corners=2pt,
    inner sep=3pt
] at (1.02,\ytop+0.26)
{$D=d_1 \vee d_2,\quad d=d_1\wedge d_2,\quad D-d=|d_1-d_2|$};

% labels (1) and (2)
\node[yellow, font=\bfseries\small] at (0.42,0.07) {\((\mathrm{A})\)};
\node[yellow, font=\bfseries\small] at (0.58,0.07) {\((\mathrm{A})\)};
\node[green, font=\bfseries\large] at (0.12,0.95) {\((\mathrm{B})\)};
\node[green, font=\bfseries\large] at (0.88,0.95) {\((\mathrm{B})\)};

\end{tikzpicture}


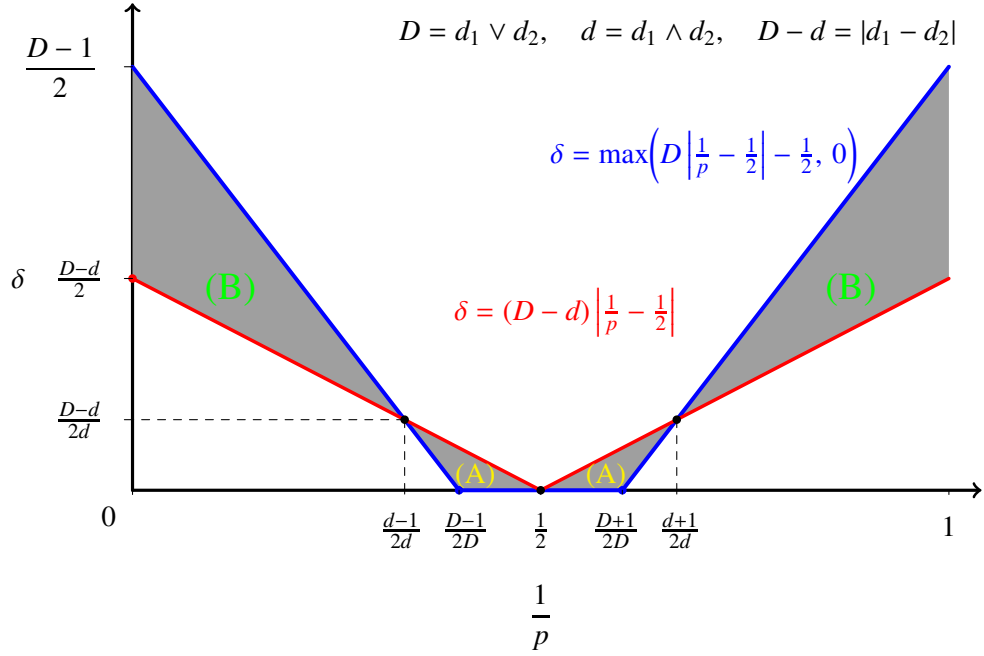
\captionof{figure}{The two regions in Proposition~6.1.}

\end{center}
\end{remark}

\begin{lemma}\label{Counter1}
Let $1\le p \le \infty$. Suppose 
\begin{equation}\label{counter1,1}
\delta > \max \left( D \left|\frac{1}{p}-\frac{1}{2}\right|-\frac{1}{2}, 0 \right),
\end{equation}
where $D=d_1 \vee d_2$.

Then the Bochner--Riesz mean $\sigma_{1}^\delta(\Lapd)$ is unbounded in $L^p(\widetilde{\R})$ if
$$
\delta < |d_1-d_2| \left| \frac{1}{p} - \frac{1}{2} \right|.
$$
\end{lemma}

\begin{proof}
By the classical results of Bochner--Riesz means for radial functions, it suffices to show the unboundedness of the \(kk\) part. Next, thanks to Corollary~\ref{cor1}, we are reduced to prove the unboundedness of the linear combination of $\sigma_{\mathrm{off}}^{\ell j, +}(1)$ and $\sigma_{\mathrm{off}}^{\ell j, -}(1)$. We decompose 
\begin{align*}
    \sigma_{\mathrm{off}}^{\ell j, \pm}(1)(x,y) = \mathcal{I}_{L,\pm}^1(x,y)+\mathcal{I}_{H,\pm}^1(x,y) 
\end{align*}
as in Proposition~\ref{prop:smallR}. Next, since $\mathcal{I}_{L,\pm}^1$ acts boundedly in $L^p$ for all $1<p<\infty$, it is enough to consider $\mathcal{I}_{H,\pm}^1(x,y)$. Recall that $\mathcal{I}_{H,\pm}^1 = \mathcal{I}_{\pm}^1 + \mathcal{II}_{\pm}^1$, where for $|x|\ge |y|$
\begin{align*}
&\mathcal{I}_{\pm}^1(x,y) = |x|^{1-\frac{d_\ell}{2}}|y|^{1-\frac{d_j}{2}}
\int_0^1
(1-\phi(\lambda |x|))\phi(\lambda |y|)(1-\lambda^2)^\delta
\lambda^{\frac{d_\ell+d_j}{2}-1}
K_{\nu_\ell}(\pm i\lambda |x|)K_{\nu_j}(\pm i\lambda |y|) d\lambda,\\
&\mathcal{II}_{\pm}^1(x,y)
= |x|^{1-\frac{d_\ell}{2}}|y|^{1-\frac{d_j}{2}}
\int_0^1
(1-\phi(\lambda |x|))(1-\phi(\lambda |y|))(1-\lambda^2)^\delta
\lambda^{\frac{d_\ell+d_j}{2}-1}
K_{\nu_\ell}(\pm i\lambda |x|)K_{\nu_j}(\pm i\lambda |y|) d\lambda.
\end{align*}
By Lemma~\ref{lem:I-smallR}, one shows that
$$
\mathcal{I}_{\pm}^1 \sim |x|^{\frac{1-d_\ell}{2}}|y|^{2-d_j}
\int_0^1
(1-\phi(\lambda |x|))\phi(\lambda |y|)(1-\lambda^2)^\delta
\lambda^{\frac{d_\ell-1}{2}}e^{\mp i\lambda |x|} d\lambda
$$
Observe that the integrand above is smooth and supported in $\left[ \frac{1}{2}|x|^{-1}, \frac{3}{4}|y|^{-1} \right] \subset [0,1]$. Integration by parts $N = \left\lceil \frac{d_\ell+1}{2} \right\rceil$ times yields that for $|x|\ge |y|$
\begin{align*}
    |\mathcal{I}_{\pm}^1| &\lesssim |x|^{-\frac{d_\ell-1}{2}-N}|y|^{2-d_j} \int_0^1 \left| \frac{d^N}{d\lambda^N} (1-\phi(\lambda |x|))\phi(\lambda |y|)(1-\lambda^2)^\delta
\lambda^{\frac{d_\ell-1}{2}}    \right| d\lambda \\
&\lesssim |x|^{-\frac{d_\ell-1}{2}-N}|y|^{2-d_j} \int_{2^{-1}|x|^{-1}}^1 (1-\lambda)^{\delta-N} \lambda^{\frac{d_\ell-1}{2}-N} d\lambda\\
& \lesssim |x|^{-\frac{d_\ell-1}{2}-N}|y|^{2-d_j} \int_{2^{-1}|x|^{-1}}^1 \lambda^{\frac{d_\ell-1}{2}-N} d\lambda\\
&\lesssim |x|^{-d_\ell} |y|^{2-d_j},
\end{align*}
where the second to last inequality follows by the uniform upper bound of $(1-\lambda)^{\delta-N}$, i.e., it is bounded by $1$ if $\delta \ge N$ and is controlled by $4^{N-\delta}$ if $\delta<N$. Thus by Lemma~\ref{Nix}, $\mathcal{I}_{\pm}^1$ acts boundedly in $L^p$ for all $1<p<\infty$.

We are therefore reduced to showing the linear combination of $\mathcal{II}_{\pm}^1$ is $L^p$-unbounded.

Following by Lemma~\ref{lem:II-smallR}, we decompose again via function $\rho(1-\lambda)$:
\begin{equation*}
    \mathcal{II}_{\pm}^1(x,y) \sim 
|x|^{\frac{1-d_\ell}{2}}|y|^{\frac{1-d_j}{2}} \left( II_{0,\pm} + II_{1,\pm} \right),
\end{equation*}
where
\begin{equation*}
    II_{0,\pm} = \int_0^1 \rho(1-\lambda)
(1-\phi(\lambda |x|))(1-\phi(\lambda |y|))(1-\lambda^2)^\delta
\lambda^{\frac{d_\ell+d_j}{2}-2}
e^{\mp i\lambda (|x|+|y|)} d\lambda,
\end{equation*}
and
\begin{equation*}
    II_{1,\pm} = \int_0^1 \left(1-\rho(1-\lambda) \right)
(1-\phi(\lambda |x|))(1-\phi(\lambda |y|))(1-\lambda^2)^\delta
\lambda^{\frac{d_\ell+d_j}{2}-2}
e^{\mp i\lambda (|x|+|y|)} d\lambda.
\end{equation*}

By the proof of Lemma~\ref{lem:II-smallR}, $|II_{1,\pm}|$ is bounded by $|x|^{-N} |y|^{N+1-\frac{d_\ell+d_j}{2}}$ for any $N> \frac{d_\ell+d_j}{2}-1$. Therefore, since $|x|\ge |y|$, it is bounded by $|x|^{-d_\ell} |y|^{2-d_j}$, acting as a bounded operator for all $1<p<\infty$.

It remains to treat $II_{0,\pm}$. By the change of variables $s=1-\lambda$, then Lemma~\ref{lem:osc-expansion-order-M} (see the proof of Lemma~\ref{lem:II-smallR} for detail), it becomes
\begin{equation}\label{II_00}
    II_{0,\pm} = e^{\mp i(|x|+|y|)} \sum_{m=0}^{N-1} c_{m,\delta} \varphi_{x,y,1}^{(m)}(0) \left(|x|+|y|\right)^{-1-\delta-m} + O\left(\left(|x|+|y|\right)^{-N}\right),
\end{equation}
where
\begin{equation*}
    \varphi_{x,y,1}(s)
=
\rho(s) (1-\phi((1-s)|x|))(1-\phi((1-s)|y|))(2-s)^\delta(1-s)^{\frac{d_\ell+d_j}{2}-2}.
\end{equation*}
Note that $\varphi_{x,y,1}(0) = \rho(0) (1-\phi(|x|)) (1-\phi(|y|)) 2^\delta = 2^\delta$ and $\left\| \varphi_{x,y,1}^{(l)}  \right\|_\infty \le c_{\delta,l}$ for all $l\ge 0$. We therefore can write
\begin{equation}
    II_{0,\pm} = c_{\delta} \frac{e^{\mp i(|x|+|y|)}}{(|x|+|y|)^{\delta+1}} + \sum_{m=1}^{N-1} O\left( (|x|+|y|)^{-1-\delta-m} \right)  + O\left(\left(|x|+|y|\right)^{-N}\right),
\end{equation}
and hence for $|x|\ge |y|$
\begin{equation}\label{eq_mainterm}
    |x|^{\frac{1-d_\ell}{2}}|y|^{\frac{1-d_j}{2}} II_{0,\pm} = \frac{e^{\mp i(|x|+|y|)} |x|^{\frac{1-d_\ell}{2}}|y|^{\frac{1-d_j}{2}}}{(|x|+|y|)^{\delta+1}} + \sum_{m=1}^{N-1} \frac{|x|^{\frac{1-d_\ell}{2}}|y|^{\frac{1-d_j}{2}}}{(|x|+|y|)^{\delta+1+m}} + \frac{|x|^{\frac{1-d_\ell}{2}}|y|^{\frac{1-d_j}{2}}}{(|x|+|y|)^N}.
\end{equation}
By Lemma~\ref{lem:II-smallR}, the last term above is again $O\left(|x|^{-d_\ell} |y|^{2-d_j}\right)$, so it acts boundedly in $L^p$ for all $1<p<\infty$. For the lower order terms, it is plain that for each $1\le m\le N-1$, it is bounded by the kernel
\begin{equation*}
    K(x,y):= \begin{cases}
        |x|^{-\frac{d_\ell-1}{2}} |y|^{-\delta-\frac{d_j+1}{2}-m}, & |x|\le |y|,\\
        |x|^{-\delta-\frac{d_\ell+1}{2}-m} |y|^{-\frac{d_j-1}{2}}, & |x|\ge |y|.
    \end{cases}
\end{equation*}
By Lemma~\ref{Nix}, the kernel $K$ acts as a bounded operator in $L^p$ for
\begin{equation}\label{BBR1}
    \frac{d_\ell}{d_\ell \wedge \left( \delta+\frac{d_\ell+1}{2}+m \right)} < p < \frac{d_j}{0 \vee \left( \frac{d_j-1}{2} - \delta-m \right)},
\end{equation}
and 
\begin{equation}\label{BBR2}
    p \left( \delta+m+\frac{d_\ell+1}{2}+\frac{d_j-1}{2} - d_j \right) \ge d_\ell-d_j \iff \delta \ge (d_\ell-d_j)\left( \frac{1}{p} - \frac{1}{2} \right) - m.
\end{equation}
By complex interpolation, it is enough to show the $L^p$-unboundedness of the first term of \eqref{eq_mainterm} for 
\begin{align}\label{Ubddp}
    \delta > \max \left( D\left|\frac{1}{p}-\frac{1}{2}\right|-\frac{1}{2}, 0 \right) \quad \textrm{and} \quad |d_1-d_2| \left|\frac{1}{p}-\frac{1}{2}\right| - 1 \le \delta < |d_1-d_2| \left|\frac{1}{p}-\frac{1}{2}\right|.
\end{align}
Define
\begin{align}
    \mathcal{S}_{\pm}^{\ell j}f(x) = \int_1^\infty \frac{x^{-\frac{d_\ell-1}{2}} y^{-\frac{d_j-1}{2}}}{(x+y)^{\delta+1}} e^{\mp i(x+y)} f(y) y^{d_j-1} dy.
\end{align}
Then we are reduced to proving the  $L^p$-unboundedness of $\sum_{\pm}\sum_{\ell \ne j}c_{\pm}^{\ell j}\mathcal{S}_{\pm}^{\ell j}$, where $c_{\pm}^{\ell j}\in \C$ for $p$ in \eqref{Ubddp}.

Without loss of generality, we suppose that $d_1<d_2$. Let $p>2$. We claim first that $\mathcal{S}_{\pm}^{21}$ is $L^p$-bounded for $p$ in \eqref{Ubddp} and $p>2$. Indeed, note that
\begin{align*}
    |\mathcal{S}_{\pm}^{21}(x,y)| \lesssim \begin{cases}
    |x|^{-\frac{d_2-1}{2}} |y|^{-\frac{d_1+1}{2}-\delta}, & |x|\le |y|,\\
        |x|^{-\frac{d_2+1}{2}-\delta} |y|^{-\frac{d_1-1}{2}}, & |x|\ge |y|,
    \end{cases}
\end{align*}
which acts boundedly in $L^p$ for $p$ in \eqref{counter1,1} and 
\begin{equation*}
    \delta \ge (d_2-d_1)\left(\frac{1}{p}-\frac{1}{2}\right).
\end{equation*}
Therefore, we are reduced to showing the linear combination of $\mathcal{S}_{\pm}^{12}$ is unbounded from 
$$L^p\left( [1,\infty), r^{d_2-1}dr \right) \to L^p\left( [1,\infty), r^{d_1-1}dr \right),
$$
where $d_1<d_2$, $p$ in \eqref{counter1,1} and $p>2$ and $\delta< (d_1-d_2)\left(\frac{1}{p}-\frac{1}{2}\right)$.

Without loss of generality, further assume $c_-^{12} \ne 0$. Consider function $f_\varepsilon(y) = e^{-iy} \mathbf{1}_{[\varepsilon,2\varepsilon]}(y) y^{-\frac{d_2-1}{2}}$ with $\varepsilon\ge 1$ (if $c_-^{12}=0$, thus $c_+^{12}\ne0$, then we put $f_\varepsilon(y) = e^{iy} \mathbf{1}_{[\varepsilon,2\varepsilon]}(y) y^{-\frac{d_2-1}{2}}$ instead). Then for $x\in [3\varepsilon, 4\varepsilon]$,
\begin{align}\label{counter1}
    |\mathcal{S}_-^{12}f_\varepsilon(x)| = \left| \int_\varepsilon^{2\varepsilon} \frac{x^{-\frac{d_1-1}{2}} y^{-\frac{d_2-1}{2}}}{(x+y)^{\delta+1}} y^{\frac{d_2-1}{2}} dy\right| \gtrsim  \varepsilon^{-\frac{d_1-1}{2}-\delta},
\end{align}
and 
\begin{align}\label{counter2}
    |\mathcal{S}_+^{12}f_\varepsilon(x)| &\lesssim \varepsilon^{-\frac{d_1-1}{2}} \left|  \int_\varepsilon^{2\varepsilon} \frac{e^{-2iy}}{(x+y)^{\delta+1}} dy  \right|\\ \nonumber
    &\lesssim \varepsilon^{-\frac{d_1-1}{2}} \left| \left[ \frac{e^{-2iy}}{(x+y)^{\delta+1}} \right]_{y=\varepsilon}^{2\varepsilon} + (\delta+1)  \int_\varepsilon^{2\varepsilon} \frac{e^{-2iy}}{(x+y)^{\delta+2}} dy  \right|\\ \nonumber
    &\lesssim \varepsilon^{-\frac{d_1+1}{2}-\delta}.
\end{align}
Moreover, a direct computation yields
\begin{align}\label{counter3}
    \|f_\varepsilon\|_p \lesssim \varepsilon^{\frac{d_2}{p} - \frac{d_2-1}{2}}.
\end{align}
Combining \eqref{counter1}, \eqref{counter2} and \eqref{counter3}, we conclude for $\varepsilon$ large,
\begin{align*}
    \left\| \sum_{\pm} c_\pm^{12} \mathcal{S}_{\pm}^{12}f_\varepsilon \right\|_{L^p([1,\infty), x^{d_1-1}dx)} \ge \left\| \sum_{\pm} c_\pm^{12} \mathcal{S}_{\pm}^{12}f_\varepsilon \right\|_{L^p([3\varepsilon, 4\varepsilon],x^{d_1-1}dx)} \gtrsim \varepsilon^{-\frac{d_1-1}{2}-\delta+\frac{d_1}{p}}.
\end{align*}
So the $L^p$ boundedness of $\sum_{\pm} c_\pm^{12} \mathcal{S}_{\pm}^{12} $ would imply
\begin{equation*}
    \varepsilon^{-\frac{d_1-1}{2}-\delta+\frac{d_1}{p}} \lesssim \varepsilon^{\frac{d_2}{p} - \frac{d_2-1}{2}} \iff \varepsilon^{(d_1-d_2)\left(\frac{1}{p}-\frac{1}{2}\right)-\delta}  \lesssim 1,
\end{equation*}
which cannot hold as $\varepsilon \to \infty$.

The proof of Lemma~\ref{Counter1} is now complete.

\end{proof}

\begin{remark}
Lemma~\ref{Counter1} is based on exactly the same counterexample as Theorem~\ref{fefferman_nondoubling}. More precisely, we use the same test function $f_\varepsilon$ supported on the larger end and evaluate the output on the smaller end. The only difference is that, in the Bochner--Riesz case, the kernel carries the additional factor $(1-\lambda^2)^\delta$, which modifies the size of the leading term by replacing $(|x|+|y|)^{-1}$ with $(|x|+|y|)^{-1-\delta}$. Thus the geometry of the obstruction is unchanged: the contradiction still comes from the mismatch between the two end dimensions, while the order $\delta$ only changes the decay rate of the main term.
\end{remark}

\begin{lemma}\label{Counter2}
Let $1\le p\le \infty$. Suppose that $\delta>0$ and
\begin{align}\label{BR2}
\delta \ge |d_1-d_2| \left| \frac{1}{p} - \frac{1}{2} \right|.
\end{align}
Then the Bochner--Riesz mean $\sigma_{1}^\delta(\Lapd)$ is unbounded in $L^p(\widetilde{\R})$ if
\begin{equation}\label{counter2.1}
    0< \delta \le D \left|\frac{1}{p}-\frac{1}{2}\right|-\frac{1}{2},
\end{equation}
where $D=d_1 \vee d_2$.
\end{lemma}

\begin{proof}
The proof is similar to Lemma~\ref{Counter1}. Without loss of generality, we may assume $2<d_1<d_2$. It suffices to consider operator $   M_{\chi_{(-\infty,-1]}}\sigma_{1}^\delta(\Lapd) M_{\chi_{[1,\infty)}} = \sigma_{\mathrm{off}}^{12,+}(1) - \sigma_{\mathrm{off}}^{12,-}(1)$ (up to constant). Next, by the argument from Lemma~\ref{Counter1}, we only need to treat $\mathcal{II}_{\pm}^1$, since the parts $\mathcal{I}_{L,\pm}^1$ and $\mathcal{I}_{\pm}^1$ are both bounded in $L^p$ for all $1<p<\infty$. Hence, by Lemma~\ref{lem:II-smallR}, we may focus on 
\begin{align*}
    \mathcal{II}_{\pm}^1(x,y)
&= |x|^{1-\frac{d_1}{2}}|y|^{1-\frac{d_2}{2}}
\int_0^1
(1-\phi(\lambda |x|))(1-\phi(\lambda |y|))(1-\lambda^2)^\delta
\lambda^{\frac{d_1+d_2}{2}-1}
K_{\nu_1}(\pm i\lambda |x|)K_{\nu_2}(\pm i\lambda |y|) d\lambda\\
&\sim\frac{|x|^{-\frac{d_1-1}{2}} |y|^{-\frac{d_2-1}{2}}}{(|x|+|y|)^{\delta+1}} e^{\mp i(|x|+|y|)} + \mathcal{R}(x,y),
\end{align*}
where $\mathcal{R}$ acts boundedly in $L^p$ for $p$ satisfying \eqref{BR2} and \eqref{BBR1} with $m=1$. Therefore, it follows by complex interpolation that 
% It follows by \eqref{remainderest} that the remainder term is bounded by (for $|x|\ge |y|$)
% \begin{align*}
%     |x|^{-\frac{d_2-1}{2}}|y|^{-\frac{d_1-1}{2}} e^{\mp i(|x|+|y|)} \mathcal{R}(x,y)\lesssim \begin{cases}
%         |x|^{-\frac{d_2+1}{2}-\delta-1} |y|^{-\frac{d_1-1}{2}}, & d_1+d_2\ge 6,\\
%         |x|^{-\frac{d_2+1}{2}-\delta-1} |y|^{-\frac{d_1-1}{2} - \frac{d_1+d_2-6}{2}}, & d_1+d_2< 6,
%     \end{cases}
% \end{align*}
% which acts as a bounded operator in $L^p$ for
% \begin{align}\label{BR2}
%     \frac{d_2}{d_2 \wedge \left( \frac{d_2+3}{2} + \delta \right)} < p < \frac{d_1}{0 \vee \left( \frac{d_1-3}{2}-\delta \right)}
% \end{align}
% and (see \eqref{remainder1} and \eqref{remainder2} above)
% \begin{align*}
% \delta \ge |d_1-d_2| \left| \frac{1}{p} - \frac{1}{2} \right|.
% \end{align*}
we are reduced to showing the main term of $\mathcal{II}_{\pm}^1$ is unbounded in $L^p$ for $p$ satisfying \eqref{BR2} and
\begin{equation*}
    D\left| \frac{1}{p} - \frac{1}{2} \right| - \frac{3}{2} < \delta \le D\left| \frac{1}{p} - \frac{1}{2} \right| - \frac{1}{2}
\end{equation*}
In other words, we need to verify the unboundedness of the non-zero linear combination of operators:
\begin{equation*}
    \mathcal{S}_{\pm}f(x) := \int_1^\infty \frac{x^{-\frac{d_1-1}{2}} y^{-\frac{d_2-1}{2}}}{(x+y)^{\delta+1}} e^{\mp i(x+y)} f(y) y^{d_2-1} dy
\end{equation*}
from $L^p\left([1,\infty), r^{d_2-1}dr\right)$ to $L^p\left([1,\infty), r^{d_1-1}dr\right)$, where $d_1<d_2$.

Since $\Lapd$ is self-adjoint, we may assume $p>2$. Consider function $f_\varepsilon(y) = e^{-iy} \mathbf{1}_{(2,\varepsilon)}(y) y^{-\frac{d_2}{p}}$ with $\varepsilon > 100$ large. It is clear that for $x\in [1,2]$,
\begin{align*}
    |\mathcal{S}_{-}f_\varepsilon(x)| &= x^{-\frac{d_1-1}{2}} \left|\int_2^{\varepsilon} \frac{y^{\frac{d_2-1}{2}-\frac{d_2}{p}}}{(x+y)^{\delta+1}}  dy\right|\\
    &\gtrsim  \int_2^{\varepsilon} y^{\frac{d_2-1}{2}-\frac{d_2}{p}-\delta-1}  dy\\
    &\gtrsim \begin{cases}
        \varepsilon^{\frac{d_2-1}{2}-\frac{d_2}{p}-\delta}, & \frac{d_2-1}{2}-\frac{d_2}{p}-\delta >0,\\
        \log{\varepsilon}, & \frac{d_2-1}{2}-\frac{d_2}{p}-\delta = 0,
    \end{cases}
\end{align*}
since the assumption gives
\begin{equation*}
    \delta \le D \left|\frac{1}{p}-\frac{1}{2}\right|-\frac{1}{2} = \frac{d_2-1}{2}-\frac{d_2}{p}.
\end{equation*}
On the other hand,
\begin{align*}
    |\mathcal{S}_+ f_\varepsilon(x)|
    &=
    x^{-\frac{d_1-1}{2}}
    \left|
    \int_2^\varepsilon
    \frac{y^{\frac{d_2-1}{2}-\frac{d_2}{p}}}{(x+y)^{\delta+1}}e^{2iy}dy
    \right|\\
    &=
    x^{-\frac{d_1-1}{2}}
    \left|
    \frac{y^{\frac{d_2-1}{2}-\frac{d_2}{p}}e^{2iy}}{(x+y)^{\delta+1}}\bigg|_{y=2}^{\varepsilon}
    -
    \int_2^\varepsilon
    e^{2iy}\frac{d}{dy}
    \left(
    \frac{y^{\frac{d_2-1}{2}-\frac{d_2}{p}}}{(x+y)^{\delta+1}}
    \right)dy
    \right|\\
    &\lesssim
    \varepsilon^{\frac{d_2-1}{2}-\frac{d_2}{p}-\delta-1}+1.
\end{align*}
Assume, for instance, that the coefficient of $\mathcal S_-$ in the principal linear combination is non-zero. If instead the non-zero coefficient is attached to $\mathcal S_+$, we replace $e^{-iy}$ by $e^{iy}$ in the definition of $f_\varepsilon$. Hence, for $\varepsilon$ sufficiently large,
\begin{equation*}
    \left\|\sum_{\pm}c_\pm \mathcal S_\pm f_\varepsilon\right\|_{L^p([1,\infty),x^{d_1-1}dx)}
    \gtrsim
    \begin{cases}
        \varepsilon^{\frac{d_2-1}{2}-\frac{d_2}{p}-\delta}, & \frac{d_2-1}{2}-\frac{d_2}{p}-\delta > 0,\\
        \log\varepsilon, & \frac{d_2-1}{2}-\frac{d_2}{p}-\delta = 0.
    \end{cases}
\end{equation*}
If $\sum_{\pm}c_\pm \mathcal S_\pm$ were bounded from $L^p([1,\infty),r^{d_2-1}dr)$ to $L^p([1,\infty),r^{d_1-1}dr)$, then we would have
\begin{equation*}
    \begin{cases}
        \varepsilon^{\frac{d_2-1}{2}-\frac{d_2}{p}-\delta}, & \frac{d_2-1}{2}-\frac{d_2}{p}-\delta>0,\\
        \log\varepsilon, & \frac{d_2-1}{2}-\frac{d_2}{p}-\delta=0
    \end{cases}
    \lesssim
    \|f_\varepsilon\|_{L^p([1,\infty),y^{d_2-1}dy)}
    \lesssim
    (\log\varepsilon)^{\frac{1}{p}},
\end{equation*}
which is impossible as $\varepsilon\to\infty$.

This completes the proof of Lemma~\ref{Counter2}.

\end{proof}

Proposition~\ref{propBRR} now follows from
Lemmas~\ref{Counter1} and~\ref{Counter2}. Together with the sufficiency
result in Theorem~\ref{thm:BRR}, this completes the proof of
Theorem~\ref{main_BR}.

\section{A H\"ormander-type spectral multiplier theorem}

We conclude with a H\"ormander-type spectral multiplier consequence of our
Bochner--Riesz estimates. Although the resulting condition is unlikely to be
sharp, estimates of the form \eqref{hor} are expected to hold more generally
on manifolds with ends; see \cite{CSY}. The argument is closely related in
spirit to the radial multiplier theorem of Carbery, Gasper and Trebels
\cite{CarberyGasperTrebels1984}.

For $s\geq 0$, define the standard Sobolev norm on $\mathbb R$  by
\[
  \|F\|_{W^{s,p}}
  =
  \left\|(I-d^2/dx^2)^{s/2}F\right\|_{p}.
\]

\begin{coro}\label{cor:Hormander-multiplier}
Let \(1<p<\infty\), and let \(F\) be a bounded Borel function satisfying
\[
    \supp F\subset[0,1]
    \qquad\text{and}\qquad
    F\in W^{s,1}(\mathbb R),
\]
where
\[
    s>
    \max\left\{
        D\left|\frac1p-\frac12\right|+\frac12,\,
        |d_1-d_2|
        \left|\frac1p-\frac12\right|+1
    \right\}.
\]
Then \(F(t\Lapd)\) is bounded on \(L^p\), uniformly in \(t>0\), and
\begin{equation}
    \label{hor}
    \sup_{t>0}
    \|F(t\Lapd)\|_{L^p\to L^p}
    \leq
    C\,
    \|F\|_{W^{s,1}(\mathbb R)}.
\end{equation}
\end{coro}

\begin{proof}
For \(a\in\mathbb C\), let
\[
    \chi_-^a(x)
    :=
    \frac{x_-^a}{\Gamma(a+1)},
    \qquad
    x_-^a
    :=
    \begin{cases}
        |x|^a, & x\leq 0,\\
        0,     & x>0.
    \end{cases}
\]
Initially this defines a distribution for \(\Re a>-1\). The identity
\[
    \frac{d}{dx}\chi_-^a=-\chi_-^{a-1}
\]
extends \(a\mapsto\chi_-^a\) to an entire distribution-valued family.
Moreover,
\[
    \chi_-^w*\chi_-^z=\chi_-^{w+z+1},
    \qquad w,z\in\mathbb C;
\]
see \cite[(3.4.10)]{Hormander}. Define $\delta_-^\nu:=\chi_-^{-\nu-1}$
and, for \(F\) supported in \([0,\infty)\), define its Weyl fractional
derivative of order \(\nu\) by
\[
    F^{(\nu)}:=F*\delta_-^\nu.
\]
The convolution identity gives
\[
    F^{(\nu)}*\delta_-^{-\nu}=F.
\]
Consequently, for every \(\delta>-1\),
\[
\begin{aligned}
    F(\Lapd)
    &=
    \frac{1}{\Gamma(\delta+1)}
    \int_0^\infty
        F^{(\delta+1)}(r)
        (r-\Lapd)_+^\delta\,dr \\
    &=
    \frac{1}{\Gamma(\delta+1)}
    \int_0^\infty
        F^{(\delta+1)}(r)\,
        r^\delta
        \sigma_r^\delta(\Lapd)\,dr .
\end{aligned}
\]

Choose \(\delta\) so that
\[
    \max\left\{
        D\left|\frac1p-\frac12\right|-\frac12,\,
        |d_1-d_2|
        \left|\frac1p-\frac12\right|
    \right\}
    <\delta<s-1.
\]
The Bochner--Riesz theorem then yields
\[
    \sup_{r>0}
    \|\sigma_r^\delta(\Lapd)\|_{L^p\to L^p}
    <\infty.
\]
It follows that
\[
    \|F(\Lapd)\|_{L^p\to L^p}
    \lesssim
    \int_0^\infty
        |F^{(\delta+1)}(r)|\,r^\delta\,dr.
\]
Since \(\supp F\subset[0,1]\) and \(\delta+1<s\), the standard
one-dimensional Sobolev estimate gives
\[
    \int_0^\infty
        |F^{(\delta+1)}(r)|\,r^\delta\,dr
    \lesssim
    \|F\|_{W^{s,1}(\mathbb R)}.
\]
For dilation \(F_t(\lambda)=F(t\lambda)\) we get
$$
\int_0^\infty
\left|F_t^{(\delta+1)}(r)\right|r^\delta\,dr
=
\int_0^\infty
\left|F^{(\delta+1)}(u)\right|u^\delta\,du,
$$
which implies  the estimate uniformly for \(t>0\).
\end{proof}

\section{Open problems}\label{oppr}
We conclude by recording several directions suggested by the present model for future research.
\medskip 

{Q1. Analysis of the low-dimensional regime.}
In this paper, we restrict attention to the case \(d_i>2\). When \(d_i\leq 2\), the small-\(\lambda\) asymptotics of the special functions \(k_d(i\lambda)\) and their derivatives differ substantially from those in the range \(d_i>2\). Consequently, the low-energy analysis developed here must be modified. Treating all possible pairs of end dimensions would require several distinct versions of the argument. We therefore defer the cases \(1\leq d_i\leq 2\) to a subsequent paper.

{Q2. From the one-dimensional model to manifolds with ends.}
The operator \(\Delta_{d_1,d_2}\) provides a robust one-dimensional model for analysis on manifolds with ends. Our results suggest that, on genuine manifolds with ends of different volume-growth dimensions, Bochner--Riesz summability should be governed not only by the largest end dimension, but also by an additional obstruction arising from the interaction between unequal ends.

More precisely, it is natural to ask whether analogues of the necessary condition \eqref{asym} hold for connected sums or more general manifolds with ends in the sense of \cite{GS}. Our counterexamples suggest that such negative results should persist in the geometric setting, since they depend on the separation between ends and the mismatch in volume growth, rather than on specifically one-dimensional features.

The more difficult question is sufficiency. One would like to identify geometric and analytic hypotheses under which the conditions in Theorem~\ref{main_BR} remain sufficient for the \(L^p\)-boundedness of Bochner--Riesz means on manifolds with ends. This problem remains open. A radially symmetric model of \(\mathbb{R}^2\#\mathbb{R}^2\) may provide a more tractable first case, following the ideas developed in \cite{ST}.

{Q3.  Inverse-square potentials.}
Herz also considers variants involving an inverse-square potential, which can be reduced to operators of the form
\[
\Delta^{(0,\infty)}_{\mathrm{Neu},d} + c\,r^{-2}.
\]
It would be interesting to develop an analogue of our Bochner--Riesz theory for
$\Delta^{(0,\infty)}_{\mathrm{Neu},d} + c\,r^{-2}$ and for the corresponding two-end operators with such a potential on one or both ends. We plan to investigate the inverse-square potential case in a separate paper.

{Q4. H\"ormander-type spectral multipliers.}
We ask for a sharp version of Corollary~\ref{cor:Hormander-multiplier}. One expects a result formulated in terms of the Sobolev norm
\(\|F\|_{W^{s,2}}\), with the required differentiability order reduced by
\(1/2\); see \cite{CarberyGasperTrebels1984}. For an illuminating discussion
of the analogous problem for the harmonic oscillator, see \cite{Th89,PCWHAS}.

\bigskip

\noindent
{\bf Acknowledgements:}  PC was supported  by National Key R$\&$D Program of China 2022YFA1005700 and
by  NNSF of China No. 12171489.
AS and DH were 
 supported by Australian Research Council  Discovery Grant  DP200101065. 
LY was  supported  by National Key R$\&$D Program of China 2022YFA1005700  and
by  NNSF of China No. 12571111.

%\bibliographystyle{abbrv}
%\bibliography{name}

\begin{thebibliography}{10}
	
	
	\bibitem{AS}
	M.~Abramowitz and I.~A. Stegun.
	\newblock {\em Handbook of mathematical functions with formulas, graphs, and
		mathematical tables}, volume~55.
	\newblock US Government printing office, 1948.
	
	\bibitem{Bochner1936}
	S.~Bochner.
	\newblock Summation of multiple {F}ourier series by spherical means.
	\newblock {\em Trans. Am. Math. Soc.}, 40(2):175--207, 1936.
	
	\bibitem{Bourgain}
	J.~Bourgain.
	\newblock Besicovitch type maximal operators and applications to fourier
	analysis.
	\newblock {\em Geom. Funct. Anal.}, 1:147--187, 1991.
	
	\bibitem{CarberyGasperTrebels1984}
	A.~Carbery, G.~Gasper, and W.~Trebels.
	\newblock Radial {F}ourier multipliers of {$L\sp{p}({\bf R}\sp{2})$}.
	\newblock {\em Proc. Nat. Acad. Sci. U.S.A.}, 81(10):3254--3255, 1984.
	
	\bibitem{CaS}
	L.~Carleson and P.~Sj\"olin.
	\newblock Oscillatory integrals and a multiplier problem for the disc.
	\newblock {\em Studia Math.}, 44:287--299. (errata insert), 1972.
	
	\bibitem{PCWHAS}
	P.~Chen, W.~Hebisch, and A.~Sikora.
	\newblock Bochner-{Riesz} profile of anharmonic oscillator {{\(\mathcal{L} =
			-\frac{d^2}{dx^2} + |x|\)}}.
	\newblock {\em J. Funct. Anal.}, 271(11):3186--3241, 2016.
	
	\bibitem{COSY}
	P.~Chen, E.~M. Ouhabaz, A.~Sikora, and L.~Yan.
	\newblock Spectral multipliers without semigroup framework and application to
	random walks.
	\newblock {\em J. Math. Pures Appl. (9)}, 143:162--191, 2020.
	
	\bibitem{CSY}
	P.~Chen, A.~Sikora, and L.~Yan.
	\newblock Spectral multipliers via resolvent type estimates on non-homogeneous
	metric measure spaces.
	\newblock {\em Math. Z.}, 294(1-2):555--570, 2020.
	
	\bibitem{MiCh}
	M.~Christ.
	\newblock On almost everywhere convergence of {B}ochner-{R}iesz means in higher
	dimensions.
	\newblock {\em Proc. Amer. Math. Soc.}, 95(1):16--20, 1985.
	
	\bibitem{cordoba}
	A.~C\'ordoba.
	\newblock The {K}akeya maximal function and the spherical summation
	multipliers.
	\newblock {\em Amer. J. Math.}, 99(1):1--22, 1977.
	
	\bibitem{cord}
	A.~C\'ordoba.
	\newblock The disc multiplier.
	\newblock {\em Duke Math. J.}, 58(1):21--29, 1989.
	
	\bibitem{Er1956}
	A.~Erd\'elyi.
	\newblock {\em Asymptotic expansions}.
	\newblock Dover Publications, Inc., New York, 1956.
	
	\bibitem{Fefferman}
	C.~Fefferman.
	\newblock The multiplier problem for the ball.
	\newblock {\em Ann. of Math. (2)}, 94:330--336, 1971.
	
	\bibitem{Grad-Ry}
	I.~S. Gradshteyn and I.~M. Ryzhik.
	\newblock {\em Table of integrals, series, and products}.
	\newblock Elsevier/Academic Press, Amsterdam, seventh edition, 2007.
	\newblock Translated from the Russian, Translation edited and with a preface by
	Alan Jeffrey and Daniel Zwillinger, With one CD-ROM (Windows, Macintosh and
	UNIX).
	
	\bibitem{GIS}
	A.~Grigor'yan, S.~Ishiwata, and L.~Saloff-Coste.
	\newblock Geometric analysis on manifolds with ends.
	\newblock In {\em Analysis and partial differential equations on manifolds,
		fractals and graphs. Contributions of the conference, Tianjin, China,
		September 2019}, pages 325--343. Berlin: De Gruyter, 2021.
	
	\bibitem{GS}
	A.~Grigor'yan and L.~Saloff-Coste.
	\newblock Heat kernel on manifolds with ends.
	\newblock {\em Ann. Inst. Fourier}, 59(5):1917--1997, 2009.
	
	\bibitem{HLP}
	G.~H. Hardy, J.~E. Littlewood, and G.~P\'{o}lya.
	\newblock {\em Inequalities}.
	\newblock Cambridge Mathematical Library. Cambridge University Press,
	Cambridge, 1988.
	\newblock Reprint of the 1952 edition.
	
	\bibitem{HNS}
	A.~Hassell, D.~Nix, and A.~Sikora.
	\newblock Riesz transforms on a class of non-doubling manifolds {II}.
	\newblock {\em Indiana Univ. Math. J.}, 73(3):955--1003, 2024.
	
	\bibitem{HS1D}
	A.~Hassell and A.~Sikora.
	\newblock Riesz transforms in one dimension.
	\newblock {\em Indiana Univ. Math. J.}, 58(2):823--852, 2009.
	
	\bibitem{HS}
	A.~Hassell and A.~Sikora.
	\newblock Riesz transforms on a class of non-doubling manifolds.
	\newblock {\em Comm. Partial Differential Equations}, 44(11):1072--1099, 2019.
	
	\bibitem{HE_phd}
	D.~He.
	\newblock Analysis of the laplacian on a class of nondoubling connected sums
	and manifolds with quadratically decaying ricci curvature.
	\newblock {\em Bull. Aust. Math. Soc.}, 113(1):173–174, 2026.
	
	\bibitem{Herz}
	C.~S. Herz.
	\newblock On the mean inversion of {F}ourier and {H}ankel transforms.
	\newblock {\em Proc. Nat. Acad. Sci. U.S.A.}, 40:996--999, 1954.
	
	\bibitem{Hormander}
	L.~H\"ormander.
	\newblock {\em The analysis of linear partial differential operators. {I}}.
	\newblock Springer Study Edition. Springer-Verlag, Berlin, second edition,
	1990.
	\newblock Distribution theory and Fourier analysis.
	
	\bibitem{IoLi2014}
	A.~Iosevich and E.~Liflyand.
	\newblock {\em Decay of the {F}ourier transform}.
	\newblock Birkh\"auser/Springer, Basel, 2014.
	\newblock Analytic and geometric aspects.
	
	\bibitem{LeeS}
	S.~Lee and J.~Ryu.
	\newblock Bochner-{R}iesz means for the {H}ermite and special {H}ermite
	expansions.
	\newblock {\em Adv. Math.}, 400:Paper No. 108260, 52, 2022.
	
	\bibitem{MinicozziSogge}
	W.~P. Minicozzi and C.~D. Sogge.
	\newblock Negative results for nikodym maximal functions and related
	oscillatory integrals in curved space.
	\newblock {\em Math. Res. Lett.}, 4:221--237, 1997.
	
	\bibitem{Muck1}
	B.~Muckenhoupt.
	\newblock Mean convergence of {H}ermite and {L}aguerre series. {I}.
	\newblock {\em Trans. Amer. Math. Soc.}, 147:419--431;, 1970.
	
	\bibitem{Muck2}
	B.~Muckenhoupt.
	\newblock Mean convergence of {H}ermite and {L}aguerre series. {II}.
	\newblock {\em Trans. Amer. Math. Soc.}, 147:433--460;, 1970.
	
	\bibitem{Muckenhoupt}
	B.~Muckenhoupt.
	\newblock Weighted norm inequalities for the {H}ardy maximal function.
	\newblock {\em Trans. Amer. Math. Soc.}, 165:207--226, 1972.
	
	\bibitem{Nix}
	D.~Nix.
	\newblock The resolvent and riesz transform on connected sums of manifolds with
	different asymptotic dimensions.
	\newblock {\em Bull. Aust. Math. Soc.}, 104(2):344--345, 2021.
	
	\bibitem{SharmaSikoraThangavelu}
	H.~Sharma, A.~Sikora, and S.~Thangavelu.
	\newblock Bochner--riesz means for {L}aguerre expansions.
	\newblock Work in preparation, 2026.
	
	\bibitem{ST}
	A.~Sikora and T.~Tao.
	\newblock Bochner-{R}iesz summability for analytic functions on the
	{$m$}-complex unit sphere and for cylindrically symmetric functions on {$\Bbb
		R^{n-1}\times\Bbb R$}.
	\newblock {\em Comm. Anal. Geom.}, 12(1-2):43--57, 2004.
	
	\bibitem{chSo}
	C.~D. Sogge.
	\newblock Oscillatory integrals and spherical harmonics.
	\newblock {\em Duke Math. J.}, 53(1):43--65, 1986.
	
	\bibitem{Sogge}
	C.~D. Sogge.
	\newblock {\em Fourier integrals in classical analysis}, volume 105 of {\em
		Cambridge Tracts in Mathematics}.
	\newblock Cambridge University Press, Cambridge, 1993.
	
	\bibitem{Stein2}
	Stein.
	\newblock {\em Some problems in harmonic analysis, Proc. Symposia in Pure
		Math.}, volume~1.
	\newblock AMS publications, 1979.
	
	\bibitem{Stein}
	E.~M. Stein.
	\newblock {\em Harmonic analysis: real-variable methods, orthogonality, and
		oscillatory integrals}, volume~43 of {\em Princeton Mathematical Series}.
	\newblock Princeton University Press, Princeton, NJ, 1993.
	\newblock With the assistance of Timothy S. Murphy, Monographs in Harmonic
	Analysis, III.
	
	\bibitem{TT}
	T.~Tao.
	\newblock The {B}ochner-{R}iesz conjecture implies the restriction conjecture.
	\newblock {\em Duke Math. J.}, 96(2):363--375, 1999.
	
	\bibitem{Th89}
	S.~Thangavelu.
	\newblock Summability of {H}ermite expansions. {I}, {II}.
	\newblock {\em Trans. Amer. Math. Soc.}, 314(1):119--142, 143--170, 1989.
	
	\bibitem{wang2025}
	H.~Wang and J.~Zahl.
	\newblock Volume estimates for unions of convex sets, and the kakeya set
	conjecture in three dimensions, 2025.
		
	
\end{thebibliography}

\end{document}